\documentclass[11pt]{article}
\usepackage{hyperref} 

\usepackage{amsmath}
\usepackage{amssymb}
\usepackage{amscd}
\usepackage{textcomp}

\usepackage{amsthm}
\theoremstyle{plain}
\newtheorem{theorem}{Theorem}[section]

\newtheorem{lemma}[theorem]{Lemma}

\newtheorem{corollary}[theorem]{Corollary}
\newtheorem{proposition}[theorem]{Proposition}

\theoremstyle{definition}
\newtheorem{definition}[theorem]{Definition}

\newtheorem{example}[theorem]{Example}
\newtheorem{remark}[theorem]{Remark}

\theoremstyle{remark}

\usepackage{xy}
\xyoption{all}

\newdir^{ (}{{}*!/-3pt/\dir^{(}}    
\newdir^{  }{{}*!/-3pt/\dir^{}}    
\newdir_{ (}{{}*!/-3pt/\dir_{(}}    
\newdir_{  }{{}*!/-3pt/\dir_{}}

\newcounter{zahl}

\def\theenumi{(\alph{enumi})}

\def\p@enumii{\theenumi}

\newcommand{\DS}{\displaystyle}
\newcommand{\TS}{\textstyle}
\newcommand{\SC}{\scriptstyle}
\newcommand{\SSC}{\scriptscriptstyle}

\DeclareMathOperator{\Frob}{Frob}

\DeclareMathOperator{\GL}{GL}

\DeclareMathOperator{\CKoh}{\check H}
\DeclareMathOperator{\Hom}{Hom}

\DeclareMathOperator{\Ind}{Ind}
\DeclareMathOperator{\Int}{Int}

\DeclareMathOperator{\Lie}{Lie}

\DeclareMathOperator{\PGL}{PGL}

\DeclareMathOperator{\Quot}{Frac}

\DeclareMathOperator{\Rep}{Rep}
\DeclareMathOperator{\Res}{Res}

\DeclareMathOperator{\SL}{SL}

\DeclareMathOperator{\Spec}{Spec}
\DeclareMathOperator{\Spf}{Spf}

\DeclareMathOperator{\Tor}{Tor}

\DeclareMathOperator{\Var}{V}
\newcommand{\Zar}{{\it Zar\/}}

\newcommand{\ad}{{\rm ad}}
\newcommand{\alg}{{\rm alg}}

\DeclareMathOperator{\codim}{codim}
\DeclareMathOperator{\coker}{coker}

\newcommand{\cont}{{\rm cont}}

\newcommand{\dom}{{\rm dom}}

\newcommand{\et}{{\acute{e}t\/}}
\newcommand{\fppf}{{\it fppf\/}}
\newcommand{\fpqc}{{\it fpqc\/}}

\DeclareMathOperator{\id}{\,id}

\renewcommand{\mod}{{\rm\,mod\,}}
\DeclareMathOperator{\ord}{ord}

\newcommand{\PHS}{\mbox{\it P\hspace{-0.1em}HS}}

\newcommand{\red}{{\rm red}}

\DeclareMathOperator{\rk}{rk}

\newcommand{\topol}{{*}}

\newcommand{\whtimes}{\wh{\raisebox{0ex}[0ex]{$\times$}}}

\renewcommand{\phi}{\varphi}
\renewcommand{\epsilon}{\varepsilon}
\usepackage{amsfonts}
\newcommand{\BOne} {{\mathchoice{\hbox{\rm1\kern-2.7pt l\kern.9pt}}
                              {\hbox{\rm1\kern-2.7pt l\kern.9pt}}
                              {\hbox{\scriptsize\rm1\kern-2.3pt l\kern.4pt}}
                              {\hbox{\scriptsize\rm1\kern-2.4pt l\kern.5pt}}}}

\newcommand{\BA}{{\mathbb{A}}}

\newcommand{\BC}{{\mathbb{C}}}
\newcommand{\BD}{{\mathbb{D}}}

\newcommand{\BF}{{\mathbb{F}}}
\newcommand{\BG}{{\mathbb{G}}}

\newcommand{\BN}{{\mathbb{N}}}

\newcommand{\BQ}{{\mathbb{Q}}}

\newcommand{\BZ}{{\mathbb{Z}}}

\newcommand{\CB}{{\cal{B}}}
\newcommand{\CC}{{\cal{C}}}

\newcommand{\CF}{{\cal{F}}}
\newcommand{\CG}{{\cal{G}}}
\newcommand{\CH}{{\cal{H}}}
\newcommand{\CI}{{\cal{I}}}

\newcommand{\CL}{{\cal{L}}}
\newcommand{\CM}{{\cal{M}}}
\newcommand{\CN}{{\cal{N}}}
\newcommand{\CO}{{\cal{O}}}
\newcommand{\CP}{{\cal{P}}}

\newcommand{\CS}{{\cal{S}}}

\newcommand{\Fa}{{\mathfrak{a}}}

\newcommand{\Fm}{{\mathfrak{m}}}

\newcommand{\Fp}{{\mathfrak{p}}}

\let\setminus\smallsetminus

\newcommand{\es}{\enspace}

\newcommand{\dual}{^{\SSC\lor}}

\newcommand{\kompl}{^{\SSC\land}}
\newcommand{\mal}{^{\SSC\times}}

\newcommand{\ul}[1]{{\underline{#1}}}
\newcommand{\ol}[1]{{\overline{#1}}}
\newcommand{\wh}[1]{{\widehat{#1}}}
\newcommand{\wt}[1]{{\widetilde{#1}}}

\usepackage{ifthen}

\newcommand{\invlim}[1][]{\ifthenelse{\equal{#1}{}}% falls Argument leer
{\DS \lim_{\longleftarrow}}%                         verwende niedrige Version
{\DS \lim_{\underset{#1}{\longleftarrow}}}%  sonst:  verwende Argument
}

\newcommand{\dirlim}[1][]{\ifthenelse{\equal{#1}{}}% falls Argument leer
{\DS \lim_{\longrightarrow}}%                        verwende niedrige Version
{\DS \lim_{\underset{#1}{\longrightarrow}}}% sonst:  verwende Argument
}

\newcommand{\dbl}{{\mathchoice{\mbox{\rm [\hspace{-0.15em}[}}
                              {\mbox{\rm [\hspace{-0.15em}[}}
                              {\mbox{\scriptsize\rm [\hspace{-0.15em}[}}
                              {\mbox{\tiny\rm [\hspace{-0.15em}[}}}}
\newcommand{\dbr}{{\mathchoice{\mbox{\rm ]\hspace{-0.15em}]}}
                              {\mbox{\rm ]\hspace{-0.15em}]}}
                              {\mbox{\scriptsize\rm ]\hspace{-0.15em}]}}
                              {\mbox{\tiny\rm ]\hspace{-0.15em}]}}}}
\newcommand{\dpl}{{\mathchoice{\mbox{\rm (\hspace{-0.15em}(}}
                              {\mbox{\rm (\hspace{-0.15em}(}}
                              {\mbox{\scriptsize\rm (\hspace{-0.15em}(}}
                              {\mbox{\tiny\rm (\hspace{-0.15em}(}}}}
\newcommand{\dpr}{{\mathchoice{\mbox{\rm )\hspace{-0.15em})}}
                              {\mbox{\rm )\hspace{-0.15em})}}
                              {\mbox{\scriptsize\rm )\hspace{-0.15em})}}
                              {\mbox{\tiny\rm )\hspace{-0.15em})}}}}

\newcommand{\BFZ}{{\BF_q\dbl\zeta\dbr}}

\newcommand{\dotBD}{\vbox{\hbox{\kern2pt\bf.}\vskip-4.5pt\hbox{$\BD$}}}

\usepackage{color}
\newcounter{commentcounter}

\def\?{\ 
{\bf\color{red}???}\ 
\immediate\write16{}
\immediate\write16{Warning: There was still a question mark . . . }
\immediate\write16{}}

\DeclareMathOperator{\QIsog}{QIsog}
\DeclareMathOperator{\Nilp}{\CN \!{\it ilp}}

\newcommand{\NilpF}{{\Nilp_{\BFZ}}}
\DeclareMathOperator{\Sets}{\CS \!{\it ets}}
\DeclareMathOperator{\Gr}{Gr}
\def\ulM{{\underline{M\!}\,}{}}

\def\s{\sigma^\ast}

\def\longto{\longrightarrow}

\def\isoto{\stackrel{}{\mbox{\hspace{1mm}\raisebox{+1.4mm}{$\SC\sim$}\hspace{-3.5mm}$\longrightarrow$}}}

\newbox\mybox
\def\arrover#1{\mathrel{
       \setbox\mybox=\hbox spread 1.4em{\hfil$\scriptstyle#1$\hfil}
       \vbox{\offinterlineskip\copy\mybox
             \hbox to\wd\mybox{\rightarrowfill}}}}

\newcommand{\ppsi}{\delta}
\newcommand{\X}{\Gr}
\newcommand{\Defo}{\mathcal{D}}
\newcommand{\DefoI}{{\overline{\mathcal{D}}_I}}
\newcommand{\Mnull}{{\underline{\CM}}}
\newcommand{\Test}{{Y}}

\begin{document}

%%%%%%%%%%%%%%%%%%%%%%%%%%%%%%%%%%%%%%%%%%%%%%%%%%%%%%%%%%%%%%%%%%%%%%
\author{Urs Hartl and Eva Viehmann}

\title{The Newton stratification on deformations of local $G$-shtukas}

\maketitle

\begin{abstract}
Bounded local $G$-shtukas are function field analogs for $p$-divisible groups with extra structure. We describe their deformations and moduli spaces. The latter are analogous to Rapoport-Zink spaces for $p$-divisible groups. The underlying schemes of these moduli spaces are affine Deligne-Lusztig varieties. For basic Newton polygons the closed Newton stratum in the universal deformation of a local $G$-shtuka is isomorphic to the completion of a corresponding affine Deligne-Lusztig variety in that point. This yields bounds on the dimension and proves equidimensionality of the basic affine Deligne-Lusztig varieties.
   
\noindent
{\it Mathematics Subject Classification (2000)\/}: 
20G25   % Linear algebraic groups over local fields and their integers
%11F80,  % Galois representations (Discontinuous groups and automorphic forms)
(11G09, % Drinfeld Modules, higher dimensional motives
%11S20,  % Galois theory of local and $p$-adic fields
%11S25,  % Galois cohomology of local and $p$-adic fields
%13A35,  % Characteristic $p$ methods (Frobenius endomorphism) ...
%14F30,  % $p$-adic cohomology, crystalline cohomology
%14G20,  % Local ground fields
%14G22,  % Rigid analytic geometry
14L05,  % Formal groups, $p$-divisible groups
14M15)  % Grassmannians, Schubert varieties, flag manifolds
\end{abstract}

\bigskip

%%%%%%%%%%%%%%%%%%%%%%%%%%%%%%%%%%%%%%%%%%%%%%%%%%%%%%%%%%%%%%%%%%%%%%
%
%    Introduction
%
%%%%%%%%%%%%%%%%%%%%%%%%%%%%%%%%%%%%%%%%%%%%%%%%%%%%%%%%%%%%%%%%%%%%%%

\section{Introduction}
%\addcontentsline{toc}{section}{Introduction}

Deformations and moduli spaces of $p$-divisible groups play an important role for the local theory and the reduction modulo $p$ of Shimura varieties. A first case was studied by Drinfeld~\cite{Drinfeld2} who used such a moduli space to uniformize certain Shimura curves. Generalizing Drinfeld's result, Rapoport and Zink~\cite{RZ} constructed formal schemes over $\mathbb{Z}_p$ parametrizing $p$-divisible groups together with a quasi-isogeny to a fixed $p$-divisible group, and also variants including extra structure such as a polarization, endomorphisms, or a level structure. These spaces are used to uniformize Shimura varieties (for a corresponding group, i.e. a restriction of scalars of some general linear or symplectic group) along Newton strata.

In this article we consider an analog over the local function field $\BF_q\dpl z\dpr$. Here we are not restricted to the general linear or symplectic groups. We replace $p$-divisible groups by so-called local $G$-shtukas for any split connected reductive group $G$ over the finite field $\BF_q$. 
To define them, let $LG$ be the loop group of $G$, that is, the ind-scheme over $\BF_q$ representing the sheaf of groups for the \fpqc-topology whose sections for an $\BF_q$-algebra $R$ are given by $LG(\Spec R)\,=\,G\bigl(R\dbl z\dbr[\frac{1}{z}]\bigr)$. Let $K$ be the infinite dimensional affine group scheme over $\BF_q$ with $K(\Spec R)\,=\,G\bigl(R\dbl z\dbr\bigr)$. Every $K$-torsor over an $\BF_q$-scheme $S$ for the \fpqc-topology is already a $K$-torsor for the \'etale topology (Proposition~\ref{PropGzTorsor}). For such a $K$-torsor $\CG$ on $S$ let $\CL\CG$ be the associated $LG$-torsor and $\s\CL\CG$ the pullback of $\CL\CG$ under the $q$-Frobenius morphism $\Frob_q:S\to S$.
As our base scheme for a local $G$-shtuka let $S$ be an $\BF_q\dbl z\dbr$-scheme on which $z$ is locally nilpotent and denote the image of $z$ in $\CO_S$ by $\zeta$. Then a \emph{local $G$-shtuka} over $S$ is a pair $\ul\CG=(\CG,\phi)$ consisting of a $K$-torsor $\CG$ on $S$ and an isomorphism of $LG$-torsors $\phi:\s\CL\CG\isoto\CL\CG$. Let $k$ be an algebraically closed field extension of $\BF_q$. Local $G$-shtukas over $k$ can also be described as follows. There exists a trivialization $\CG\cong K_k$ and with respect to such a trivialization $\phi$ corresponds to an element $b\in LG(k)$. A change of the trivialization replaces $b$ by $g^{-1}b\s(g)$ for $g\in K(k)$, where $\s$ is the endomorphism of $K$ and of $LG$ induced by $\Frob_q:S\to S$.

An important invariant of $\ul\CG$ over $k$ is its \emph{Hodge polygon} which controls the relative position of $\CG$ and the image of  $\s\CG$ under $\phi$. We fix a Borel subgroup $B$ of $G$ containing a split maximal torus $T$. Let $b\in LG(k)$ be as above. Then the Hodge polygon of $\ul\CG$ is defined (using the Cartan decomposition) to be the unique dominant cocharacter $\mu_{\ul\CG}\in X_\ast(T)$ with $b\in K(k)z^{\mu_{\ul\CG}}K(k)$. Clearly $\mu_{\ul\CG}$ does not depend on the chosen trivialization $\CG\cong K_k$. If $\mu$ is any dominant cocharacter of $G$ we say that $\ul\CG$ is \emph{bounded by $\mu$} if $\mu_{\ul\CG}\preceq\mu$ in the Bruhat ordering of $X_*(T)$. In Definition~\ref{DefBounded} we also extend this ad hoc definition of boundedness to not necessarily reduced base schemes $S$.

The \emph{affine Grassmannian} is the quotient sheaf $\X=LG/K$ for the \fppf-topology. It is an ind-scheme over $\mathbb{F}_q$ which is of ind-finite type; see \cite[\S4.5]{BeilinsonDrinfeld}, \cite{BeauvilleLaszlo}, \cite{LaszloSorger}, who comprehensively develop the theory of the affine Grassmannian over the field of complex numbers. Most of their results, and in particular all we use here, also hold with the same proofs over $\BF_q$. For instance \cite{Ngo-Polo}, \cite{Faltings03} reprove some statements. Moreover \cite{PR2} present a generalization of results and proofs to twisted affine flag varieties over $\BF_q$. 

One main result of our present article, Theorem~\ref{ThmDefoSp}, identifies the universal deformation space $\Spf \Defo$ for deformations bounded by $\mu$ of a local $G$-shtuka $\ul\CG$ over a field $k'$ with the completion of a closed subscheme of $\X\times_{\Spec \mathbb{F}_q} \Spec k'\dbl\zeta\dbr$ determined by $\mu$. It is noetherian of relative dimension $\langle 2\rho,\mu\rangle$ over $k'\dbl\zeta\dbr$, and $\ol\Defo:=(\Defo/\zeta \Defo)_\red$ is normal and Cohen-Macaulay. Here $\rho$ is the half-sum of the positive roots of $G$. We call $\Spf \ol\Defo$ the \emph{universal $(\zeta=0)$-deformation space} of $\ul\CG$.

We construct moduli spaces of local $G$-shtukas $\ul\CG$ bounded by $\mu$ together with a quasi-isogeny $\ppsi$ to a fixed trivializable local $G$-shtuka $\ul\BG$ over $k'$, which are the analogs of the moduli spaces of $p$-divisible groups defined by Rapoport and Zink, \cite{RZ}. They are formal schemes locally formally of finite type over $k'\dbl\zeta\dbr$ (Theorem~\ref{ThmADLV=RZSp}) whose underlying topological spaces can be described as follows. Inside the affine Grassmannian we consider the \emph{affine Deligne-Lusztig variety} which is defined to be the locally closed reduced ind-subscheme of $\X$ whose points over an algebraically closed extension $k$ of $k'$ are
\[
X_\mu(b)(k)\;:=\;\bigl\{\,g\in \X(k)=LG(k)/K(k):\es g^{-1}b\s(g)\in K(k)z^\mu K(k)\,\bigr\}
\]
and the \emph{closed affine Deligne-Lusztig variety} $X_{\preceq\mu}(b)\;=\;\bigcup_{\mu'\preceq\mu}X_{\mu'}(b)$. They are schemes locally of finite type over $k'$ (Corollary~\ref{ADLVisScheme}). The underlying topological space of the Rapoport-Zink space of local $G$-shtukas bounded by $\mu$ is isomorphic to $X_{\preceq\mu}(b)$ where $b\in LG(k')$ describes the Frobenius $\phi$ of $\ul\BG$. 

We consider a second invariant of a local $G$-shtuka $\ul\CG$ over $k$, its \emph{Newton polygon}. It is the quasi-cocharacter $\nu_{\ul\CG}\in X_\ast(T)_\BQ$ associated by Kottwitz \cite{Kottwitz85,Kottwitz97} to the $\sigma$-conjugacy class of $b\in LG(k)$. Kottwitz's articles only consider the analogous case of $\sigma$-conjugacy classes of elements $b\in G(\Quot(W(k))$ where $W(k)$ is the ring of Witt vectors over $k$. The arguments carry over literally to the equal characteristic case. By definition $\nu_{\ul\CG}$ is invariant under isogeny. Under specialization Hodge and Newton polygon behave like those of $p$-divisible groups: there is a generalization of Grothendieck's specialization theorem by Rapoport and Richartz \cite{RapoportRichartz}. By results of Vasiu \cite{Vasiu} the Newton stratification is pure, that is the Newton polygon of a local $G$-shtuka over a connected scheme jumps in codimension one or is constant (Theorem~\ref{reinheit}). Using these two properties we study the Newton stratification on the universal $(\zeta=0)$-deformation space of an arbitrary local $G$-shtuka $\ul\CG$ over $k$ and give lower bounds for the dimensions of the strata (Proposition~\ref{propnewtondim}). 

We show that a local $G$-shtuka over a reduced complete local ring whose Newton polygon $\nu=\nu_{\ul\CG}$ is \emph{basic}, that is central in $G$, is isogenous to a constant local $G$-shtuka (Proposition \ref{prop41}), which is an analog of a result of Oort and Zink, \cite[Proposition 3.3]{OortZink}. It allows to compare the Newton stratum $\CN_\nu$ in the universal $(\zeta=0)$-deformation to the Rapoport-Zink space. Here we use a bijection between bounded local shtukas over $\Spec R$ and $\Spf R$ for any complete local ring $R$ of characteristic $p$ (Proposition \ref{remdejong}) to pass from the universal deformation (which is a formal scheme) to the associated scheme. We obtain

\begin{theorem}\label{thmnewtonadlv}
Let $b\in LG(k)$ be basic and let $\mu\in X_\ast(T)$ be dominant and such that $X_{\preceq\mu}(b)\neq \emptyset$. Let $g\in X_{\preceq \mu}(b)$ be a $k$-valued point and let $(X_{\preceq\mu}(b))\kompl_g$ be the completion of $X_{\preceq\mu}(b)$ in this point. Let $\mathcal{N}_{\nu}$ be the basic Newton stratum in the universal $(\zeta=0)$-deformation of $\ul\CG=\bigl(K_k,g^{-1}b\s(g)\s\bigr)$ bounded by $\mu$. Then $(X_{\preceq\mu}(b))\kompl_g$ is canonically isomorphic to $\mathcal{N}_{\nu}$.
\end{theorem}

Using the dimension formula for affine Deligne-Lusztig varieties in \cite{GHKR}, \cite{Viehmann06}, and Theorem \ref{thmnewtonadlv}, we show that these varieties are also equidimensional. A similar strategy is used by de Jong and Oort in \cite[Proposition 5.19 and Corollary 5.20]{deJongOort} to show equidimensionality of a moduli space of $p$-divisible groups.

\begin{theorem}\label{thmequidim}
Let $b\in LG(k)$ be basic and let $\mu\in X_\ast(T)$ be dominant and such that $X_{\preceq\mu}(b)\neq \emptyset$. Then $X_{\mu}(b)$ and $X_{\preceq \mu}(b)$ are equidimensional of dimension $\langle \rho, \mu\rangle-\frac{1}{2}{\rm def}(b)$. Here $\rho$ is the half-sum of the positive roots of $G$ and ${\rm def}(b)=\rk(G)-\rk(J)$ is the defect of $b$ (see \ref{remnewtondim}).
\end{theorem}

For $X_{\mu}(b)$, this finishes the proof of a conjecture by Rapoport \cite[Conjecture 5.10]{RapoportGuide}. For $X_{\mu}(b)$ with $b\in T\bigl(k\dpl z\dpr\bigr)$ (not necessarily basic) Theorem \ref{thmequidim} has been shown in \cite[Proposition 2.17.1]{GHKR}.  For $X_{\preceq\mu}(b)$, this theorem immediately implies

\begin{corollary}
Let $b\in LG(k)$ be basic and $\mu\in X_\ast(T)$ be dominant. Then $X_{\preceq\mu}(b)$ is the closure of $X_{\mu}(b)$ in $\X$.
\end{corollary}

\begin{proof}
We have $\langle \rho, \mu'\rangle<\langle \rho, \mu\rangle$ for every $\mu'\preceq \mu$ with $\mu'\neq\mu$. Thus $\dim X_{\mu'}(b)<\dim X_{\mu}(b)$ for each such $\mu'$. Hence equidimensionality of $X_{\preceq\mu}(b)$ implies that $X_{\mu}(b)$ is dense in $X_{\preceq\mu}(b)$.
\end{proof}

A more careful analysis of the proof of Theorem \ref{thmequidim} yields 
\begin{corollary}\label{corunivnp}
Let $\ul\CG$ be a local $G$-shtuka over $k$ which is basic and bounded by some $\mu\in X_\ast(T)$. Let $\ol\Defo$ be the universal $(\zeta=0)$-deformation ring for deformations of $\ul\CG$ that are bounded by $\mu$. Then the generic fibre of the local $G$-shtuka over $\Spec\ol\Defo$ corresponding to the universal family over $\Spf \ol\Defo$ has Newton polygon $\mu$.
\end{corollary}

We prove Theorem~\ref{thmnewtonadlv}, Theorem~\ref{thmequidim} and Corollary~\ref{corunivnp} in Section~\ref{SectBasic}.

In the non-basic case the dimensions of the affine Deligne-Lusztig varieties and the lower bound on the dimension of the corresponding (non-basic) Newton stratum differ by $\langle 2\rho, \nu\rangle$. Especially, Theorem \ref{thmnewtonadlv} does not hold in this general situation. The same difference between the dimensions already occurs for $p$-divisible groups. There, work of Oort \cite{Oort04} shows that up to a finite morphism the Newton stratum $\CN_\nu$ is a product of the underlying scheme of a Rapoport-Zink space with a so-called central leaf, and by \cite{Oort08}, or \cite{Chai2} the difference $\langle 2\rho, \nu\rangle$ between the dimensions is the dimension of this central leaf. In a sequel to this article we show that a similar structure exists in the non-basic case for our deformations of local $G$-shtukas.

In the last section we define local $G$-shtukas with Iwahori-level structure. Using a variant of Theorem \ref{thmnewtonadlv} we can compare the Newton stratum of an associated universal $(\zeta=0)$-deformation with an affine Deligne-Lusztig variety inside the affine flag manifold in the basic case. This proves the basic case of a conjecture by Beazley \cite[Conjecture 1]{Beazley}. The strategy used to prove Theorem~\ref{thmequidim} also gives a lower bound on the dimension of the affine Deligne-Lusztig variety in the Iwahori setting. However, in this case the lower bound depends on the maximal length of a chain of Newton polygons in the universal deformation, for which a general formula is not known to us.

\bigskip

\noindent
\emph{Acknowledgement.} We are grateful to O.~Gabber, R.~Kottwitz, and W.~Soergel for helpful explanations and to A.\ Genestier for many stimulating discussions. We also thank W.~Soergel for providing the proof of Proposition~\ref{PropFaithfulRep} and M.~E.~Arasteh Rad for telling us about an inaccuracy in a previous version of this article. The first author acknowledges support of the DFG (German Research Foundation) in form of DFG-grant HA3006/2-1 and SFB 478, project C5. This work was also supported by the SFB/TR 45 ``Periods, Moduli Spaces and Arithmetic of Algebraic Varieties'' of the DFG.

\tableofcontents

%%%%%%%%%%%%%%%%%%%%%%%%%%%%%%%%%%%%%%%%%%%%%%%%%%%%%%%%%%%%%%%%%%%%%%
%
%    Notation
%
%%%%%%%%%%%%%%%%%%%%%%%%%%%%%%%%%%%%%%%%%%%%%%%%%%%%%%%%%%%%%%%%%%%%%%

\section{Notation and complements on torsors for loop groups}\label{SectNotation}

We denote by $\BF_q$ the finite field with $q$ elements and by $k$ an algebraically closed field which is also an $\BF_q$-algebra. Let $\BF_q\dbl z\dbr$ be the ring of power series in the indeterminate $z$ and let $\Nilp_{\BF_q\dbl z\dbr}$ be the category of $\BF_q\dbl z\dbr$-schemes on which $z$ is locally nilpotent. For a scheme $S$ in $\Nilp_{\BF_q\dbl z\dbr}$ we denote by $\zeta\in\CO_S$ the image of $z$ and we write from now on $S\in\NilpF$. We view $z$ and $\zeta$ as algebraically independent indeterminates.

A reference for the various topologies ($\fpqc$, $\fppf$, $\et$ in addition to the Zariski topology) that we consider on the category of all $\BF_q$-schemes is \cite[Expos\'e IV, \S6.3]{SGA3}. Let $\CO_S\dbl z\dbr$ be the sheaf of $\CO_S$-algebras on $S$ for the \fpqc-topology whose ring of sections on an $S$-scheme $\Test$ is the ring of power series $\CO_S\dbl z\dbr(\Test):=\Gamma(\Test,\CO_\Test)\dbl z\dbr$.  This is indeed a sheaf being the countable direct product of $\CO_S$. Let $\CO_S\dpl z\dpr$ be the \fpqc-sheaf of $\CO_S$-algebras on $S$ associated with the presheaf $\Test\mapsto\Gamma(\Test,\CO_\Test)\dbl z\dbr[\frac{1}{z}]$. If $\Test$ is quasi-compact then $\CO_S\dpl z\dpr(\Test)=\Gamma(\Test,\CO_\Test)\dbl z\dbr[\frac{1}{z}]$ by \cite[Exercise II.1.11]{Hartshorne}. We denote by $\s$ the endomorphism of $\CO_S\dbl z\dbr$ and $\CO_S\dpl z\dpr$ that acts as the identity on the variable $z$, and as $b\mapsto b^q$ on local sections $b\in \CO_S$. For a sheaf $M$ of $\CO_S\dbl z\dbr$-modules on $S$ we set $\s M:=M\otimes_{\CO_S\dbl z\dbr,\s}\CO_S\dbl z\dbr$. 

Let $G$ be a split connected reductive group over $\BF_q$. Let $B\supset T$ be a Borel subgroup and a maximal split torus of $G$. Recall that a weight $\lambda\in X^\ast(T)$ is called dominant if $\langle\lambda,\alpha\dual\rangle\ge0$ for every positive coroot $\alpha\dual$ of $G$, and similarly for coweights. We consider the ordering $\preceq$ on the group of coweights $X_\ast(T)$ of $G$, which is defined as $\mu_1\preceq\mu_2$ if and only if the difference $\mu_2-\mu_1$ is a non-negative integral linear combination of simple coroots. If $\mu_1,\mu_2$ are dominant it coincides with the Bruhat ordering. On $X^*(T)$ we consider the analogous ordering. On $X_*(T)_{\mathbb{Q}}=X_*(T)\otimes_{\mathbb{Z}}\mathbb{Q}$ we use the ordering with $\mu_1\preceq\mu_2$ if and only if the difference $\mu_2-\mu_1$ is a non-negative rational linear combination of simple coroots. For $\mu\in X_*(T)$ we denote by $z^{\mu}$ the image of $z\in\mathbb{G}_m\bigl(\mathbb{F}_q\dpl z\dpr\bigr)$ under
  the morphism $\mu:\mathbb{G}_m\rightarrow T$. 

Let $LG$ be the loop group of $G$, see \cite[Definition 1]{Faltings03}. That is, $LG$ is the ind-scheme of ind-finite type over $\BF_q$ representing the sheaf of groups for the \fpqc-topology on $\BF_q$-schemes $S$ given by
\[
\TS LG(S)\;=\;G\bigl(\CO_S\dpl z\dpr(S)\bigr)\,.
\]
On $\NilpF$ the local nilpotency of $\zeta$ on $S$ implies that the sheaf $LG$ is canonically isomorphic to the sheaf associated with the presheaf $S\mapsto G\bigl(\Gamma(S,\CO_S)\dbl z\dbr[\frac{1}{z-\zeta}]\bigr)$. We denote by $K$ the infinite dimensional affine group scheme over $\BF_q$ whose $S$-valued points for an $\BF_q$-scheme $S$ are $K(S)\,=\,G\bigl(\CO_S\dbl z\dbr(S)\bigr)\,=\,G\bigl(\Gamma(S,\CO_S)\dbl z\dbr\bigr)$. In the natural way $K$ can be viewed as a subsheaf of $LG$. For any scheme $S\in\NilpF$ we denote by $K_S$ the induced sheaf of groups for the \fpqc-topology on $S$. 
The endomorphism $\s$ induces an endomorphism of $LG$ and $K$. 

Let $\topol\in\{\fpqc,\fppf,\et\}$. Let $S\in\NilpF$ and let $H$ be a sheaf of groups on $S$ for the $\topol$-topology. In this article a (\emph{right}) \emph{$H$-torsor for the $\topol$-topology} on $S$ is a sheaf $\CG$ for the $\topol$-topology on $S$ together with a (right) action of the sheaf $H$ such that $\CG$ is isomorphic to $H$ on a $\topol$-covering of $S$. Here $H$ is viewed as an $H$-torsor by right multiplication. The $H$-torsors are classified up to isomorphism by \v{C}ech cohomology $\CKoh^1(S_\topol,H)$, which is a pointed set. For $H=K$ the categories of torsors for the three different topologies $\fpqc$, $\fppf$, $\et$ are equivalent by Proposition~\ref{PropGzTorsor} below. Therefore we simply speak of \emph{$K$-torsors on $S$}.

Let $K'$ be equal to $K$ or to the Iwahori subgroup 
\[
I(S)\;:=\;\bigl\{\,g\in K(S)=G\bigl(\Gamma(S,\CO_S)\dbl z\dbr\bigr):\es g\mod z\in B(S)\,\bigr\}\,.
\]
Let $\CG$ be a $K'$-torsor on $S$ for the $\topol$-topology. Then $\CG$ is trivialized by a $\topol$-covering $S'\to S$ and $\CG_{S'}\cong K'_{S'}$ is an affine scheme. The covering $S'\to S$ is in particular an \fpqc-covering, and the \fpqc-descent is effective by \cite[\S6.1, Theorem 6]{BLR}. Hence $\CG$ is an (infinite dimensional) affine scheme over $S$. With $\CG$ we associate an $LG$-torsor $\CL\CG$ for the $\topol$-topology via the map $\CKoh^1(S_\topol,K')\to\CKoh^1(S_\topol,LG)$. Also for an $LG$-torsor $\CL\CG$ on $S$ we denote by $\s\CL\CG$ the pullback of $\CL\CG$ under the $q$-Frobenius morphism $\Frob_q:S\to S$.

The pointed set $\CKoh^1(S_\topol,K')$ also classifies another kind of torsors which we introduce now. Throughout this article (i.e.~for the definition of local $G$-shtukas) we only use the first kind of torsors. However, the second interpretation of $\CKoh^1(S_\topol,K')$ turns out to be useful to compare these sets for the various topologies in Proposition \ref{PropGzTorsor}. Besides, it relates the torsors underlying local $G$-shtukas to global versions as studied by Varshavsky~\cite{Varshavsky}. To explain it let $S$ be an $\BF_q$-scheme and let $S\dbl z\dbr$ be the formal scheme over $\Spf\BF_q\dbl z\dbr$ in the sense of \cite[I$_{new}$, Section~10]{EGA} consisting of the topological space $S$ endowed with the structure sheaf $\CO_S\dbl z\dbr$. Equivalently $S\dbl z\dbr$ is the formal completion of $\BA^1_S=S\times_{\BF_q}\Spec\BF_q[z]$ along the closed subscheme $S\times_{\BF_q}\Var(z)$. Let $G'$ be the smooth affine parahoric group scheme over $\mathbb{F}_q\dbl z\dbr$ with $K'(\mathbb{F}_q)=G'(\mathbb{F}_q\dbl z\dbr)$, with generic fiber $G\times_{\BF_q}\BF_q\dpl z\dpr$ and connected special fiber; see \cite{HR} or \cite[4.6.2, 4.6.26, and 5.2.6]{BT}. We write $\wh{G'}=G'\times_{\Spec\BF_q\dbl z\dbr}\Spf\BF_q\dbl z\dbr$ for the $z$-adic formal completion of $G'$. It is an affine formal scheme over $\Spf\BF_q\dbl z\dbr$. There is an equality
\begin{eqnarray*}
K'(S)&=&G'\bigl(\Gamma(S,\CO_S)\dbl z\dbr\bigr)\\[1mm]
&=&\Hom_{\BF_q\dbl z\dbr}\bigl(\Gamma(G',\CO_G')\,,\,\Gamma(S,\CO_S)\dbl z\dbr\bigr)\\[1mm]
&=&\Hom_{\BF_q\dbl z\dbr}^\cont\bigl(\Gamma(\wh{G'},\CO_{\wh{G'}})\,,\,\Gamma(S,\CO_S)\dbl z\dbr\bigr)\\[1mm]
&=&\Hom_{\Spf\BF_q\dbl z\dbr}\bigl(S\dbl z\dbr\,,\,\wh{G'}\bigr)
\end{eqnarray*}
where the latter $\Hom$ group denotes morphisms of formal schemes over $\Spf\BF_q\dbl z\dbr$, and the last equation holds by \cite[I$_{\rm new}$, Proposition 10.4.6]{EGA} since $\wh{G'}$ is affine.

\begin{definition}\label{DefGzTorsor}
Let $\topol\in\{\fpqc,\fppf,\et\}$. We denote by $\PHS_\topol(S\dbl z\dbr,\wh{G'})$ (for principal homogenous space) the set of isomorphism classes of $z$-adic formal schemes $\wh\CG$ over $S\dbl z\dbr$ together with an action $\wh{G'}\whtimes_{\Spf\BF_q\dbl z\dbr}\wh\CG\to\wh\CG$ of $\wh{G'}$ on $\wh\CG$ such that there is a covering $S'\dbl z\dbr\to S\dbl z\dbr$ where $S'\to S$ is a $\topol$-covering and a $\wh{G'}$-equivariant isomorphism $\alpha:\wh\CG\whtimes_{S\dbl z\dbr}S'\dbl z\dbr\isoto\wh{G'}\whtimes_{\Spf\BF_q\dbl z\dbr}S'\dbl z\dbr$. Here $\wh{G'}$ acts on itself by right multiplication. The set $\PHS_\topol(S\dbl z\dbr,\wh{G'})$ is a pointed set, the distinguished point being the trivial torsor.
\end{definition}

\begin{proposition}\label{PropGzTorsor}
\begin{enumerate}
\item 
There is a natural bijection of pointed sets $\CKoh^1(S_\topol,K')=\PHS_\topol(S\dbl z\dbr,\wh{G'})$.
\item 
$\CKoh^1(S_\et,K')\;=\;\CKoh^1(S_\fppf,K')\;=\;\CKoh^1(S_\fpqc,K')$.
\item 
Let $S=\Spec A$ be an affine $\BF_q$-scheme, let $\Fa\subset A$ be a nilpotent ideal and set $\ol A=A/\Fa$ and $\ol S=\Spec\ol A$. Then for any $K'$-torsor $\CG$ over $S$ whose pull back to $\ol S$ admits a trivialization $\bar\alpha:\CG_{\ol S}\isoto K'_{\ol S}$, there is a trivialization $\alpha:\CG\isoto K'_S$ lifting $\bar\alpha$. Likewise for any formal $\wh{G'}$-torsor $\wh\CG$ over $S\dbl z\dbr$ as in Definition~\ref{DefGzTorsor} whose pull back to $\ol S\dbl z\dbr$ admits a trivialization $\bar\alpha$, there is a trivialization $\alpha$ of $\wh\CG$ lifting $\bar\alpha$.
\end{enumerate}
\end{proposition}

We are mainly interested in statement (c) for $K'$ and in statement (b). However, we prove them using $\wh{G'}$.

\begin{proof}
(a) Let $\wh\CG$ be a formal $\wh{G'}$-torsor representing a class in the set $\PHS_\topol(S\dbl z\dbr,\wh{G'})$, and let $\alpha:\wh\CG\whtimes_{S\dbl z\dbr}S'\dbl z\dbr\isoto \wh{G'}\whtimes_{\Spf\BF_q\dbl z\dbr}S'\dbl z\dbr$ be a trivialization over $S'\dbl z\dbr$ for a $\topol$-covering $S'\to S$. We can pull back $\alpha$ to $S'\dbl z\dbr\whtimes_{S\dbl z\dbr}S'\dbl z\dbr=S''\dbl z\dbr$ for $S''=S'\times_S S'$ under the two projections $p_i:S''\dbl z\dbr\to S'\dbl z\dbr$. Then $h:=p_2^\ast\alpha\circ p_1^\ast\alpha^{-1}$ is a $\wh{G'}$-equivariant automorphism of $\wh{G'}\whtimes_{\Spf\BF_q\dbl z\dbr}S''\dbl z\dbr$, hence given by an element $h\in\Hom_{\Spf\BF_q\dbl z\dbr}\bigl(S''\dbl z\dbr\,,\,\wh{G'}\bigr)\,=\,K'(S'')$. Clearly $h$ is a $1$-cocycle for the covering $S'\to S$ and induces a cohomology class in $\CKoh^1(S_\topol,K')$. 

Conversely any such cohomology class is represented by a $\topol$-covering $S'\to S$ and a $1$-cocycle $h\in K'(S'')$. The latter defines a descent datum on $\wh{G'}\whtimes_{\Spf\BF_q\dbl z\dbr}S'\dbl z\dbr$ which we must prove to be effective. Let $n$ be a positive integer and consider the situation modulo $z^n$. Then $h\pmod{z^n}$ is a descent datum on 
\[
\wh{G'}\whtimes_{\Spf\BF_q\dbl z\dbr}S'\dbl z\dbr\whtimes_{\Spf\BF_q\dbl z\dbr}\Spec\BF_q\dbl z\dbr/(z^n)\es=\es G'\times_{\Spec\BF_q\dbl z\dbr}\bigl(S'\times_{\BF_q}\Spec\BF_q[z]/(z^n)\bigr)\,.
\]
Since $G'$ is affine and $S'\times_{\BF_q}\Spec\BF_q[z]/(z^n)\to S\times_{\BF_q}\Spec\BF_q[z]/(z^n)$ is an \fpqc-covering by base change from $S'\to S$, the descent is effective by \cite[\S6.1, Theorem 6a]{BLR} and we obtain an affine scheme $\wh\CG_n$ over $S\times_{\BF_q}\Spec\BF_q[z]/(z^n)$. It is of finite presentation and smooth by \cite[IV$_2$, Proposition 2.7.1 and IV$_4$, Corollaire 17.7.3]{EGA} since the same holds for $G'$ over $\BF_q\dbl z\dbr$. For varying $n$ we obtain an inductive system $\wh\CG_n$ with $\wh\CG_{n+1}\times_{\BF_q\dbl z\dbr/(z^{n+1})}\BF_q\dbl z\dbr/(z^n)=\wh\CG_n$ whose limit exists as a $z$-adic affine formal scheme $\wh\CG$ over $S\dbl z\dbr$ with underlying topological space $\wh\CG_1$ by \cite[I$_{\rm new}$, Corollary 10.6.4]{EGA}.

\medskip\noindent
(b) Due to (a) and the injectivity of the natural maps $\PHS_\et(S\dbl z\dbr,\wh{G'})\hookrightarrow\PHS_\fppf(S\dbl z\dbr,\wh{G'})\hookrightarrow\PHS_\fpqc(S\dbl z\dbr,\wh{G'})$ we only need to prove surjectivity of these maps. So let $\wh\CG$ be a formal $\wh{G'}$-torsor over $S\dbl z\dbr$ representing an element of $\PHS_\fpqc(S\dbl z\dbr,\wh{G'})$. We must show that $\wh\CG$ is trivial over $S'\dbl z\dbr$ for an \'etale covering $S'\to S$. As in (a), $\wh\CG=\dirlim\wh\CG_n$ is the limit of smooth affine schemes $\wh\CG_n$ over $S\times_{\BF_q}\Spec\BF_q[z]/(z^n)$. Since $\wh\CG_1$ is smooth over $S$ it has a section $s_1:S\to\wh\CG_1$ over an \'etale covering $S'\to S$. We may further assume that $S'$ is (the disjoint union of a family of) affine (schemes). Then the section $s_1$ lifts inductively for all $n$ to a section $s_n:S'\times_{\BF_q}\Spec\BF_q[z]/(z^n)\to\wh\CG_n$ by the smoothness of $\wh\CG_n$. In the limit this gives a section $s:S'\dbl z\dbr\to\wh\CG$ which induces a trivialization of $\wh\CG\whtimes_{S\dbl z\dbr}S'\dbl z\dbr$ over $S'\dbl z\dbr$ as desired.
 
\medskip\noindent
(c) We first consider a cohomology class in $\CKoh^1(S_\et,K')$ whose image under $\CKoh^1(S_\et,K')\to \CKoh^1(\ol S_\et,K')$ is trivial. By (a) this class corresponds to a formal $\wh{G'}$-torsor $\wh\CG$ over $\Spf A\dbl z\dbr$ whose pull back to $\Spf\ol A\dbl z\dbr$ is trivial. As above $\wh\CG$ is the limit of $\Spec A[z]/(z^n)$-schemes $\wh\CG_n$ which are of finite presentation and smooth. By the triviality of $\wh\CG\whtimes_{\Spf A\dbl z\dbr}\Spf\ol A\dbl z\dbr$ there exists a section $\bar\alpha:\Spf\ol A\dbl z\dbr\to\wh\CG$. Modulo $z$ it induces a section $\Spec \ol A\to\wh\CG_1$ which lifts by the smoothness of $\wh\CG_1$ to a section $\Spec A\to\wh\CG_1$. It inductively lifts by the smoothness of $\wh\CG_n$ to a section $\Spec A[z]/(z^n)\to\wh\CG_n$. In the limit we obtain a section $\beta:\Spf A\dbl z\dbr\to\wh\CG$ (which in general does not lift the initial section $\bar\alpha:\Spf\ol A\dbl z\dbr\to\wh\CG$). Nevertheless, the section yields a trivialization of $\wh\CG$ over $\Spf A\dbl z\dbr$ and this proves that the only cohomology class which becomes trivial under $\CKoh^1(S_\et,K')\to\CKoh^1(\ol S_\et,K')$ is the trivial one.

We have to prove the existence of a trivialization $\alpha$ which lifts $\bar\alpha$ both in the situation for formal $\wh{G'}$-torsors $\wh\CG$ and for $K'$-torsors $\CG$. Let the torsor be represented by $S'\to S$ and $h\in K'(S'')$. The trivialization $\bar\alpha$ after pull back to $\ol S$ of the formal $\wh{G'}$-torsor $\wh\CG$ over $\ol S\dbl z\dbr$, respectively the $K'$-torsor $\CG$ over $\ol S$, is given by an element $\bar g\in K'(\ol S')$ with $p_2^\ast(\bar g)\,p_1^\ast(\bar g)^{-1}=h\pmod \Fa$ in $K'(\ol S'')$. Similarly the trivialization $\beta$ of $\wh\CG$ over $S\dbl z\dbr$, respectively $\CG$ over $S$, is given by an element $f\in K'(S')$ with $p_2^\ast(f)\,p_1^\ast(f)^{-1}=h$ in $K'(S'')$. Let $\bar f$ be the pullback of $f$ to $\ol S$. Then $p_2^\ast(\bar f^{-1}\bar g)=p_1^\ast(\bar f^{-1}\bar g)$ in $K'(\ol S'')$ implies that $\bar f^{-1}\bar g\in K'(\ol S)=G'\bigl(\ol A\dbl z\dbr\bigr)$. Since $\ker(A\dbl z\dbr\to\ol A\dbl z\dbr)$ is nilpotent and $G'$ is smooth over $\BF_q\dbl z\dbr$ we can lift $\bar f^{-1}\bar g$ to an element $f^{-1}g\in G'\bigl(\ol A\dbl z\dbr\bigr)$ with which we multiply $f$ to obtain an element $g\in K'(S')$ satisfying $p_2^\ast(g)\,p_1^\ast(g)^{-1}=h$ and lifting $\bar g\in K'(\ol S')$. Then $g$ induces a trivialization $\alpha$ of $\wh\CG$ over $S\dbl z\dbr$, respectively $\CG$ over $S$, which lifts the trivialization $\bar\alpha$ as desired. 
\end{proof}

\begin{proposition}[Hilbert 90 for loop groups]\label{PropHilbert90}
Let $G=\GL_r$. As before let $K$ be the associated \fpqc-sheaf over $\BF_q$ whose $\Test$-valued points are $K(\Test)=\GL_r\bigl(\Gamma(\Test,\CO_\Test)\dbl z\dbr\bigr)$. Then for any $\BF_q$-scheme $S$
\[
\CKoh^1(S_\Zar,K)\;=\;\CKoh^1(S_\et,K)\;=\;\CKoh^1(S_\fppf,K)\;=\;\CKoh^1(S_\fpqc,K)
\]
and this pointed set parametrizes isomorphism classes of locally free sheaves
of rank $r$ on the formal scheme $S\dbl z\dbr$, or equivalently sheaves of
$\CO_S\dbl z\dbr$-modules on $S$ which locally for the Zariski-topology on $S$
are isomorphic to $\CO_S\dbl z\dbr^{\oplus r}$. In particular, any \fpqc-sheaf of $\CO_S\dbl z\dbr$-modules on $S$ which is isomorphic to $\CO_S\dbl z\dbr^{\oplus r}$ \fpqc-locally on $S$ is already isomorphic to $\CO_S\dbl z\dbr^{\oplus r}$ Zariski-locally on $S$.
\end{proposition}

\begin{proof}
The second and third equality were proved in the preceding proposition. Consider a cohomology class in $\CKoh^1(S_\fpqc,K)$ and represent it by an \fpqc-covering $S'\to S$ and a 1-cocycle $h\in K(S'')$ for $S''=S'\times_SS'$. Let $p_i:S''\to S'$ be the projection onto the $i$-th factor. On the sheaf $\CF'=\CO_{S'}\dbl z\dbr^{\oplus r}$ consider the descent datum given by the isomorphism $h:p_1^\ast\CO_{S'}\dbl z\dbr^{\oplus r}\isoto p_2^\ast\CO_{S'}\dbl z\dbr^{\oplus r},\,v''\mapsto h\,v''$ over $S''$. In this way $\CKoh^1(S_\fpqc,K)$ parametrizes isomorphism classes of pairs $(S'\to S\,,\,h)$ where $S'\to S$ is an \fpqc-covering and $h$ is a descent datum on $\CF'=\CO_{S'}\dbl z\dbr^{\oplus r}$.

Modulo $z^n$ such a descent datum induces a descent datum on the trivial sheaf $\CF'_n$ of rank $r$ on $S'\times_{\BF_q}\Spec\BF_q[z]/(z^n)$. By \fpqc-descent \cite[\S6.1, Theorem 4]{BLR} and \cite[IV$_2$, Proposition 2.5.2]{EGA} the sheaf $\CF'_n$ descends to a locally free sheaf $\CF_n$ of rank $r$ on $S\times_{\BF_q}\Spec\BF_q[z]/(z^n)$. By \cite[I$_{\rm new}$, Proposition 10.10.8.6]{EGA} the projective limit $\CF=\invlim\CF_n$ is a locally (in the Zariski-topology) free sheaf of rank $r$ on $S\dbl z\dbr$. Since $S$ is the underlying topological space and $\CO_S\dbl z\dbr$ is the structure sheaf of the formal scheme $S\dbl z\dbr$ the sheaf $\CF$ is nothing else than a sheaf of $\CO_S\dbl z\dbr$-modules on $S$. This proves the proposition.
\end{proof}

%%%%%%%%%%%%%%%%%%%%%%%%%%%%%%%%%%%%%%%%%%%%%%%%%%%%%%%%%%%%%%%%%%%%%%
%
%   Local $G$-Shtukas 
%
%%%%%%%%%%%%%%%%%%%%%%%%%%%%%%%%%%%%%%%%%%%%%%%%%%%%%%%%%%%%%%%%%%%%%%

\section{Local $G$-shtukas} \label{SectLocGShtuka}
\setcounter{equation}{0}

\begin{definition} \label{DefLocGShtuka}
A \emph{local $G$-shtuka} over $S$ is a pair $\ul\CG=(\CG,\phi)$ consisting of a $K$-torsor $\CG$ on $S$ together with an isomorphism $\phi:\s\CL\CG\isoto\CL\CG$ of the associated $LG$-torsors.
\end{definition}

If $f:S'\to S$ is a morphism in $\NilpF$ there is a natural way to pull back a local $G$-shtuka $\ul\CG$ from $S$ to $S'$. We denote the pullback by $\ul\CG_{S'}$ or $f^\ast\ul\CG$.

\begin{remark}\label{RemEtaleTrivial}
In particular, if $\ul\CG=(\CG,\varphi)$ is a local $G$-shtuka over a noetherian complete local ring $R$ with algebraically closed residue field, then $\CG$ is isomorphic to the trivial torsor, and the morphism $\varphi$ can be written as $b\s$ for some $b\in LG(R)$. Indeed, $\CG$ is trivialized by an \'etale $R$-algebra by Proposition~\ref{PropGzTorsor}, but every \'etale $R$-algebra decomposes into a direct sum of copies of $R$.
\end{remark}

For $G=\GL_r$ local $G$-shtukas previously appeared in the literature as pairs $(M,\phi)$ where $M$ is a locally free sheaf of $\CO_S\dbl z\dbr$-modules of rank $n$, and where $\phi:\s M[\frac{1}{z-\zeta}]\isoto M[\frac{1}{z-\zeta}]$ is an isomorphism. They were first introduced by Anderson~\cite{Anderson2} in the case when $S$ is the spectrum of a complete discrete valuation ring. Genestier~\cite{Genestier} constructed Rapoport-Zink spaces for them in the Drinfeld case and used these to uniformize Drinfeld modular varieties. Local $\GL_r$-shtukas were further used and studied, among others, in \cite{HartlAbSh,HartlPSp}. (To be precise, in all these references they have to satisfy an additional boundedness condition in the style of Definition \ref{DefBounded} and Example \ref{exend} below.) The two definitions of local $\GL_r$-shtukas given above are equivalent by Proposition~\ref{PropHilbert90}. For more details see Section~\ref{SectG=GLr}. Local $\GL_r$-shtukas are analogs in the arithmetic of local function fields of $p$-divisible groups, or more precisely of $F$-crystals. To a large extent they behave similarly (see also \cite{HartlDict} and \cite{HartlAbSh} where they were called Dieudonn\'e $\BF_q\dbl z\dbr$-modules). Similarly to $p$-divisible groups which describe the local behaviour of abelian varieties at the prime $p$, there is also a global version of local $G$-shtukas. In \cite{Varshavsky} Varshavsky describes $G$-shtukas which are defined using a smooth projective curve instead of $\Spf\BF_q\dbl z\dbr$.

\medskip

We define the Hodge polygon of a local $G$-shtuka. Let $S\in \NilpF$. Let $\CG$ and $\CH$ be $K$-torsors on $S$ and let $\delta:\CL\CH\isoto\CL\CG$ be an isomorphism of the associated $LG$-torsors. For a geometric point $s:\Spec k\to S$, we can choose trivializations of $s^\ast\CG$ and $s^\ast\CH$. Then $\delta$ corresponds to an element $g\in LG(k)$. Changing the trivializations changes $g$ within its $K(k)$-double coset. The Cartan decomposition shows that $LG(k)$ is the disjoint union of the sets $K(k)z^\mu K(k)$ where $\mu\in X_\ast(T)$ is a dominant coweight. Obviously the double coset of $g$ does not depend on the chosen algebraic closure $k$ of $\kappa(s)$.

\begin{definition}\label{DefHodgePoly}
The dominant element $\mu_\delta(s)\in X_\ast(T)$ with $g\in K(k)z^{\mu_\delta(s)}K(k)$ is called the \emph{Hodge polygon of $\delta$ at $s$}. If $\ul\CG=(\CG,\phi)$ is a local $G$-shtuka over $S$, we write $\mu_{\ul\CG}(s):=\mu_\phi(s)$.
\end{definition}

Let $\pi_1(G)$ be the quotient of $X_\ast(T)$ by the lattice generated by the coroots. For any coweight $\mu\in X_\ast(T)$ let $[\mu]$ be its image in $\pi_1(G)$.

\begin{proposition}\label{PropHPLocConst}
Let $\CG$ and $\CH$ be $K$-torsors on $S$ and let $\delta:\CL\CH\isoto\CL\CG$ be an isomorphism of the associated $LG$-torsors. Then the function $S\to\pi_1(G),\,s\mapsto[\mu_\delta(s)]$ is locally constant on $S$.
\end{proposition}

\begin{proof}
By Proposition~\ref{PropGzTorsor} we can choose trivializations for the torsors $\CG$ and $\CH$ over an \'etale covering $f:S'\to S$. Thus $\delta$ corresponds to an element of $LG(S')$. This induces a morphism from $S'$ to the affine Grassmannian $\X$. By \cite[Theorem 0.1]{PR2} the group $\pi_1(G)$ equals the set of connected components $\pi_0(\X)$ and the induced morphism $S'\to\pi_0(\X)=\pi_1(G)$ does not depend on the chosen trivializations. Hence it descends to a morphism $S\to\pi_0(\Gr)$ which coincides with $[\mu_\delta(\,.\,)]$. Therefore $[\mu_\delta(\,.\,)]$ is locally constant.
\end{proof}

Let $\ol{B}\subset G$ be the Borel subgroup opposite to our fixed $B$. For a dominant weight $\lambda$ of $G$ we let $V(\lambda):=\bigl(\Ind_{\ol{B}}^G(-\lambda)_\dom\bigr)\dual$ be the Weyl module of $G$ with highest weight $\lambda$. It is a cyclic $G$-module generated by a $B$-stable line on which $B$ acts through $\lambda$. Any other such $G$-module is a quotient of $V(\lambda)$, see for example \cite[II.2.13]{Jantzen}. For a $K$-torsor $\CG$ on a scheme $S$ we denote by $\CG_\lambda$ the \fpqc-sheaf of $\CO_S\dbl z\dbr$-modules on $S$ associated with the presheaf
\[
\Test\;\longmapsto\;\Bigl(\CG(\Test)\times\bigl(V(\lambda)\otimes_{\BF_q}\CO_S\dbl z\dbr(\Test)\bigr)\Bigr)\big/K(\Test)\,.
\]
This means in particular that if $S'\to S$ is an \'etale covering trivializing $\CG$ (see Proposition~\ref{PropGzTorsor}) and if $\alpha:\CG_{S'}\isoto K_{S'}$ is an isomorphism of $K$-torsors with $p_2^\ast\alpha\circ (p_1^\ast\alpha)^{-1}=g\in K(S'')$ on $S''=S'\times_S S'$ (with $p_i$ the projection onto the $i$-th factor) then
\[
\CG_\lambda(\Test)\;\cong\;\bigl\{\,v'\in V(\lambda)\otimes_{\BF_q}\CO_S\dbl z\dbr(\Test\times_S S'):\es p_2^\ast v'=g\cdot p_1^\ast v'\text{ on }\Test\times_S S''\,\bigr\}\,.
\]
By Proposition~\ref{PropHilbert90} the sheaf $\CG_\lambda$ is locally free in the Zariski-topology on $S$.

Furthermore, if $\CG$ and $\CH$ are $K$-torsors on $S$ and $\delta:\CL\CH\isoto\CL\CG$ is an isomorphism of the associated $LG$-torsors then $\delta$ induces an isomorphism of sheaves of $\CO_S\dpl z\dpr$-modules 
\[
\delta:\;\CH_\lambda\otimes_{\CO_S\dbl z\dbr}\CO_S\dpl z\dpr\;\isoto\;\CG_\lambda\otimes_{\CO_S\dbl z\dbr}\CO_S\dpl z\dpr\,.
\]
Indeed, if over an \'etale covering $S'\to S$ there are trivializations $\alpha:\CG_{S'}\isoto K_{S'}$ and $\beta:\CH_{S'}\isoto K_{S'}$ with $p_2^\ast\alpha\circ (p_1^\ast\alpha)^{-1}=g\in K(S'')$ and $p_2^\ast\beta\circ (p_1^\ast\beta)^{-1}=h\in K(S'')$ and $\alpha\delta\beta^{-1}=\Delta\in LG(S')$, then $\delta$ sends $v'\in \CH_\lambda(\Test)\otimes_{\CO_S\dbl z\dbr(\Test)}\CO_S\dpl z\dpr(\Test)$ to $\Delta\cdot v'$.

\begin{definition}\label{DefBounded} Let $S$ be a connected scheme in $\NilpF$ and let $\mu$ be a dominant coweight of $G$. Let either $\tilde z=z-\zeta$ or $\tilde z=z$.
\begin{enumerate}
\item 
Let $\CG$ and $\CH$ be $K$-torsors on $S$ and let $\delta:\CL\CH\isoto\CL\CG$ be an isomorphism of the associated $LG$-torsors. The isomorphism $\delta$ is \emph{bounded by $(\mu,\tilde z)$} if for each dominant weight $\lambda$ of $G$
\begin{eqnarray} 
\delta(\CH_\lambda)&\subset&\tilde z\;^{-\langle\,(-\lambda)_\dom,\mu\rangle}\,\CG_\lambda\es\subset\es\CG_\lambda\otimes_{\CO_S\dbl z\dbr}\CO_S\dpl z\dpr\quad\text{and}\label{EqBounded1}\\
~[\mu]&=&[\mu_\delta(s)] \text{ in }\pi_1(G)\text{ for all }s\in S.\label{EqBounded2}
\end{eqnarray}
\item 
A local $G$-shtuka $(\CG,\phi)$ over $S$ is \emph{bounded by $\mu$} if the isomorphism $$\phi:\s\CL\CG\isoto\CL\CG$$ is bounded by $(\mu,z-\zeta)$.
\end{enumerate}
\end{definition}

\begin{remark}\label{RemWeylModules}
Bounds using $\tilde z=z$ are only used in some of the proofs to bound quasi-isogenies between local $G$-shtukas (see for example Definition \ref{DefQIsog} and the proof of Theorem \ref{ThmADLV=RZSp}).

Instead of using Weyl modules for the definition of boundedness one could also use a different class of representations such as for example the class of all induced representations $V(\lambda)\dual=\Ind_{\ol{B}}^G(-\lambda)_\dom$, or the class of all representations with lowest weight $(-\lambda)_\dom$, or the class of all tilting modules \cite[Chapter E]{Jantzen}. It follows from Lemma~\ref{LemmaBoundedIfQC}\ref{(b)} below that on reduced schemes $S$ in $\NilpF$ the boundedness definitions for all these classes are equivalent. However, they may differ in their nilpotent structure.
\end{remark}

\begin{lemma}\label{LemmaCritBounded}
The monoid $X^\ast(T)_+$ of dominant weights is finitely generated. Condition (\ref{EqBounded1}) holds for all dominant weights (and fixed $\tilde z=z-\zeta$ or $\tilde z=z$) if and only if it holds for a finite generating system.
\end{lemma}

\begin{proof}
Every dominant weight is a non-negative integral linear combination of the dominant weights in a generating system. So it suffices to show that if condition (\ref{EqBounded1}) holds for two dominant weights $\lambda$ and $\lambda'$ then it also holds for $\lambda+\lambda'$. Now $V(\lambda+\lambda')$ is a $G$-submodule of $V(\lambda)\otimes V(\lambda')$. Indeed, as $V(\lambda)$ and $V(\lambda')$ are Weyl modules, they trivially admit a Weyl-filtration, i.e. a filtration with Weyl-modules as factors. By \cite[II.4.19, II.4.21]{Jantzen} the same is true for $V(\lambda)\otimes V(\lambda')$. By Weyl's character formula \cite[Corollary II.5.11]{Jantzen} the character of $V(\lambda)\otimes V(\lambda')$ contains $\lambda+\lambda'$ as highest weight. Thus we can find $V(\lambda+\lambda')$ as a submodule of $V(\lambda)\otimes V(\lambda')$ by \cite[II.4.19]{Jantzen}. Consequently also $\CH_{\lambda+\lambda'}$ is a submodule of $\CH_\lambda\otimes\CH_{\lambda'}$ and the same holds for $\
 CG_{\lambda+\lambda'}$. Now the linearity of $\lambda\mapsto\langle\,(-\lambda)_\dom,\mu\rangle$ for dominant $\lambda$ proves the lemma.
\end{proof}

\begin{definition}\label{DefQIsog}
A \emph{quasi-isogeny} between local $G$-shtukas $(\CG,\phi)\to(\CG',\phi')$ over $S$ is an isomorphism of the associated $LG$-torsors $f:\CL\CG\isoto\CL\CG'$ with $\phi'\s(f)=f\phi$. The set of quasi-isogenies between $\ul\CG$ and $\ul\CG'$ over $S$ is denoted $\QIsog_S(\ul\CG,\ul\CG')$.
\end{definition}

As one might expect from the analogy with $p$-divisible groups, quasi-isogenies are rigid in the following sense.

\begin{proposition} \label{PropRigidity}
Let $\ul\CG=(\CG,\phi)$ and $\ul\CG'=(\CG',\phi')$ be two local $G$-shtukas over $S\in\NilpF$ and let $i:\ol S\hookrightarrow S$ be a closed immersion defined by a sheaf of ideals $\CI$ which is locally nilpotent. Then
\[
i^\ast:\QIsog_S(\ul\CG,\ul\CG')\isoto\QIsog_{\ol S}(i^\ast\ul\CG,i^\ast\ul\CG')\;,\es f\mapsto i^\ast f
\]
is a bijection of sets. Let now $S$ be quasi-compact and let $\ul\CG, \ul\CG'$ both be bounded. Then an element of the right hand side is bounded by $(\mu,z)$ for some $\mu$ if and only if the corresponding element over $S$ is bounded by $(\tilde\mu,z)$ for some $\tilde\mu$. 
\end{proposition}
\begin{proof}
Arguing by induction over $\CO_S/\CI^{q^n}$ it suffices to treat the case where $\CI^q=(0)$. In this case the $q$-Frobenius $\Frob_p$ factors as $S\xrightarrow{\;j}\ol S\xrightarrow{\;i}S$ where $j$ is the identity on the underlying topological space $|\ol S|=|S|$ and on the structure sheaf this factorization is given by
\begin{eqnarray*}
\TS\CO_S \es \xrightarrow{\es i^\ast\;} & \CO_{\ol S} & \TS\xrightarrow{\es j^\ast\;} \es \CO_S\\
\TS b\quad \mapsto\;\: & b \mod I& \TS\;\:\mapsto\quad b^q\,.
\end{eqnarray*}
Therefore $\s  f=j^\ast(i^\ast f)$ for any $f\in\QIsog_S(\ul\CG,\ul\CG')$. We obtain the diagram
\begin{equation}\label{EqRigidity}
\xymatrix @C=5pc @R+0.5pc {
\CL\CG \ar[r]_\cong^{\TS f} & \CL\CG' \\
\s \CL\CG \ar[u]_\cong^{\TS\phi} \ar[r]_\cong^{\TS j^\ast(i^\ast f)\es} & \s \CL\CG'\ar[u]^\cong_{\TS\phi'}
}
\end{equation}
from which the bijectivity is obvious.

For the assertion on boundedness we may again assume that $\CI^q=(0)$. One direction is obvious, thus we now consider the case that $i^\ast f$ is bounded by $\mu$. We use Lemma~\ref{LemmaCritBounded}. Let $\Lambda$ be a finite generating system of the monoid of dominant weights. The $\lambda\in\Lambda$ fall into two classes according to whether $\langle\lambda,2\rho\dual\rangle=0$ (i.e.~$\lambda$ is central) or $\langle\lambda,2\rho\dual\rangle>0$. Since $(-2\rho\dual)_\dom=2\rho\dual$, we have $\langle(-\lambda)_\dom,2\rho\dual\rangle>0$ for each $\lambda$ in the second class. Since $S$ is quasi-compact, we can find an $m\geq 0$ such that for all those $\lambda$, (\ref{EqBounded1}) is satisfied for the coweight $2m\rho\dual+\mu$. For the $\lambda$ with $\langle\lambda,2\rho\dual\rangle=0$ note that as $\lambda$ is central, $-\lambda$ is also dominant.  If $\ul\CG$ is bounded by some $\omega$, applying (\ref{EqBounded1}) to $\lambda$ and $-\lambda$ implies that $\phi$ induces
  an isomorphism
\[
\phi:\;\s\CG_\lambda \;\isoto \;\bigl({\TS\frac{1}{z-\zeta}}\bigr)^{\langle-\lambda,\omega\rangle}\CG_\lambda\,.
\]
We get a similar isomorphism for $\ul\CG'$ with some $\omega'$. Observe that $\langle\lambda,\omega\rangle=\langle\lambda,\omega'\rangle$ since $\lambda$ is central, $[\omega]=[\omega']$ in $\pi_1(G)$ by (\ref{EqBounded2}), and $f\circ \phi=\phi'\circ \s f$. Also the boundedness of $i^\ast f$ over $\ol S$ by $\mu$ yields (by applying $j^{ \ast}$) an isomorphism $$f_{\lambda}:\CG_\lambda\cong\bigl({\TS\frac{1}{z-\zeta}}\bigr)^{\langle\lambda,\omega\rangle}\s\CG_\lambda \isoto \bigl({\TS\frac{1}{z-\zeta}}\bigr)^{\langle\lambda,\omega\rangle}\bigl({\TS\frac{1}{z}}\bigr)^{\langle-\lambda,\mu\rangle}\s\CG'_\lambda\cong\bigl({\TS\frac{1}{z}}\bigr)^{\langle-\lambda,\mu\rangle}\CG'_\lambda=\bigl({\TS\frac{1}{z}}\bigr)^{\langle-\lambda,2m\rho\dual+\mu\rangle}\CG'_\lambda$$ over $S$, proving that $f$  is bounded by $2m\rho\dual+\mu$. 
\end{proof}

\begin{lemma}\label{LemmaBoundedIfQC}
Let $\CG$ and $\CH$ be $K$-torsors on $S$ for a connected scheme $S\in\NilpF$. Let either $\tilde z=z-\zeta$ or $\tilde z=z$ and let $\delta:\CL\CH\isoto\CL\CG$ be an isomorphism of the associated $LG$-torsors. 
Let $\mu$ be a dominant coweight of $G$ satisfying (\ref{EqBounded2}). 
\begin{enumerate}
\item \label{(a)}
Then the condition that $\delta$ is bounded  by $(\mu,\tilde z)$ is representable by a finitely presented closed immersion into $S$.
\item \label{(b)}
If $S$ is reduced then $\delta$ is bounded by $(\mu,\tilde z)$ if and only if this holds for the pullback to every geometric point of $S$. By \ref{(a)} it is even enough to consider the pullback to the generic points of $S$.
\end{enumerate}
\end{lemma}

\begin{proof} We prove the lemma for $\tilde z=z-\zeta$, the other case is completely analogous.
Fix a generating system $\Lambda$ of the monoid $X^\ast(T)_+$ of dominant weights and consider for each $\lambda\in\Lambda$ the isomorphism
\[
\delta:\;\CH_\lambda\otimes_{\CO_S\dbl z\dbr}\CO_S\dbl z\dbr[{\TS\frac{1}{z-\zeta}}] \;\isoto\;\CG_\lambda\otimes_{\CO_S\dbl z\dbr}\CO_S\dbl z\dbr[{\TS\frac{1}{z-\zeta}}].
\]
Since both questions are local on $S$ we can assume that $S=\Spec R$ and $\CG_\lambda$ and $\CH_\lambda$ are free $R\dbl z\dbr$-modules. Then $\delta(\CH_\lambda)$ is contained in $(z-\zeta)^{-N_\lambda}\CG_\lambda\subset\CG_\lambda\otimes_{R\dbl z\dbr}R\dbl z\dbr[\frac{1}{z-\zeta}]$ for some $N_\lambda\gg0$. So $\delta$ is bounded by $\mu$ if and only if $\delta$ maps all generators of $\CH_\lambda$ to zero in $(z-\zeta)^{-N_\lambda}\CG_\lambda/(z-\zeta)^{-\langle(-\lambda)_\dom,\mu\rangle}\CG_\lambda=:M_\lambda$ for all $\lambda\in\Lambda$ (by Lemma~\ref{LemmaCritBounded}). Let $M:=\bigoplus_{\lambda\in\Lambda} M_\lambda$. Since $M$ is a free $R$-module of finite rank this condition is represented by a finitely presented closed immersion. This proves (a).

The condition in (b) is clearly necessary. Since $M\hookrightarrow M\otimes_R\prod_{\Fp\subset R}\kappa(\Fp)^\alg$ where the product is taken over all prime ideals $\Fp$ of $R$, the condition is also sufficient.
\end{proof}

\begin{lemma}\label{LemmaHPBounded}
Let $k$ be an algebraically closed field in $\NilpF$ (hence $\zeta=0$ in $k$) and let $\CG$ and $\CH$ be $K$-torsors over $k$. Let $\delta:\CL\CH\isoto\CL\CG$ be an isomorphism of the associated $LG$-torsors with Hodge polygon $\mu_\delta(\Spec k)$ and let $\mu$ be a dominant coweight of $G$. Then $\delta$ is bounded by $(\mu,z)$ if and only if $\mu_\delta(\Spec k)\preceq\mu$.
\end{lemma}

\begin{proof}
After choosing trivializations of $\CG$ and $\CH$ the isomorphism $\delta$ is represented by left multiplication with an element $sz^{\mu'}t\in LG(k)$ for $\mu':=\mu_\delta(\Spec k)$ and $s,t\in K(k)$. Let $\lambda$ be a dominant weight of $G$ and consider the Weyl module $V(\lambda)$. By Weyl's character formula \cite[Corollary II.5.11]{Jantzen} its weights $\lambda'$ all satisfy $-(-\lambda)_\dom\preceq\lambda'\preceq\lambda$ and the weight $-(-\lambda)_\dom$ occurs. Let $v_{\lambda'}\in V(\lambda)$ be an element of the weight space of $\lambda'$. Then
\[
\delta(t^{-1}v_{\lambda'})=z^{\langle \lambda',\mu'\rangle}\cdot
s(v_{\lambda'})\;\in\;z^{\langle\lambda',\mu'\rangle}\,\CG_\lambda\,.
\]
In particular $\delta(\CH_\lambda)\subset z^{-\langle(-\lambda)_\dom,\mu'\rangle}\CG_\lambda \setminus z^{-n}\CG_\lambda$
for all $n<\langle(-\lambda)_\dom,\mu'\rangle$. So $\delta$ is bounded by $\mu$
if and only if $\mu-\mu'$ is a linear combination of coroots (by (\ref{EqBounded2})) with
$\langle(-\lambda)_\dom,\mu-\mu'\rangle\ge 0$ for every dominant $\lambda$. This
is equivalent to $\mu'\preceq\mu$.
\end{proof}

\begin{example}\label{exunbounded}
For $G=\BG_m$ we give an example of a local $G$-shtuka over $R=k[\varepsilon]/(\varepsilon^2)$ which is not bounded by any coweight $\mu$. Let $\ul\CG=(K_R,b\s)$ with $b=1+\frac{\varepsilon}{z}$. The Hodge polygon in the special point is equal to $\mu=(0)\in\mathbb{Z}$. Thus in the special point it is bounded by $\mu=(0)$ and by no other $\mu'$. But as $b\notin \mathbb{G}_m\bigl(R\dbl z\dbr\bigr)$, the boundedness condition does not hold on all of $\Spec R$.
\end{example}

On the other hand we have
\begin{lemma}
Let $\CG$ and $\CH$ be $K$-torsors on a quasi-compact connected scheme $S\in\NilpF$. Let either $\tilde z=z-\zeta$ or $\tilde z=z$ and let $\delta:\CL\CH\isoto\CL\CG$ be an isomorphism of the associated $LG$-torsors. Assume that $G$ is semi-simple or that $S$ is reduced. Then there is a dominant $\mu\in X_*(T)$ such that $\delta$ is bounded by $(\mu,\tilde z)$.
\end{lemma}
\begin{proof}
Let $\{\lambda_1,\dotsc,\lambda_n\}$ with $\lambda_i\neq 0$ be a finite generating set of the monoid of dominant weights of $G$. Let $\{\lambda_{n'},\dotsc,\lambda_n\}$ be the subset of weights that are central in $G$, that is orthogonal to the sum $2\rho\dual$ of the positive coroots. If $G$ is semi-simple, the latter set is empty. Since $S$ is quasi-compact, there are constants $c_i$ for $i=1,\dotsc,n$ with $\delta(\CH_{\lambda_i})\subset\tilde z\,^{-c_i}\CG_{\lambda_i}$. We have to show that there is a dominant $\mu$ with $[\mu]=[\mu_{\delta}(s)]$ in $\pi_1(G)$ for all $s\in S$ and $\langle(-\lambda_i)_{\dom},\mu\rangle\geq c_i$ for all $i$. As $S$ is connected, $[\mu_{\delta}(s)]\in \pi_1(G)$ is constant. Let $\mu$ be dominant with $[\mu]=[\mu_{\delta}(s)]$. Replacing $\mu$ by $\mu+2c\rho\dual$ for $c\in \mathbb{N}$ leaves the image in $\pi_1(G)$ invariant and replaces $\langle(-\lambda_i)_{\dom},\mu\rangle$ by $\langle(-\lambda_i)_{\dom},\mu\rangle +c\langle(-\lambda_i)_{\dom},2\rho\dual\rangle$. For non-central $\lambda_i$, the last bracket is strictly positive. Thus for $i=1,\dotsc, n'-1$ we can choose $c$ large enough to ensure that $\langle(-\lambda_i)_{\dom},\mu+2c\rho\dual\rangle\geq c_i$. For $G$ semi-simple this proves the lemma. Consider now the case that $S$ is reduced. We want to show that $\mu+2c\rho\dual$ also satisfies the analogous condition for $i\geq n'$. As $S$ is reduced, it is by Lemma~\ref{LemmaBoundedIfQC} enough to check the condition in each geometric point separately. Thus we have to show that $\langle(-\lambda_i)_{\dom},\mu_{\delta}(s)\rangle\leq\langle(-\lambda_i)_{\dom},\mu+2c\rho\dual\rangle$ for $i\geq n'$. As $\lambda_i$ is central, both sides are determined by the images of $\mu_{\delta}(s)$ resp. $\mu+2c\rho\dual$ in $\pi_1(G)=X_*(T)/(\text{coroot~lattice})$. As those images agree, the two sides are equal.
\end{proof}

\begin{proposition}[Soergel]\label{PropFaithfulRep}
Let $V=\bigoplus_{\lambda\in\Lambda} V(\lambda)$ be the direct sum of Weyl modules over a finite generating system $\Lambda$ of the monoid $X^\ast(T)_+$ of dominant weights for $G$. Then the natural map $\eta:G\to\GL(V)$ is a closed immersion.
\end{proposition}

\begin{proof}
We first prove injectivity for points with values in an algebraic closure of $\BF_q$. Every such point lies in a Borel subgroup 
and is thus conjugate to an element $g\in B(\BF_q^\alg)$. Assume that $g\in\ker (\eta)(\BF_q^\alg)$. Let $g=g_sg_u$ be its decomposition into semi-simple and unipotent part. For each $\lambda\in \Lambda$ consider a highest weight vector $v_\lambda\in V(\lambda)$. That is, $v_\lambda$ generates the $G$-module $V(\lambda)$, and $B$ stabilizes $\BF_q\cdot v_\lambda$ and operates through the quotient $T$ and the weight $\lambda$ on $v_\lambda$. In particular $\lambda(g_s)\cdot v_\lambda=\lambda(g)\cdot v_\lambda=\eta(g)(v_\lambda)=v_\lambda$, and hence $\lambda(g_s)=1$. Since this holds for all $\lambda$ in a generating system of $X^\ast(T)$, we must have $g_s=1$. So $\ker(\eta)$ is a closed normal subgroup consisting solely of unipotent elements. Since $G$ is reductive we conclude $\ker(\eta)(\BF_q^\alg)=(1)$.

For $\eta$ to be a closed immersion it remains to prove that the induced map on Lie algebras is injective. Since $\ker(\eta)$ is normal in $G$, the torus $T$ acts on it and the Lie algebra of $\ker(\eta)$ decomposes into weight spaces under $T$. So it suffices to show that $\eta$ is injective on $\Lie T$ and on $\Lie U_\alpha$ for each root subgroup $U_\alpha$. Our argument for injectivity on points also shows that $\eta$ is injective on $\Lie T$.

Now we consider $U_\alpha$. By conjugation with the longest element of the Weyl group it suffices to treat the case where $\alpha$ is a negative root. For this consider a dominant weight $\lambda$ (not necessarily in $\Lambda$) and the $B$-representation $W_\lambda:=(-\lambda)\otimes \Res^G_B\Ind^G_{\ol{B}}\lambda$ where 
\[
\Ind^G_{\ol{B}}\lambda \;:=\;\bigl\{\,f\in\BF_q[G]:\; f(xb)=\lambda(b)^{-1}f(x)\text{ for all }x\in G,b\in \ol{B}\,\bigr\}\,.
\]
Let $U$ be the unipotent radical of $B$. Via restriction to $U$ we get a map $W_\lambda\to \BF_q[U]$. This is an inclusion, since every $f\in W_\lambda$ which is zero on $U$ is also zero on the big cell $U\ol{B}$, hence zero in $\BF_q[G]$. For $t\in T$ we have
$(tf)(x)=f(t^{-1}x)=f(t^{-1}xt\,t^{-1})=\lambda(t)\,\Int_t^\ast(f)(x)$ with $\Int_t^\ast(f)(x):=f(t^{-1}xt)$. The inclusion $W_\lambda\subset\BF_q[U]$ is equivariant under $T$, where $T$ acts on $\BF_q[U]$ through $\Int_t^\ast$. Note that the weights of $\BF_q[U]$ are in $\BQ_{\le0}\cdot R^+:=\{\sum_{\beta\in R^+}n_\beta\cdot\beta: n_\beta\in\BQ,n_\beta\le0\}$, where $R^+$ is the set of positive roots. The inclusion is also equivariant under $U$, where $U$ operates on $\BF_q[U]$ through left translation.

We choose $\lambda$ such that $w(\lambda+\rho)-\lambda-\alpha-\rho\notin\BQ_{\ge0}\cdot R^+$ for all $w\ne1$ in the Weyl group $W$, where $\rho$ is the halfsum of all positive roots. To see that such a $\lambda$ exists, let $\alpha_1,\dotsc,\alpha_r$ be the simple roots of $G$ and choose $\mu_j\in X_*(T)_{\mathbb{Q}}$ for $j=1,\dotsc,r$ with $\langle\alpha_i,\mu_j\rangle=\delta_{ij}$. Let
\[
N=\max\bigl\{\,\langle -\alpha\,,\,\mu_j\rangle:\;j=1,\ldots,r\,\bigr\}\in\mathbb{N}_0
\]
and let $\lambda=2N\rho$. For any $1\ne w\in W$ there exists a $\mu_j$ with $\langle\rho-w(\rho),\mu_j\rangle \neq 0$, i.~e.  $\langle\rho-w(\rho),\mu_j\rangle \geq 1/2$. Then
\[
\langle w(\lambda+\rho)-\lambda-\alpha-\rho\,,\,\mu_j\rangle\;\le\;(2N+1)\,\langle w(\rho)-\rho\,,\,\mu_j\rangle + \langle -\alpha\,,\,\mu_j\rangle\;<\;0\,,
\]
and this implies that $w(\lambda+\rho)-\lambda-\alpha-\rho\notin\BQ_{\ge0}\cdot R^+$ for all $w\ne1$. Note that for this choice of $\lambda$ we have $\lambda=(-\lambda)_\dom$.

We next use Kostant's character formula \cite[Corollary 5.83]{Knapp}. Note that although it is stated in loc.\ cit.\ for Lie algebra representations in characteristic zero, the formula also holds in our situation because it is a formal consequence (see \cite[proof of Corollary 5.83]{Knapp}) of Weyl's character formula \cite[Corollary II.5.10]{Jantzen}. Let $\mu$ be a weight and
\[
\CP(\mu)\;:=\;\#\bigl\{\,(n_\beta)_{\beta\in  R^+}:\;n_\beta\in\BN_{>0},\;\mu=\sum_{\beta\in R^+}n_\beta\cdot\beta\,\bigr\}\;=\;\dim_{\BF_q}(\BF_q[U])^{-\mu}\,,\]
the dimension of the $(-\mu)$-weight space of $\BF_q[U]$. If $\mu\notin\BQ_{\ge0}\cdot R^+$ then $\CP(\mu)=0$. Kostant's formula  says that the $\alpha$-weight space of $W_\lambda$ has dimension
\begin{align*}
\dim_{\BF_q}(W_\lambda)^\alpha\es&=\es\dim_{\BF_q}\bigl(\Ind^G_{\ol{B}}\lambda\bigr)^{\lambda+\alpha}\es=\es\sum_{w\in W}(-1)^{\ell(w)}\CP(w(\lambda+\rho)-\lambda-\alpha-\rho)
\intertext{which is for our choice of $\lambda$}
&=\es\CP(\lambda+\rho-\lambda-\alpha-\rho)\es=\es\dim_{\BF_q}(\BF_q[U])^\alpha\,.
\end{align*}
In particular, $W_\lambda$ generates the $\BF_q$-algebra $\BF_q[U_\alpha]$. Assume that $u\in\Lie U_\alpha$ lies in $\ker(\eta)$. Since $V(\lambda)$ is a $G$-submodule of a tensor power of $V$ (see the argument of Lemma~\ref{LemmaCritBounded}), $u$ acts trivially on $V(\lambda)$. As $\lambda=(-\lambda)_{\dom}$, it also acts trivially on $W_\lambda$. Therefore $u$ acts trivially on $\BF_q[U_\alpha]$. Since this operation is via left translation we must have $u=0$. This proves that $\eta$ is injective on $\Lie U_\alpha$ and finishes the proof.
\end{proof}

\begin{remark}
The corresponding map $\eta':G\to\GL(V')$ where $V'$ is a sum over $\Lambda$ of corresponding \emph{irreducible} highest weight representations is in general not an immersion. This already occurs for $G=\PGL_2$ in characteristic $2$, where one can choose $\Lambda$ to consist of the only positive root $\alpha$. The corresponding Weyl module $V(\alpha)$ is the dual of the adjoint representation. With respect to the decomposition $V(\alpha)=(\Lie T\oplus\Lie U_\alpha\oplus\Lie U_{-\alpha})\dual$ it is given by 
\[
\eta:\PGL_2\to\GL(V(\alpha))\;,\;g=\left(\begin{array}{cc}a&b\\c&d\end{array}\right)\longmapsto\left(\begin{array}{ccc}1&\frac{ac}{\det g}&\frac{bd}{\det g}\\[1mm]
0&\frac{a^2}{\det g}&\frac{b^2}{\det g}\\[1mm]
0&\frac{c^2}{\det g}&\frac{d^2}{\det g}\end{array}\right)\;.
\]
It is an extension of a two-dimensional irreducible representation $V'$ by the trivial representation, and the corresponding map $\eta':\PGL_2\rightarrow \GL(V')$ is not injective on tangent spaces.
\end{remark}

For the next proposition recall that by \cite[0$_{\rm I}$, Lemma 7.7.2]{EGA}, a linearly topologized $\BFZ$-algebra $R$ is admissible if $R=\invlim R_\alpha$ for a projective system $(R_\alpha,u_{\alpha\beta})$ of discrete rings such that the filtered index-poset has a smallest element $0$, all maps $R\to R_\alpha$ are surjective, and the kernels $I_\alpha:=\ker u_{\alpha,0}\subset R_\alpha$ are nilpotent.

\begin{proposition}\label{remdejong}
Let $R$ be an admissible $\BFZ$-algebra as above with filtered index-poset $\BN_0$. Then the pullback under the natural morphism $\Spf R\to \Spec R$ defines a bijection between local $G$-shtukas bounded by $\mu$ over $\Spec R$ and over $\Spf R$. 
\end{proposition}

\begin{remark} Without the boundedness by $\mu$ the pullback map is in general only injective. The corresponding result for $p$-divisible groups is shown by Messing in \cite[Lemma II.4.16]{Messing} and by de Jong in \cite[Lemma 2.4.4]{dJ95}.
\end{remark}

\begin{proof}[Proof of Proposition \ref{remdejong}]
A local $G$-shtuka over $\Spf R$ is by definition a projective system $(\ul\CG_n)_{n\in\BN_0}$ of local $G$-shtukas $\ul\CG_n$ over $R_n$ with $\ul\CG_{n-1}\cong\ul\CG_n\otimes_{R_n}R_{n-1}$. The pullback under $\Spf R\to\Spec R$ sends a local $G$-shtuka $\ul\CG$ to the projective system $(\ul\CG\otimes_RR_n)_n$. We describe the inverse map. By Proposition~\ref{PropGzTorsor} there is an \'etale $R_0$-algebra $A_0$ which trivializes $\CG_0$. By \cite[Th\'eor\`eme I.5.5]{SGA1} there is a unique \'etale $R$-algebra $A$ with $A\otimes_RR_0\cong A_0$. Set $A_n:=A\otimes_RR_n$. By Proposition~\ref{PropGzTorsor} we can simultaneously trivialize all $\CG_n$ over $A_n$. Thus $\ul\CG_n\otimes_{R_n}A_n\cong(K_{A_n},b_n\s)$ for $b_n\in LG(A_n)\,=\,G\bigl(A_n\dbl z\dbr[\frac{1}{z-\zeta}]\bigr)$ with $b_n\otimes_{A_n}A_{n-1}=b_{n-1}$. We must show that $b:=\invlim b_n$ exists as an element of $G\bigl(A\dbl z\dbr[\frac{1}{z-\zeta}]\bigr)$. Note that without the boundedness by $\mu$ this is in general false.

Consider the representation $\eta:G\to\GL(V)$ with $V$ being the $\mathbb{F}_q$-vector space $\bigoplus_{\lambda\in \Lambda}V(\lambda)$ from Proposition~\ref{PropFaithfulRep}. Let $N$ be a positive integer with $N\ge \langle(-\lambda)_\dom,\mu\rangle$ for all $\lambda\in\Lambda$. Then the boundedness by $\mu$ implies that $(z-\zeta)^N\cdot\eta(b_n)\in M_r\bigl(A_n\dbl z\dbr\bigr)$, that is the denominator of $b_n$ is bounded by $(z-\zeta)^N$ independently of $n$. Clearly the projective system $\bigl((z-\zeta)^N\cdot\eta(b_n)\bigr)_n$ has a limit in $M_r\bigl(A\dbl z\dbr\bigr)$. It is of the form $(z-\zeta)^N\cdot\eta(b)$ for a $b\in G\bigl(A\dbl z\dbr[\frac{1}{z-\zeta}]\bigr)$ by construction. The local $G$-shtuka $\bigl(K_{\Spec A},b\s\bigr)$ over $\Spec A$ inherits a descent datum from the $\ul\CG_n$. This gives us the desired local $G$-shtuka over $\Spec R$.
\end{proof}

%%%%%%%%%%%%%%%%%%%%%%%%%%%%%%%%%%%%%%%%%%%%%%%%%%%%%%%%%%%%%%%%%%%%%%
%
%   The general linear group
%
%%%%%%%%%%%%%%%%%%%%%%%%%%%%%%%%%%%%%%%%%%%%%%%%%%%%%%%%%%%%%%%%%%%%%%

\section{The general linear group} \label{SectG=GLr}
\setcounter{equation}{0}

In this section we consider the case $G=\GL_r$ and give the translation between $\GL_r$-torsors and locally free sheaves of finite rank. This special case has the advantage that the similarity to the theory of crystals is visible more clearly. 

Let $B\subset\GL_r$ be the Borel subgroup of upper triangular matrices and let $T$ be the torus of diagonal matrices. Then $X_*(T)\cong \mathbb{Z}^r$ with simple coroots $e_{i}-e_{i+1}$ for $i=1,\dotsc,r-1$. Also $X^*(T)\cong \mathbb{Z}^r$. Let $\lambda_{i}=(1,\dotsc,1,0,\dotsc,0)$ with multiplicities $i$ and $r-i$. The Weyl module $V(\lambda_1)=\bigl(\Ind_{\ol{B}}^{\GL_r}(-\lambda_1)_\dom\bigr)\dual$ of highest weight $\lambda_1$ is simply the standard representation of $\GL_r$ on the space of column vectors with $r$ rows, and $V(\lambda_i)=\wedge^i V(\lambda_1)$. For an $\BF_q$-scheme $S$ we have $K(S)=\GL_r\bigl(\Gamma(S,\CO_S)\dbl z\dbr\bigr)$. By Proposition~\ref{PropHilbert90} there is an equivalence between the category of $K$-torsors on $S$ and the category of sheaves of $\CO_S\dbl z\dbr$-modules which Zariski-locally on $S$ are free of rank $r$ with isomorphisms as the only morphisms. This equivalence sends $\CG$ to the sheaf $\CG_{\lambda_1}$ associated with the presheaf
\[
\Test\;\longmapsto\;\Bigl(\CG(\Test)\times\bigl(V(\lambda_1)\otimes_{\BF_q}\CO_S\dbl z\dbr(\Test)\bigr)\Bigr)\big/\GL_r\bigl(\Test\dbl z\dbr\bigr)\,;
\]
compare the discussion before Definition~\ref{DefBounded}.
On the sheaf $\CG_{\lambda_1}$ the group $\GL_r$ acts via the standard representation (which corresponds to $\lambda_1$ as above).
  
\begin{definition}\label{DefG.1}
A \emph{local shtuka} over $S\in\NilpF$ (of rank $r$) is a pair $(M,\phi)$ where $M$ is a sheaf of $\CO_S\dbl z\dbr$-modules on $S$ which Zariski-locally is free of rank $r$, together with an isomorphism of $\CO_S\dpl z\dpr$-modules $\phi:\s  M\otimes_{\CO_S\dbl z\dbr}\CO_S\dpl z\dpr\isoto M\otimes_{\CO_S\dbl z\dbr}\CO_S\dpl z\dpr$.\\
A \emph{quasi-isogeny} between local shtukas $(M,\phi)\to(M',\phi')$ is an isomorphism of $\CO_S\dpl z\dpr$-modules $f:M\otimes_{\CO_S\dbl z\dbr}\CO_S\dpl z\dpr\isoto M'\otimes_{\CO_S\dbl z\dbr}\CO_S\dpl z\dpr$ with $\phi'\s (f)=f\phi$.
\end{definition}

\begin{lemma}\label{LemmaG.2}
Under the functor $\CG\mapsto\CG_{\lambda_1}=:M$ the category of local $\GL_r$-shtukas over $S$ with quasi-isogenies as morphisms is equivalent to the category of local shtukas over $S$ with quasi-isogenies as morphisms.\qed
\end{lemma}

Let $\ulM=(M,\phi)$ be a local shtuka over $S\in\NilpF$ and let $s:\Spec L\to S$ be a point with $L$ a field. By the elementary divisor theorem (i.e.\ the Cartan decomposition for $G=\GL_r$) there are $L\dbl z\dbr$-bases $m_1,\ldots,m_r$ of $M_s$ and $n_1,\ldots,n_r$ of $\s  M_s$ with $\phi(n_i)=z^{-e_i}\!\cdot\! m_i$ and $e_i=e_i(\ulM;s)$. The elementary divisors $e_1\ge\ldots\ge e_r$ are called the \emph{Hodge weights} of $\ulM_s$. The decreasing ordering $e_1\ge\ldots\ge e_r$ corresponds to the choice that the Borel subgroup $B$ is the group of upper triangular matrices and that we want $(e_1,\ldots,e_r)$ to be dominant. The vector $\mu=(e_1,\dotsc,e_r)$ is also called the \emph{Hodge polygon} of $\ulM_s$. This name comes from the fact that one often considers the polygon associated to $\mu$ which is the graph of the piecewise linear continuous function $[0,r]\rightarrow \mathbb{R}$ with $0\mapsto 0$ and slope $e_i$ on $[i-1,i]$. Note that we chose the opposite ordering $e_1\ge\ldots\ge e_r$ than Katz~\cite{Katz} does and therefore the shape of our polygons is opposite to the one of Katz. 

The second coordinate $e_1+\ldots+e_r$ of the endpoint of the polygon associated to $\mu$ equals the valuation $\ord_z(\det\phi)$ of the determinant of $\phi$ with respect to any $L\dbl z\dbr$-bases of $M_s$ and $\s  M_s$. One easily checks that it is a locally constant function on $S$. In view of $\pi_1(\GL_r)\cong\BZ$ and $[\mu_\ulM(s)]=\ord_z(\det s^\ast\phi)$ this gives a simple proof of Proposition~\ref{PropHPLocConst} for $G=\GL_r$. 

The ordering on $X_*(T)$ can also be visualized using the associated polygons: $\mu'\preceq\mu$ if and only if the polygon associated to $\mu'$ lies below the polygon associated to $\mu$ and both polygons have the same endpoint. Writing $\mu'=(d_i)$ and $\mu=(e_i)$ this is equivalent to $e_1+\dotsc+e_i\geq d_1+\ldots+d_i$ for all $1\le i\le r$ with equality for $i=r$.

\begin{lemma}\label{LemmaG.3}
Let $d_1\ge\ldots\ge d_r$ be integers. Then a local $\GL_r$-shtuka $\ul\CG$ is bounded by $\mu=(d_1,\ldots,d_r)\in\BZ^r=X_\ast(T)$ if and only if its associated local shtuka $(M,\phi)$ satisfies
\[
\phi\bigl(\wedge^i\s  M\bigr)\;\subset\;(z-\zeta)^{d_{r-i+1}+\ldots+d_r}\wedge^i M\qquad\text{for }1\le i\le r\text{ with equality for } i=r\,. 
\]
In this case $\coker\bigl(\phi:\s M\to (z-\zeta)^{d_r}M\bigr)$ is a locally free sheaf of $\CO_S$-modules of finite rank on $S$.
\end{lemma}

\begin{proof}
To prove the first assertion we use Lemma~\ref{LemmaCritBounded}. The dominant weights $\lambda_i=(1,\ldots,1,0,\ldots,0)$ with $1$ repeated $i$ times for $1\le i\le r$, together with $-\lambda_r$ generate the monoid $X^\ast(T)_+$ of dominant weights of $\GL_r$. The Weyl module $V(\lambda_i)$ associated with $\lambda_i$ is the $i$-th exterior power of the standard representation $V(\lambda_1)$. The Weyl module $V(\lambda_r)$ is one dimensional and therefore equal to $V(-\lambda_r)\dual$; cf.\ \cite[Remark 3, p.\ 177]{Jantzen}. This is responsible for the equality for $i=r$.

To prove that $\coker\phi$ is locally free note that $\phi:\s  M\to (z-\zeta)^{d_r}M$ is injective. Let $s:\Spec\kappa(s)\to S$ be a point of $S$ and consider the sequence
\[
0\;\longto\;\Tor_1^{\CO_S}(\kappa(s),\coker\phi)\;\longto\;s^\ast\s  M\;\xrightarrow{\es s^\ast\phi\;}\;s^\ast (z-\zeta)^{d_r}M\;\longto\;s^\ast\coker\phi\;\longto\;0\,.
\]
As $M$ is locally free we have $s^\ast\s M\cong \kappa(s)\dbl z\dbr^{\oplus r}\cong s^\ast (z-\zeta)^{d_r}M$. The boundedness condition says for $i=r$ that $\det(s^\ast\phi)$ equals $z^{d_1+\ldots+d_r}$ times a unit. Therefore the map $s^\ast\phi$ is injective. So $\coker\phi$ is finitely presented by construction and flat over $S$ by Nakayama's Lemma; e.g.\ \cite[Exercise~6.2]{Eisenbud}, hence locally free of finite rank.
\end{proof}

\begin{remark} Over a field $S=s=\Spec k$ the elementary divisor theorem (or Lemma~\ref{LemmaHPBounded}) tells us that $\ulM$ is bounded by $\mu=(d_i)$ if and only if $e_1(\ulM;s)+\ldots+e_i(\ulM;s)\le d_1+\ldots+d_i$ for all $1\le i\le r$ with equality for $i=r$, that is the Hodge polygon $\mu_\ulM(s):=\bigl(e_i(\ulM;s)\bigr)$ satisfies $\mu_\ulM(s)\preceq\mu$.
\end{remark}

\begin{example}\label{exend}
Consider $\mu=(d,0,\ldots,0)$ for some $d\geq 0$. Then a local $\GL_r$-shtuka $\ul\CG$ over $S$ is bounded by $\mu$ if and only if its associated local shtuka $(M,\phi)$ satisfies that $\phi:\s  M\to M$ is a true morphism of $\CO_S\dbl z\dbr$-modules and $\coker\phi$ is locally free of rank $d$ on $S$.
\end{example}

%%%%%%%%%%%%%%%%%%%%%%%%%%%%%%%%%%%%%%%%%%%%%%%%%%%%%%%%%%%%%%%%%%%%%%
%
%   Deformation Theory of Local $G$-Shtukas
%
%%%%%%%%%%%%%%%%%%%%%%%%%%%%%%%%%%%%%%%%%%%%%%%%%%%%%%%%%%%%%%%%%%%%%%

\section{Deformation theory of local $G$-shtukas} \label{SectDeformTh}
\setcounter{equation}{0}

Before we treat deformations let us return to the affine Grassmannian $\X$. Consider the formal scheme $\Spf \BFZ$ as the ind-scheme $\dirlim\Spec \mathbb{F}_q[\zeta]/(\zeta^{n})$ and let $\wh \X$ be the fiber product $\X\times_{\Spec \mathbb{F}_q}\Spf \BFZ$ in the category of ind-schemes (\cite[7.11.1]{BeilinsonDrinfeld}). Thus $\wh \X$ is the restriction of the sheaf $\X$ to the $\fppf$-site of schemes in $\NilpF$. We can also think of $\wh \X$ as the formal completion of $\X\times_{\Spec \mathbb{F}_q}\Spec \BFZ$ along the special fiber $\Var(\zeta)\subset \X\times_{\Spec \mathbb{F}_q}\Spec \BFZ$. 

The following proposition is a variant of \cite[Proposition 3.10]{LaszloSorger}. Namely loc. cit. treats the case where $\ul\CG$ is a $G$-torsor over $X\times_{\BF_q}S$ for a smooth projective connected curve, and $\delta$ is a trivialization outside a single point $p\in X$. Our proposition is the infinitesimal variant and the argument of loc. cit. carries over literally. We will give a detailed proof of this fact in a slightly modified situation in Theorem~\ref{ThmModuliSpX}.

\begin{proposition} \label{Prop4.1}
The ind-scheme $\wh \X$ pro-represents the functor $\bigl(\NilpF\bigr)^o \longto \Sets$
\begin{eqnarray*}
S&\longmapsto & \Bigl\{\,\text{Isomorphism classes of pairs }(\CG,\delta)\text{ where} \\
&&\qquad \CG \text{ is a $K$-torsor on $S$ and} \\
&&\qquad \delta:\CL\CG\isoto LG_S\text{ is an isomorphism of the associated $LG$-torsors}\,\Bigr\}\,.
\end{eqnarray*}
Here $(\CG,\delta)$ and $(\CG',\delta)$ are isomorphic if $\delta^{-1}\circ\delta'$ is an isomorphism $\CG'\to\CG$.
\end{proposition}

\begin{definition} 
We write $\wh \X=\dirlim X_n$ for some quasi-compact schemes $X_n\in\NilpF$. Let either $\tilde z=z-\zeta$ or $\tilde z=z$. Then the condition that the inverse $\gamma:=\delta^{-1}$ of the universal isomorphism $\delta:\CL\CG_{X_n}\isoto LG_{X_n}$ from Proposition~\ref{Prop4.1} is bounded by $(\mu,\tilde z)$ is represented by a closed subscheme of $X_n$ by Lemma~\ref{LemmaBoundedIfQC}. The inductive limit of these closed subschemes defines a closed ind-subscheme of $\wh \X$, which we call $\wh\Gr^{\preceq(\mu,\tilde z)}$. The fibers over $\zeta=0$ of $\wh\Gr^{\preceq(\mu,z-\zeta)}$ and $\wh\Gr^{\preceq(\mu,z)}$ coincide as ind-subschemes of $\Gr$.
\end{definition}

To describe $\wh\Gr^{\preceq(\mu,z-\zeta)}$ and $\wh\Gr^{\preceq(\mu,z)}$ recall that the Cartan decomposition defines a stratification of $\X$ by the Schubert cells $\Gr^\mu:=Kz^{(-\mu)_\dom}K/K$ for dominant coweights $\mu$. Note that our unusual definition of $\Gr^\mu$ is motivated by Proposition~\ref{PropProSchubertCell} and Remark \ref{remavoid} below. The following result is well-known.

\begin{proposition}\label{PropSchubertCell}
\begin{enumerate}
\item 
The Schubert cell $\Gr^\mu$ is a locally closed subscheme of $\X$ which is a quasi-projective scheme over $\BF_q$.
\item 
The closure of $\Gr^\mu$ in $\X$ equals the union $\Gr^{\preceq\mu}:=\bigcup_{\mu'\preceq\mu}\Gr^{\mu'}$. The latter is a projective scheme over $\BF_q$ and $\Gr^\mu$ is open in $\Gr^{\preceq\mu}$.
\item 
Both $\Gr^\mu$ and $\Gr^{\preceq\mu}$ are irreducible and of dimension $\langle 2\rho,\mu\rangle$ where $2\rho$ is the sum of all positive roots of $G$.
\end{enumerate}
\end{proposition}

\begin{proof}
Note that $(-\mu)_{\dom}=w_0(-\mu)$ where $w_0$ is the longest element of the Weyl group. Thus $\mu'\preceq\mu$ if and only if $(-\mu')_{\dom}\preceq (-\mu)_{\dom}$. Besides, $\langle 2\rho,(-\mu)_{\dom}\rangle=\langle 2\rho,\mu\rangle$. Now the assertion follows from \cite[Lemma 2.2 and the remarks thereafter]{Ngo-Polo} or \cite[4.5.8, 4.5.12]{BeilinsonDrinfeld}, except for the projectivity of $\Gr^{\preceq\mu}$ which we now prove. Since $K$ is an affine scheme, it and $\Gr^\mu$ are quasi-compact. Since there are only finitely many $\mu'\preceq\mu$, also $\Gr^{\preceq\mu}$ is quasi-compact and then its projectivity follows from the facts that $\X$ is ind-projective (that is $\X=\dirlim X_n$ for projective $\BF_q$-schemes $X_n$; see \cite[Proof of 4.5.1(iv) and 7.11.2(iii)]{BeilinsonDrinfeld}), that $\Gr^{\preceq\mu}$ is closed in $\X$, and the following Lemma~\ref{LemmaQCFactors}. Note that strictly speaking the results in \cite[\S4.5]{BeilinsonDrinfeld} are proved only over the base field $\BC$ but their proofs also work over $\BF_q$.
\end{proof}

\begin{lemma}\label{LemmaQCFactors} Let $X$ be an ind-scheme over a base scheme $S$, cf.\ \cite[7.11.1]{BeilinsonDrinfeld}, that is $X$ is a sheaf of sets for the \fppf-topology on $S$, which can be represented as the limit $\dirlim X_\alpha$ of an inductive system of quasi-compact $S$-schemes $X_\alpha$ and closed immersions $X_\alpha\hookrightarrow X_\beta$ for $\alpha\le\beta$ in a directed set. If $U$ is a quasi-compact $S$-scheme, then any $S$-morphism $f:U\to X$, that is section $f\in X(U)$, factors through some $X_{\alpha_0}$.
\end{lemma}
Note that since the presheaf $U\mapsto \dirlim X_\alpha(U)$ is not a sheaf, an argument for this is required. 
\begin{proof}
Since $X$ is the sheaf associated with the presheaf $U\mapsto \dirlim X_\alpha(U)$, the morphism $f$ is given by a covering $U'\to U$ for the $\fppf$-topology and an element $f'\in\dirlim X_\alpha(U')$. This means that there is an $\fppf$-morphism $U''\to U$ and a Zariski covering $U''=\bigcup_i U''_i$ such that $U'=\coprod_i U''_i$. In particular $f'\in\prod_i\dirlim X_\alpha(U''_i)$ and for each $i$ there is an $\alpha_i$ with $f'|_{U''_i}\in X_{\alpha_i}(U''_i)$. Since $U$ is quasi-compact, already finitely many of the $U''_i$ form an $\fppf$-covering of $U$. Replacing $U'$ by their union we can assume that $U'$ is also quasi-compact. Then taking $\alpha_0$ as an upper bound of the finitely many $\alpha_i$ involved yields $f'\in X_{\alpha_0}(U')$. By $\fppf$-descent \cite[\S6.1, Theorem 6a]{BLR} the morphism $f$ factors through $X_{\alpha_0}$.
\end{proof}

\begin{proposition}\label{PropProSchubertCell}
Let either $\tilde z=z-\zeta$ or $\tilde z=z$. Then the ind-scheme $\wh\Gr^{\preceq(\mu,\tilde z)}$ is a $\zeta$-adic noetherian formal scheme over $\BFZ$ whose underlying topological space is $\Gr^{\preceq\mu}$.
\end{proposition}

\begin{proof}
We claim that $\wh\Gr^{\preceq(\mu,\tilde z)}\times_{\Spf \BFZ}\Spec \BFZ/(\zeta^n)=:Y_n$ is a scheme locally of finite type over $\BFZ/(\zeta^n)$ with underlying topological space $\Gr^{\preceq\mu}$. From the claim the proposition follows by \cite[I$_{\rm new}$, Corollary 10.6.4]{EGA}.

To prove the claim let us first recall the structure of $\X$ as an ind-scheme. Choosing an embedding $G\subset\GL_r$ induces an injection of $\X$ into the affine Grassmannian for $\GL_r$ which we denote by $\wt \X$. The ind-scheme structure on $\wt \X$ can be described as follows. The valuation $\ord_z\det(A)$ for $A\in\GL_r\bigl(k\dpl z\dpr\bigr)$ yields a locally constant function on $\wt \X$; see Proposition~\ref{PropHPLocConst}. So $\wt \X$ is the disjoint union of the $\wt \X(h):=\{\,A\in\wt \X:\ord_z\det(A)=h\,\}$. Let $N$ be a positive integer and let $\wt X(h)^N$ be the reduced subscheme of $\wt\X(h)$ defined by
\[
\wt X(h)^N(k)\;:=\;\{\,A\in \wt \X(h)(k): A\in M_r\bigl(z^{-N}{k}\dbl z\dbr\bigr)/\GL_r\bigl({k}\dbl z\dbr\bigr)\,\}\,.
\]
Then the $\wt X(h)^N$ are projective ${\BF_q}$-schemes, $\wt X^N:=\coprod_h\wt X(h)^N$ is a scheme locally of finite type over ${\BF_q}$, and $\wt \X=\dirlim \wt X^N$. The inclusion $\X\hookrightarrow \wt \X$ realizes $\X$ as a closed ind-subscheme and the induced ind-scheme structure of $\X$ does not depend on the chosen embedding; see \cite[Theorem 4.5.1]{BeilinsonDrinfeld} in characteristic zero, but note that the proof also works over $\BF_q$.

We consider the $G$-module $V$ and the induced embedding $\eta:G\rightarrow \GL(V\dual)$ of Proposition~\ref{PropFaithfulRep}. Let $N=N(n,\mu)\in\BN$ be a power of $q$ such that $N\ge n$ and $N\ge\langle(-\lambda)_\dom,\mu\rangle$ for all $\lambda\in \Lambda$. On $Y_n$ we have $(z-\zeta)^N=z^N-\zeta^N=z^N$. By our definition of $\wh\Gr^{\preceq(\mu,\tilde z)}$, the universal isomorphism $\gamma:=\delta^{-1}:LG_{Y_n}\isoto\CL\CG|_{Y_n}$ on $Y_n$ is bounded by $(\mu,\tilde z)$, hence satisfies 
\[
\gamma\bigl(K_{Y_n}\bigr)_\lambda\;=\;\gamma\bigl(V(\lambda)\otimes_{\BF_q}\CO_{Y_n}\dbl z\dbr\bigr)\;\subset\;\tilde z\;^{-\langle(-\lambda)_\dom,\mu\rangle}\,\CG_\lambda\;\subset\;  z^{-N}\CG_\lambda
\]
for all $\lambda\in\Lambda$ and both choices of $\tilde z$. By Proposition~\ref{PropHilbert90} we may choose bases of the $\CG_\lambda$ Zariski-locally on $Y_n$. Then $\gamma$ is represented on $V=\bigoplus_{\lambda\in\Lambda}V(\lambda)$ as a matrix $A$ in the right coset
\[
A\in\GL_{\dim V}\bigl(Y_n\dbl z\dbr\bigr)\backslash M_{\dim V}\bigl(z^{-N}Y_n\dbl z\dbr\bigr)\,.
\]
On $V\dual=\bigl(\bigoplus_{\lambda\in\Lambda}V(\lambda)\bigr)\dual$ its inverse $\delta=\gamma^{-1}$ is represented by its transpose
\[
A^T\in M_{\dim V}\bigl(z^{-N}Y_n\dbl z\dbr\bigr)/\GL_{\dim V}\bigl(Y_n\dbl z\dbr\bigr)\,.
\]
Therefore $Y_n$ is a closed subscheme of $\wt X^N\times_{\Spec {\BF_q}} \Spec \BFZ/(\zeta^n)$, hence locally of finite type over $\BFZ/(\zeta^n)$. The reduced closed subscheme underlying $Y_n$ can be computed by looking at geometric points $x\in \X$. Over $x$ the isomorphism $\gamma_x=\delta_x^{-1}$ is given by an element $g\in K\bigl(\kappa(x)\bigr)z^{\mu_{\gamma}(x)}K\bigl(\kappa(x)\bigr)$ after choosing a trivialization of $\CG_x$. By definition of $\wh\Gr$ this means that the point $x$ is given by $g^{-1}\in LG(\kappa(x))$ with $g^{-1}\in\Gr^{\mu_\gamma(x)}$. By Lemma~\ref{LemmaHPBounded} the point $x$ belongs to $Y_n$ if and only if $\mu_{\gamma}(x)\preceq\mu$, which is equivalent to $x$ lying in $\Gr^{\preceq\mu}$.
\end{proof}

Let $\mu\in X_*(T)$ be dominant and let $\ul\BG=(\BG,\phi_{_\BG})$ be a local $G$-shtuka bounded by $\mu$ over a field $k'\in\NilpF$. Assume that there is a trivialization $\alpha:\BG\isoto K_{k'}$ such that $\alpha:\ul\BG\isoto(K_{k'},b_0\s)$ for some $b_0\in LG(k')$. Note that if $k'$ is algebraically closed, a trivialization always exists by Remark~\ref{RemEtaleTrivial}. The inverse $b_0^{-1}$ of $b_0$ defines a point $x\in \X(k')$ (which depends on the chosen trivialization). Since $\ul\BG$ is bounded by $\mu$, $(b_0^{-1})^{-1}=b_0$ is bounded by $\mu$ and we have $x\in\Gr^{\preceq\mu}(k')$ by Proposition~\ref{PropProSchubertCell}. Let $\Defo$ be the complete local ring of $\wh\Gr^{\preceq(\mu,z-\zeta)}$ at the point $x$. It is a complete noetherian local ring over $k'\dbl\zeta\dbr$. 

\begin{theorem}\label{ThmDefoSp}
$\Defo$ pro-represents the formal deformation functor of $\ul\BG$  
\begin{eqnarray*}
&F:&\bigl(\text{Artinian local $k'\dbl\zeta\dbr$-algebras with residue field
  $k'$}\bigr)\longto \Sets\,,\\[2mm]
A&\longmapsto& \Bigl\{\;\text{Isomorphism classes of pairs }(\ul\CG,\beta)\text{ where}\\
& & \qquad \ul\CG\text{ is a local $G$-shtuka over $\Spec A$ bounded by $\mu$ and}\\
& & \qquad \beta:\ul\BG\isoto\ul\CG\otimes_A k' \text{ is an isomorphism of local $G$-shtukas}\;\Bigr\}
\end{eqnarray*}
where $(\ul\CG,\beta)$ and $(\ul\CG',\beta')$ are isomorphic if there exists an isomorphism $\eta:\ul\CG\to\ul\CG'$ with $\beta'=(\eta\otimes_A k')\circ\beta$.
\end{theorem}

\begin{proof}
We first construct the universal pro-object over $\Defo$. Let $\Defo_n=\Defo/\Fm_\Defo^n$ where $\Fm_\Defo$ is the maximal ideal of $\Defo$. By \cite[Lemma 2.1]{Ngo-Polo} the projection morphism $LG\to \Gr$ has a section in an open neighborhood of the image of $1\in LG(k')$. Translating by $b_0^{-1}\in LG(k')$ we obtain a section in an open neighborhood of $x\in\Gr(k')$. Since $\Fm_\Defo$ is nilpotent on $\Defo_n$, the natural map $\Spec \Defo_n\to \wh\Gr^{\preceq(\mu,z-\zeta)}\to \wh \X$ can be given by some element $\tilde b_n^{-1}\in LG(\Defo_n)=G\bigl(\Defo_n\dbl z\dbr[\frac{1}{z-\zeta}]\bigr)$ with $\tilde b_n^{-1}\otimes_{\Defo_n} k'=b_0^{-1}$. Note that $\tilde b_n^{-1}$ is only determined up to multiplication on the right with elements $g\in K(\Defo_n)$ satisfying $g\otimes_{\Defo_n} k'=1$. By construction $\gamma_n=(\tilde b_n^{-1})^{-1}=\tilde b_n$ is bounded by $(\mu,z-\zeta)$. We choose the $\tilde b_n$ in a compatible way and let $\tilde b\in G\bigl(\Defo\dbl z\dbr[\frac{1}{z-\zeta}]\bigr)$ be the limit. The boundedness by $\mu$ of all $\tilde b_n$ implies the existence and boundedness by $\mu$ of $\tilde b$ by the same argument as in Proposition~\ref{remdejong}. Although $\tilde b^{-1}$ is still only determined up to right multiplication with elements $g\in K(\Defo)$ as above, we take $\ul\CG^{\rm univ}:=(K_{\Defo},\tilde b\s)$ and $\alpha:\ul\BG\isoto\ul\CG^{\rm univ}\otimes_\Defo k'$ so that $(\ul\CG^{\rm univ},\alpha)\in F(\Defo)$ because $\tilde b$ and $\tilde b\s$ are bounded by $\mu$. In particular we obtain a morphism of functors $\Hom_{k'\dbl\zeta\dbr}(\Defo,\,.\,)\to F$ by sending $u:\Defo\to A$ to the deformation $\bigl(K_A,u(\tilde b)\s,\alpha\bigr)\in F(A)$.

We will show that this morphism of functors is an isomorphism. Let $(\ul\CG,\beta)\in F(A)$. Let $\Fm\subset A$ be the maximal ideal. We use induction on $A_n:=A/\Fm^{q^n}$ to show the following

\smallskip
\noindent
{\it Claim.} For each $n$ there is a uniquely determined $k'\dbl\zeta\dbr$-homomorphism $u_n:\Defo\to A_n$ with
\[
\bigl(K_{A_n},u_n(\tilde b)\s,\alpha\bigr)\;=\;(\ul\CG,\beta)\otimes_A A_n\es\text{ in }F(A_n)\,.
\]

For $n=0$ we have $A_0=k'$ and there is only one $k'\dbl\zeta\dbr$-homomorphism, namely $u_0:\Defo\to \Defo/\Fm_\Defo=k'=A_0$. This proves the claim for $n=0$ since $u_0(\tilde b)=b_0$ and $\bigl(K_{k'},b_0\s,\alpha\bigr)$ is isomorphic to $(\ul\CG,\beta)\otimes_A k'$ via $\beta\circ\alpha^{-1}$.

Now let $n\ge 1$ and assume that we have already constructed the uniquely determined $u_{n-1}$. In particular there is an isomorphism $\eta:\bigl(K_{A_{n-1}},u_{n-1}(\tilde b)\s\bigr)\isoto\ul\CG\otimes_A A_{n-1}$ with $(\eta\otimes_{A_{n-1}}k')\circ\alpha=\beta$. By Proposition~\ref{PropGzTorsor} the trivialization $\eta$ lifts to a trivialization \mbox{$\eta':K_{A_n}\isoto\CG\otimes_A A_n$}. We let $b_n:=(\eta')^{-1}\,\phi_\CG\,\s(\eta')\in G\bigl(A_n\dbl z\dbr[\frac{1}{z-\zeta}]\bigr)$ and replace $(\ul\CG,\beta)\otimes_A A_n$ by $\bigl(K_{A_n},b_n\s,\alpha\bigr)$. Then $b_n\otimes_{A_n}A_{n-1}=u_{n-1}(\tilde b)$. Consider the diagram
\[
\xymatrix @C+3pc {
\Spec A_n\ar@{-->}[r]^{\TS b_n^{-1}} \ar@{-->}[dr]^{\TS\Spec u_n} & LG \ar[r]& \wh \X\\
\Spec A_{n-1}\ar[u]\ar[r]_{\TS\Spec u_{n-1}} & \Spec \Defo\ar[u]_{\TS\tilde b^{-1}}\ar[r] & \wh\Gr^{\preceq(\mu,z-\zeta)}\ar[u]
}
\]
Here the right square commutes by the definition of $\tilde b$. 
The element $b_n^{-1}$ yields a morphism $\Spec A_n\to\wh \X$ which factors through $\wh\Gr^{\preceq(\mu,z-\zeta)}$, since $\ul\CG$ and hence $(b_n^{-1})^{-1}$ are bounded by $\mu$. As $b_n\otimes_{A_n}k'=b_0$, the maximal ideal of $A_n$ is sent to $x\in\wh\Gr^{\preceq(\mu,z-\zeta)}$. Since $A_n$ is complete, the morphism even factors through $\Spec \Defo$. We obtain a homomorphism $u_n:\Defo\to A_n$ with $u_n(\tilde b)\otimes_{A_n}A_{n-1}=u_{n-1}(\tilde b)$ and such that $g:=b_nu_n(\tilde b^{-1})\in K(A_n)$ by definition of $\wh \X$. In particular $g\otimes_{A_n}A_{n-1}=1$, hence $\s(g)=1$. We now replace $\bigl(K_{A_n},b_n\s,\alpha\bigr)$ by $\bigl(K_{A_n},g^{-1}b_n\s(g)\cdot\s,\alpha\bigr)$, that is, we change $\eta'$ to $\eta'g$. Then $u_n(\tilde b)=g^{-1}b_n=g^{-1}b_n\s(g)$ as desired.

It remains to show that $u_n$ is uniquely determined. Let $u_n,u'_n:\Defo\to A_n$ be two homomorphisms as in our claim. Applying $\otimes_{A_n}A_{n-1}$ and exploiting the uniqueness for $n-1$ we see that $u_n(\tilde b)\otimes_{A_n}A_{n-1}=u'_n(\tilde b)\otimes_{A_n}A_{n-1}$. So there is a $g\in K(A_n)$ with $g^{-1}u_n(\tilde b)\s(g)=u'_n(\tilde b)$ and $g\otimes_{A_n}A_{n-1}=1$. In particular $\s(g)=1$ and hence $u'_n(\tilde b^{-1})=u_n(\tilde b^{-1})g$. This implies that the compositions of $u_n$ and $u'_n$ with $\Spec \Defo\to \wh \X$ coincide. By the definition of $\Defo$ this implies $u_n=u'_n$ as desired.
\end{proof}

\begin{remark}
  From now on we always consider the deformation ring $\Defo$ with a fixed isomorphism of functors $\Hom_{k'\dbl\zeta\dbr}(\Defo,\,.\,)\cong F$ (and thus a fixed  universal local $G$-shtuka over $\Defo$) as in the proof.
\end{remark}

\begin{remark}\label{remavoid}
One could avoid the inversion $b_0^{-1}$ and the definition of $\Gr^\mu$ as $Kz^{(-\mu)^\dom}K/K$ by instead considering the affine Grassmannian $(K\backslash LG)\times_{\Spec\BF_q}\Spf\BFZ$ and defining $\wh{\Gr}^{\preceq(\mu,z-\zeta)}$ as a closed subscheme of the latter. On the other hand, it is not clear to us whether the two conditions that $\delta^{-1}$ is bounded by $\mu$ and that $\delta$ is bounded by $(-\mu)_\dom$ are equivalent over a non-reduced scheme. Therefore we do not know whether one could also define $\wh\Gr^{\preceq(\mu,z-\zeta)}$ as the subscheme on which $\delta$ is bounded by $(-\mu)_\dom$.
As mentioned in Remark~\ref{RemWeylModules}, we could in the definition of boundedness also take the class of all $G$-modules of lowest weight $(-\lambda)_\dom$ or the class of all tilting modules. The latter are self-dual by \cite[E.6]{Jantzen}. Then we would have that $\delta^{-1}$ is bounded by $\mu$ if and only if $\delta$ is bounded by $(-\mu)_\dom$.
\end{remark}

\begin{proposition}\label{PropReducedDefoSp}
Let $\Defo$ be the deformation ring from Theorem~\ref{ThmDefoSp}. Then the reduced quotient $\ol\Defo=(\Defo/\zeta \Defo)_\red$ is a complete normal noetherian integral domain of dimension $\langle 2\rho, \mu\rangle$ and Cohen-Macaulay.
\end{proposition}

\begin{proof}
The ring $\ol\Defo$ equals the completion of the local ring of $x$ in $\Gr^{\preceq\mu}$. Since $\Gr^{\preceq\mu}$ is irreducible, reduced, normal, and Cohen-Macaulay (by \cite[Theorem 6.1]{PappasRapoport} for $\SL_n$ or \cite[Theorem 8 and Remark before Corollary 11]{Faltings03} in general) this local ring is a normal integral domain and Cohen-Macaulay. By Zariski's Main Theorem \cite[Theorem VIII.32]{ZariskiSamuel2} and \cite[Proposition 18.8]{Eisenbud} the same is then true for its completion $\ol\Defo$.
By Proposition \ref{PropSchubertCell}, $\Gr^{\preceq\mu}$ has dimension $\langle 2\rho, \mu\rangle$. Thus the same holds for $\ol\Defo$. 
 \end{proof}

Note that if one is only interested in the dimension (as we are in Proposition~\ref{propnewtondim}), one does not need the normality. Namely, $\Gr^{\preceq\mu}$ is equidimensional of dimension $\langle 2\rho, \mu\rangle$. Thus the same holds for every local ring. By \cite[Theorem 18.17]{HIO} the completion of a local equidimensional and universally catenary ring is equidimensional. 

\begin{example}\label{ExDefoSpNotSmooth}
Contrary to $p$-divisible groups the deformation space of a local $G$-shtuka is in general not smooth. Consider for example the ring $A=\BFZ/(\zeta^2)$ and the local shtuka $$(M,\phi)=\Bigl(A\dbl z\dbr^{\oplus2},\left(\begin{array}{cc}z&0\\0&z-2\zeta\end{array}\right)\s\Bigr).$$ According to Lemma~\ref{LemmaG.3} and since the determinant of $\phi$ equals $(z-\zeta)^2$, it is bounded by $(2,0)$ . We claim that it cannot be deformed to a local shtuka over $\wt A=\BFZ/(\zeta^3)$ bounded by any $\mu$ (which is then of the form $(\mu_1,2-\mu_1)_{\dom}$). Indeed assume there is a deformation $(\wt M,\tilde\phi)=(\wt A\dbl z\dbr^{\oplus2},\tilde b\s)$ with $$\tilde b-\left(\begin{array}{cc}z&0\\0&z-2\zeta\end{array}\right)=\zeta^2\left(\begin{array}{cc}c_{11}&c_{12}\\c_{21}&c_{22}\end{array}\right)\in \zeta^2 M_2\bigl(\wt A\dbl z\dbr\bigr).$$ Then $\det\tilde b=z^2-2\zeta z+\zeta^2 z(c_{11}+c_{22})$ and the boundedness by $(\mu_1,2-\mu_1)$ requires that this determinant differs from
  $(z-\zeta)^2$ by a unit in $\wt A\dbl z\dbr\mal$. However this is not the case.
\end{example}

Nevertheless deformation from $k'$ to $k'\dbl\zeta\dbr$ is always possible:

\begin{lemma}\label{LemmaLift}
Let $\ul\BG$ be a local $G$-shtuka over a field $k'\in\NilpF$ of the form $\ul\BG=\bigl(K_{k'},b\s\bigr)$ for some $b\in LG(k')$. Then there exists a local $G$-shtuka $\ul{\wh\BG}$ over $k'\dbl\zeta\dbr$ with $\ul{\wh\BG}\otimes_{k'\dbl\zeta\dbr}k'\cong\ul\BG$. If $\ul\BG$ is bounded by some dominant coweight $\mu\in X_\ast(T)$ then we can find a $\ul{\wh\BG}$ which is also bounded by $\mu$.
\end{lemma}

\begin{proof}
Write $b=s_0z^{\mu'} t_0$ for some dominant $\mu'\in X_\ast(T)$ and $s_0,t_0\in K(k')$. Choose lifts $s,t\in K\bigl(k'\dbl\zeta\dbr\bigr)$ of $s_0$ and $t_0$ and set $\ul{\wh\BG}:=\bigl(K_{k'\dbl\zeta\dbr},s(z-\zeta)^{\mu'} t\cdot\s\bigr)$. Then $\ul{\wh\BG}$ is a local $G$-shtuka over $k'\dbl\zeta\dbr$. Moreover, $\ul\BG$ is bounded by $\mu$ if and only if $\mu'\preceq\mu$ by Lemma~\ref{LemmaHPBounded}. Then $\langle(-\lambda)_\dom,\mu'\rangle\le\langle(-\lambda)_\dom,\mu\rangle$ for all dominant weights $\lambda$. This shows that $\ul{\wh\BG}$ is bounded by $\mu$.
\end{proof}

%%%%%%%%%%%%%%%%%%%%%%%%%%%%%%%%%%%%%%%%%%%%%%%%%%%%%%%%%%%%%%%%%%%%%%
%
%   Relation with Affine Deligne-Lusztig Varieties
%
%%%%%%%%%%%%%%%%%%%%%%%%%%%%%%%%%%%%%%%%%%%%%%%%%%%%%%%%%%%%%%%%%%%%%%

\section{Rapoport-Zink spaces for local $G$-shtukas and affine Deligne-Lusztig varieties} \label{SectADLV}
\setcounter{equation}{0}

We explain the relation between local $G$-shtukas and affine Deligne-Lusztig varieties. Let $k'$ be a field containing $\mathbb{F}_q$. Recall that $LG(k')=\coprod K(k')z^\mu K(k')$ where the union is over all dominant coweights $\mu$. We recall the definition of affine Deligne-Lusztig varieties from \cite{GHKR}.

\begin{definition}\label{DefADLV}
For an element $b\in LG(k')$ and a dominant $\mu\in X_\ast(T)$ the \emph{affine Deligne-Lusztig variety} $X_\mu(b)$ is the locally closed reduced ind-subscheme over $k'$ of the affine Grassmannian $\X_{k'}:=\X\otimes_{\BF_q}k'$ defined by
\[
X_\mu(b)(\ol{k'})\;=\;\bigl\{\,g\in \X_{k'}(\ol{k'}): g^{-1}b\s(g)\in K(\ol{k'})z^\mu K(\ol{k'})\,\bigr\}\,,
\]
where $\ol{k'}$ is an algebraic closure of $k'$.
The \emph{closed affine Deligne-Lusztig variety} $X_{\preceq\mu}(b)$ is the closed reduced ind-subscheme of $\X_{k'}$ defined by
\[
X_{\preceq\mu}(b)(\ol{k'})\;=\;\bigcup_{\mu'\preceq\mu}X_{\mu'}(b)(\ol{k'}).
\]
\end{definition}

Left multiplication by $g\in LG(k')$ induces an isomorphism between $X_\mu\bigl(g^{-1}b\s(g)\bigr)$ and $X_\mu(b)$. There is a criterion for $X_\mu(b)$ to be non-empty, see \cite{KottwitzRapoport} and \cite{Gashi}. In Corollary~\ref{ADLVisScheme} we give a proof for the well known fact that both $X_\mu(b)$ and $X_{\preceq\mu}(b)$ are schemes locally of finite type over $k'$.

We fix a local $G$-shtuka $\ul\BG$ over $k'$ of the form $\ul\BG=(K_{k'},b\s)$ for some $b\in LG(k')$. For a scheme $S\in\Nilp_{k'\dbl\zeta\dbr}$ we denote by $\bar S$ the closed subscheme $\Var(\zeta)\subset S$. Recall the ind-scheme $\wh \X$ from Section~\ref{SectDeformTh} and set $\wh \X_{k'}:=\wh \X\otimes_{\BF_q}k'$.

\begin{theorem} \label{ThmModuliSpX}
The ind-scheme $\wh \X_{k'}$ pro-represents the functor $\bigl(\Nilp_{k'\dbl\zeta\dbr}\bigr)^o \longto \Sets$
\begin{eqnarray*}
S&\longmapsto & \Bigl\{\,\text{Isomorphism classes of pairs }(\ul\CG,\bar\ppsi)\text{ where} \\
&&\qquad \ul\CG \text{ is a local $G$-shtuka over $S$ and} \\
&&\qquad \bar\ppsi:\ul\CG_{\bar S}\to\ul\BG_{\bar S}\text{ is a quasi-isogeny}\,\Bigr\}
\end{eqnarray*}
Here $(\ul\CG,\bar\ppsi)$ and $(\ul\CG',\bar\ppsi')$ are called isomorphic if $\bar\ppsi^{-1}\circ\bar\ppsi'$ lifts to an isomorphism $\ul\CG'\to\ul\CG$.
\end{theorem}

\begin{proof}
Let us first reformulate the theorem by lifting the statement to $k'\dbl\zeta\dbr$. By Lemma~\ref{LemmaLift} we may choose a local $G$-shtuka $\ul{\wh\BG}$ over $k'\dbl\zeta\dbr$ which lifts $\ul\BG$. Note that $\wh\BG$ has a trivialization $\wh\BG=(K_{k\dbl\zeta\dbr},\hat b\s)$; see the proof of Lemma~\ref{LemmaLift}. By rigidity of quasi-isogenies (Proposition~\ref{PropRigidity}) the functor in question is isomorphic to the functor $\bigl(\Nilp_{k'\dbl\zeta\dbr}\bigr)^o \longto \Sets$
\begin{eqnarray*}
S&\longmapsto & \Bigl\{\,\text{Isomorphism classes of pairs }(\ul\CG,\ppsi)\text{ where } \ul\CG \text{ is a local}\\
&&\qquad \text{$G$-shtuka over $S$ and }  \ppsi:\ul\CG\to\ul{\wh\BG}_S\text{ is a quasi-isogeny}\,\Bigr\}
\end{eqnarray*}
Here $(\ul\CG,\ppsi)$ and $(\ul\CG',\ppsi')$ are called isomorphic if $\ppsi^{-1}\circ\ppsi'$ is an isomorphism $\ul\CG'\to\ul\CG$.

Let $x\in\wh \X_{k'}(S)$ for a scheme $S\in \Nilp_{k'\dbl\zeta\dbr}$. The projection morphism $LG\to \Gr$ admits local sections for the \'etale topology by~\cite[Theorem 4.5.1]{BeilinsonDrinfeld}. Note that in \cite{BeilinsonDrinfeld} this is proved to hold even Zariski locally over an algebraically closed field of characteristic zero. Using \cite[Lemma 2.1]{Ngo-Polo} the proof carries over for our base field $\BF_q$ after allowing a finite separable extension of $\BF_q$. Consequently there is an \'etale covering $S'\to S$ such that $x$ is represented by an element $g'\in LG(S')$. Define $(\CG',\phi',\ppsi')$ over $S'$ as follows. Let $\CG'=K_{S'}$, let the quasi-isogeny $\ppsi':(\CG',\phi')\to\ul{\wh\BG}_{S'}$ be given by $y\mapsto g'y$, and the Frobenius by $\phi'=(g')^{-1}\hat b\s(g')\s$. 
We descend $(\CG',\phi',\ppsi')$ to $S$. For an $S$-scheme $\Test$ let $\Test'=\Test\times_SS'$ and $\Test''=\Test'\times_\Test\Test'$, and let $p_i:\Test''\to \Test'$ be the projection onto the $i$-th factor. Since $g'$ comes from an element $x\in\wh\X_{k'}(S)$ there is an $h\in K(S'')$ with $p_1^\ast(g')=p_2^\ast(g')\cdot h$. Consider the \fpqc-sheaf $\CG$ on $S$ whose sections over an $S$-scheme $\Test$ are given by
\[
\CG(\Test)\;:=\;\bigl\{\,y'\in K(\Test'):\es p_1^\ast(y')=h^{-1}\cdot p_2^\ast(y')\text{ in } K(\Test'')\,\bigr\}
\]
on which $K(\Test)$ acts by right multiplication. Then $\ul\CG$ is a $K$-torsor on $S$ because over $\Test=S'$ there is a trivialization
\[
K_{S'}\isoto \CG_{S'}\,,\quad f\mapsto h\,p_1^\ast(f)\es\in\;K(S'')
\]
due to the cocycle condition on $h$. Moreover, $\phi'$ descends to an isomorphism $\phi:\s\CL\CG(\Test)\isoto\CL\CG(\Test)\,,\,\s(y')\mapsto(g')^{-1}\hat b\s(g')\s(y')$ making $(\CG,\phi)$ a local $G$-shtuka over $S$. Also $\delta'$ descends to a quasi-isogeny of local $G$-shtukas
\[
\delta:\CL\CG(\Test)\isoto LG(\Test)=\bigl\{\,f'\in LG(\Test'):\;p_1^\ast(f')=p_2^\ast(f')\text{ in }LG(\Test'')\,\bigr\}\;,\quad y'\mapsto g'y'\,.
\]
Note that this is well defined. Namely, if $g'$ is replaced by $\tilde g'$ with $u'=(\tilde g')^{-1}g'\in K(S')$ then left multiplication with $u'$ defines an isomorphism $\bigl(K_{S'},(g')^{-1}\hat b\s(g')\s,g'\bigr)\isoto\bigl(K_{S'},(\tilde g')^{-1}\hat b\s(\tilde g')\s,\tilde g'\bigr)$. Also $\tilde h=p_2^\ast(u')\,h\,p_1^\ast(u')^{-1}$ and hence left multiplication with $u'$ descends to an isomorphism $\ul\CG\isoto\ul{\wt\CG}$ over $S$.

Conversely let $\ul\CG=(\CG,\phi)$ be a local $G$-shtuka and let $\ppsi:\ul\CG\to\ul\BG_{S}$ be a quasi-isogeny over $S$. By Proposition~\ref{PropGzTorsor} there is an \'etale covering $S'\to S$ such that the pull back of $\CG$ to $S'$ is trivial. After choosing a trivialization, the map $\ppsi$ is given by an element $g'\in LG(S')$ whose image in $\wh \X_{k'}(S')$ is independent of the chosen trivialization. The quasi-isogeny $\ppsi$ also determines $\phi=(g')^{-1}\hat b\s(g')\s$. Since $\ul\CG$ was defined over $S$ the element $g'$ descends to a point $x\in\wh \X_{k'}(S)$. Clearly these two constructions are inverse to each other.
\end{proof}

Next we want to show that $X_{\preceq\mu}(b)$ is the underlying reduced subscheme of an equal characteristic Rapoport-Zink space for local $G$-shtukas. More precisely, consider again the local $G$-shtuka $\ul\BG=(K_{k'},b\s)$ from above. Note that each isogeny class of local $G$-shtukas has a representative over a finite field, so we may assume that $k'$ is finite. Besides we may assume that $b\s$ is \emph{decent}, i.e. there are an integer $s>0$ and a cocharacter $\omega\in X_*(T)$, such that 
\begin{equation}\label{EqDecent}
(b\s)^s\;:=\;b\,\s(b)\cdot\ldots\cdot(\s)^{s-1}(b)\,(\s)^s\;=\;z^\omega(\s)^s\,.
\end{equation}
Let $\mu\in X_\ast(T)$ be a dominant coweight such that $\ul\BG$ is bounded by $\mu$, and consider the functor $\mathcal{M}:(\Nilp_{k'\dbl\zeta\dbr})^o\to\Sets$,
\begin{eqnarray*}
S&\longmapsto & \Bigl\{\,\text{Isomorphism classes of pairs }(\ul\CG,\bar{\ppsi})\text{ where} \\
&&\qquad \ul\CG \text{ is a local $G$-shtuka over $S$ bounded by $\mu$ and} \\
&&\qquad \bar{\ppsi}:\ul\CG_{\bar S}\to\ul\BG_{\bar S}\text{ is a quasi-isogeny}\,\Bigr\}\;.
\end{eqnarray*}
Again $\bar S$ denotes the closed subscheme $\Var(\zeta)\subset S$.

\begin{theorem} \label{ThmADLV=RZSp}
Let $\wh X_{\preceq\mu}(b)$ be the closed ind-subscheme of $\wh \X_{k'}$ defined by the condition that the universal $G$-shtuka $\ul\CG$ over $\wh \X_{k'}$ from Theorem~\ref{ThmModuliSpX} is bounded by $\mu$. Then $\wh X_{\preceq\mu}(b)$ is a formal scheme over $\Spf k'\dbl\zeta\dbr$ which is locally formally of finite type. It pro-represents the functor $\mathcal{M}:(\Nilp_{k'\dbl\zeta\dbr})^o\to\Sets$. Its underlying reduced subscheme equals $X_{\preceq\mu}(b)$.
\end{theorem}
Here by the underlying reduced subscheme we mean the subscheme corresponding to the largest ideal of definition. Also recall that a formal scheme over $k'\dbl\zeta\dbr$ in the sense of \cite[I$_{new}$, 10]{EGA} is called \emph{locally formally of finite type} if it is locally noetherian and adic and its reduced subscheme is locally of finite type over $k'$. It is called \emph{formally of finite type} if in addition it is quasi-compact. For the proof of the theorem we need the following proposition in which $2\rho\dual$ is the sum of all positive coroots of $G$.

\begin{proposition}\label{prop2.27}
Let $b\in LG(k')$ satisfy a decency condition with the integer $s$ as in (\ref{EqDecent}). Assume $k'\subset\BF_{q^s}$. Then there is a natural number $d_0$ such that for every algebraically closed field $k$ and any point $(\ul\CG,\bar{\ppsi})\in\mathcal{M}(k)$ there is a point $(\ul\CG',\bar{\ppsi}')\in\mathcal{M}(\mathbb{F}_{q^s})$ such that $\bar{\ppsi}^{-1}\circ \bar{\ppsi}'$ is bounded by $2d_0\rho\dual$. 
\end{proposition}

\begin{proof}
By trivializing $\ul\CG$ the assertion is equivalent to the statement that for any $g\in LG(k)$ with $g\in X_{\preceq\mu}(b)(k)$ there is a $g'\in LG(\BF_{q^s})$ with $g'\in X_{\preceq\mu}(b)(\BF_{q^s})$ and such that $g^{-1}g'\in Kz^{\mu'}K$ for some $\mu'\preceq 2d_0\rho\dual$. This last condition is especially satisfied if the distance of $g$ and $g'$ in the Bruhat-Tits building of $G$ is less than $d_0$. Hence the proposition follows from \cite[Theorem 1.4 and Subsection 2.1]{RZ2}.
\end{proof}

\begin{proof}[Proof of Theorem \ref{ThmADLV=RZSp}]

Note that by construction $\wh X_{\preceq\mu}(b)$ represents $\CM$. Also $\wh X_{\preceq\mu}(b)$ is a closed ind-subscheme of $\wh \X_{k'}$ by Lemma~\ref{LemmaBoundedIfQC}. We show that the underlying topological space of $\wh X_{\preceq\mu}(b)$ is the ind-scheme $X_{\preceq\mu}(b)$. As $X_{\preceq\mu}(b)$ and the underlying reduced ind-subscheme of $\wh X_{\preceq\mu}(b)$ are both reduced ind-subschemes of $\X$, it then follows that they are equal. Let $x\in\wh \X$ be a point with values in an algebraically closed field $\kappa(x)$. Put $(\ul{\bar\CG},\bar\ppsi):=(\ul\CG,\ppsi)_{\Spec\kappa(x)}$, and let $\bar\ppsi$ be given by $g\in LG(\kappa(x))$ with respect to some trivialization of $\ul{\bar\CG}=(\bar\CG,\bar\phi)$. Then $\bar\phi=g^{-1}b\s(g)\cdot\s$ and $g\in X_{\mu'}(b)$ for the Hodge polygon $\mu'=\mu_{\ul{\bar\CG}}(x)$ of $(\bar\CG,\bar\phi)$ (Definition~\ref{DefHodgePoly}). We have that $g\in\wh X_{\preceq\mu}(b)$ if and only if $(\bar\CG,\bar\phi)$ is bounded by $\mu$.  By Lemma~\ref{LemmaHPBounded} this is equivalent to $\mu'\preceq\mu$, or to $g\in X_{\preceq\mu}(b)$ as desired. 

It remains to show that $\wh X_{\preceq\mu}(b)$ is a formal scheme locally formally of finite type. We follow the proof of \cite[Theorem 2.16]{RZ}. By Lemma~\ref{LemmaLift} we can choose a lift $\ul{\wh\BG}$ of $\ul\BG$ to $k'\dbl\zeta\dbr$ which is bounded by $\mu$, and using rigidity (Proposition~\ref{PropRigidity}) we replace $\CM$ by the isomorphic functor $\bigl(\Nilp_{k'\dbl\zeta\dbr}\bigr)^o \longto \Sets$ 
\begin{eqnarray*}
S&\longmapsto & \Bigl\{\,\text{Isomorphism classes of pairs }(\ul\CG,\ppsi)\text{ where} \\
&&\qquad \ul\CG \text{ is a local $G$-shtuka over $S$ bounded by $\mu$ and} \\
&&\qquad \ppsi:\ul\CG\to\ul{\wh\BG}_S\text{ is a quasi-isogeny}\,\Bigr\}\;.
\end{eqnarray*}

For $v\in \pi_1(G)$ let $\mathcal{M}(v)$ be the open and closed ind-subscheme of $\mathcal{M}=\wh X_{\preceq\mu}(b)$ where the image of $\ppsi$ in $\pi_1(G)$ is $v$ (using Proposition~\ref{PropHPLocConst}). By \cite[Proof of Lemma 1]{conncomp} there is a self-quasi-isogeny $j_v$ of $\ul{\wh\BG}$ whose image in $\pi_1(G)$ is the given $v$. Composition of the quasi-isogenies $\ppsi$ with $j_v$ yields an isomorphism between the ind-subschemes $\mathcal{M}(0)$ and $\mathcal{M}(v)$. Hence it is enough to prove the theorem for $\Mnull:=\mathcal{M}(0)$ instead of $\wh X_{\preceq\mu}(b)$. 

We consider two kinds of ind-schemes related to $\Mnull$. Let $n\in\mathbb{N}$. By Lemma~\ref{LemmaBoundedIfQC}, the condition that the universal $\ppsi, \ppsi^{-1}$, and $\s \ppsi$ over $\wh X_{\preceq\mu}(b)$ are all bounded by $(2n\rho\dual,z)$ is represented by a closed ind-subscheme $\mathcal{M}^n$ of $\wh X_{\preceq\mu}(b)$. Note that (\ref{EqBounded2}) in Definition~\ref{DefBounded} implies that $\CM^n$ is contained in $\Mnull$.
We show that $\CM^n$ is a $\zeta$-adic noetherian formal scheme over $k'\dbl\zeta\dbr$. The condition that the universal $\ppsi$ on $\wh{\X}$ from Proposition~\ref{Prop4.1}, as well as $\ppsi^{-1}$ and $\s(\ppsi)$ are all bounded by $(2n\rho\dual,z)$ is represented by a closed ind-subscheme $\wh{X}^n$ of $\wh{\X}$ due to Lemma~\ref{LemmaBoundedIfQC}. By an argument analogous to Proposition~\ref{PropProSchubertCell}, $\wh{X}^n$ is a $\zeta$-adic noetherian formal scheme over $k'\dbl\zeta\dbr$ whose underlying topological space is the projective scheme $\Gr^{\preceq 2n\rho\dual}$ from Proposition~\ref{PropSchubertCell}. The Frobenius $\phi_\CG$ of the universal $G$-shtuka $\ul\CG=(\CG,\phi_\CG)$ over $\wh{X}^n$ from Theorem~\ref{ThmModuliSpX} satisfies $\phi_\CG=\ppsi^{-1}\circ\phi_{\wh\BG}\circ\s(\ppsi)$ and hence $\phi_\CG(\s\CG_{\lambda})\subseteq z^{-\langle (-\lambda)_{\dom},4n\rho\dual\rangle}(z-\zeta)^{-\langle(-\lambda)_{\dom},\mu\rangle}\CG_{\lambda}$ due to the definition of $\wh{X}^n$ and the boundedness of $\ul{\wh\BG}$ by $\mu$. Analogous to Lemma~\ref{LemmaBoundedIfQC}, the condition that $\phi_\CG$ as above is bounded by $(\mu,z-\zeta)$ is equivalent to the following condition. For each $\lambda$, the images of the generators of $\s\CG_\lambda$ in the locally free sheaf $z^{-\langle (-\lambda)_{\dom},4n\rho\dual\rangle}(z-\zeta)^{-\langle(-\lambda)_{\dom},\mu\rangle}\CG_{\lambda}\big/(z-\zeta)^{-\langle(-\lambda)_{\dom},\mu\rangle}\CG_{\lambda}$ of finite rank on $\wh{X}^n$ have to vanish. Hence $\CM^n$ is a $\zeta$-adic noetherian closed formal subscheme of $\wh{X}^n$ over $k'\dbl\zeta\dbr$ whose reduced subscheme is projective. 

For $n\in\mathbb{N}$ let $\mathcal{M}_n$ be the ind-scheme which is the formal completion of $\wh X_{\preceq\mu}(b)$ along the closed subset $(\mathcal{M}^n)_\red$, where $(\CM^n)_\red$ denotes the topological space underlying $\CM^n$. In particular $(\CM^n)_\red=(\CM_n)_\red$. Over a field in $\Nilp_{k'\dbl\zeta\dbr}$, Lemma \ref{LemmaHPBounded} implies that $\ppsi_s$ is bounded by $2n\rho\dual$ if and only if $\ppsi^{-1}$ or $\s\ppsi$ are bounded by $2n\rho\dual$. Hence $\CM_n$ represents the subfunctor of $\CM$
\[
S\;\longmapsto \;\{(\ul\CG,\ppsi)\in\mathcal{M}(S):\text{ for all closed points }s\in S\text{ the map }\ppsi_s \text{ is bounded by }2n\rho\dual\}.
\]

\smallskip
\noindent
{\it Claim.} $\mathcal{M}_n$ is representable by a formal scheme which is formally of finite type over $\Spf k'\dbl\zeta\dbr$. 
\smallskip

Fix $n$ and for each $m\ge n$ let $\mathcal{M}^m_n$ be the formal completion of $\CM^m$ along $(\CM_n)_\red$, i.e. the ind-scheme on which all the above boundedness conditions are satisfied. It is an adic noetherian formal scheme over $k'\dbl\zeta\dbr$.
We fix an affine open subscheme $U$ of $(\CM_n)_{\red}$. For $m\geq n$ we have $(\CM_n)_\red=(\CM_n^m)_{\red}$, and thus we get an affine open formal subscheme $\Spf R_m$ of $\CM_n^m$ whose underlying set is $U$.
Since by construction $\mathcal{M}^m\to \mathcal{M}^{m+1}$ is a closed immersion, we have a projective system of surjective maps of adic rings $R_{m+1}\rightarrow R_m$. Let $R$ be its limit. We write $R_m=R/\mathfrak{a}_m$ for ideals $\mathfrak{a}_m\subset R$. Let $J$ be the inverse image in $R$ of the largest ideal of definition of $R_n$. We have to prove that $R$ is a $J$-adic ring. Since $R_m$ is $J$-adic for all $m$ we may write $R=\invlim R/(\mathfrak{a}_m+J^c)$. The limit is taken independently over $m$ and $c$. By \cite[2.5]{RZ} it is enough to show that for each $c$ the descending sequence $\mathfrak{a}_m+J^c\supseteq \mathfrak{a}_{m+1}+J^c\supseteq \dotsm$ stabilizes. Let $\ul\CG_m$ be the universal local $G$-shtuka on $\Spf R_m$. Since $R/J^c=\invlim R_m/J^c$ is an admissible ring, $\ul\CG=\invlim \ul\CG_m$ defines a local $G$-shtuka on $\Spec R/J^c$ for every $c$ by Proposition~\ref{remdejong}. By rigidity (Proposition~\ref{PropRigidity}), we may also lift the quasi-isogeny $\ppsi$ from $R/J= R_n/J$ to $R/J^c$. We want to see that there is an integer $m_0\ge n$, such that $\ppsi,\ppsi^{-1}$, and $\s\ppsi$ are bounded by $2m_0\rho\dual$ over $R/J^c$. By the second assertion of Proposition~\ref{PropRigidity}, they are all bounded by some coweight $\omega$ whose image in $\pi_1(G)$ is trivial. Choosing $m_0$ such that $\omega\preceq 2m_0\rho\dual$ we obtain a bound of the desired form. Hence by the universal property of $\CM_n^{m_0}$ there is a unique map $R_{m_0}\rightarrow R/J^c$ inducing the given point $(\ul\CG,\ppsi)$. For any $m\geq{m_0}$ the composite $R_m\rightarrow R_{m_0}\rightarrow R/J^c\rightarrow R_m/J^cR_m$ is the projection, since both maps induce $(\ul\CG_m,\ppsi)$ over $R_m/J^cR_m$. Thus the first map yields an isomorphism $R_m/J^cR_m\rightarrow R_{m_0}/J^cR_{m_0}$. This implies that the sequence above stabilizes and proves our claim.

Let $d_0$ be as in Proposition \ref{prop2.27}. The assertion of the proposition obviously remains valid for $\Mnull$. For two closed points $x=(\ul\CG,\ppsi),x'=(\ul\CG',\ppsi')\in \Mnull$ we define $d(x,x')$ to be the smallest $n\in\mathbb{N}$ such that $\ppsi^{-1}\circ\ppsi'$ is bounded by $2n\rho\dual$. For a point $y\in \Mnull(\mathbb{F}_{q^s})$ we denote by $\mathcal{M}_n(y)$ the closed subset of points $x$ of $\mathcal{M}_n$ with $d(x,y)\leq d_0$. By the triangular inequality for $d$ we obtain that $\mathcal{M}_n(y)=\emptyset$ if $d\bigl((\ul\BG,\id),y\bigr)>n+d_0$, that is if $y\notin\CM_{n+d_0}$. For each positive integer $f$ let $U^f_n$ be the open formal subscheme of $\mathcal{M}_n$ whose underlying set is the complement of 
\[
\bigcup_{y\in \Mnull(\BF_{q^s})\,,\, d((\ul\BG,\id),y)\geq f}\mathcal{M}_n(y)\,\es=\es\bigcup_{y\in \CM_{n+d_0}(\BF_{q^s})\,,\, d((\ul\BG,\id),y)\geq f}\mathcal{M}_n(y)\,.
\]
Note that this union is finite because the underlying reduced subscheme of $\mathcal{M}_{n+d_0}$ is the projective scheme $(\CM^{n+d_0})_\red$.

\smallskip
\noindent
{\it Claim 2.} If $n\geq f+d_0$ then $U^f_n=U^f_{n+1}$.
\smallskip

We first show this for the underlying sets. Let $x\in U_{n+1}^f(k)$ be a point with values in an algebraically closed field $k$. We have to show that $d\bigl((\ul\BG,\id),x\bigr)\leq n$. By Proposition~\ref{prop2.27} there is a $y\in\Mnull(\mathbb{F}_{q^s})$ such that $d(x,y)\leq d_0$ and $x\in\CM_{n+1}(y)$. By the definition of $U_{n+1}^f$, we have $d\bigl((\ul\BG,\id),y\bigr)<f$, hence $d\bigl((\ul\BG,\id),x\bigr)<f+d_0\leq n$. The equality of formal schemes follows because $\mathcal{M}_n$ is the completion of $\mathcal{M}_{n+1}$ along the (closed) image of $(\mathcal{M}_n)_\red$. Indeed, this implies that $U_n^f$ is the completion of $U_{n+1}^f$ along $U_{n+1}^f$. Hence the claim follows.

Let $U^f=U^f_n$ for some $n\geq f+d_0$. Then $U^f\rightarrow U^{f+1}$ is an open immersion of formal schemes formally of finite type. Moreover, $\Mnull$ has the limit topology of the limit over its intersections with the subschemes of $\wh\Gr_{k'}$ where the universal $\ppsi$ is bounded by $2n\rho\dual$. The underlying topological space of such an intersection is equal to that of $\CM_n$. Since the topological space $|U^f|$ underlying $U^f$ is open in each $\CM_n$, it is also open in $\Mnull$. This shows that the formal scheme $U^f$ equals the formal completion of the open ind-scheme $\Mnull|_{|U^f|}$ of $\Mnull$ supported on $|U^f|$ along the whole set $|U^f|$, whence $\Mnull|_{|U^f|}=U^f$. Since $U^f$ is formally of finite type by construction, also $\Mnull|_{|U^f|}$ is. In order to prove the theorem it remains to show that $\Mnull=\bigcup_f U^f$. We prove that every point $x\in\Mnull$ with $d\bigl((\ul\BG,\id),x\bigr)<f-d_0$ is contained in the open set $U^f$. Indeed, if $x$ is in the complement of $U^f$, there is a $y\in \Mnull(\mathbb{F}_{q^s})$ with $d(x,y)\leq d_0$ and $d\bigl((\ul\BG,\id),y\bigr)\geq f$, a contradiction. The theorem follows.
\end{proof}

\begin{corollary}\label{ADLVisScheme}
Let $b\in LG(k')$ be decent. The affine Deligne-Lusztig varieties $X_{\preceq\mu}(b)$ and $X_\mu(b)$ are schemes locally of finite type over $k'$. The irreducible components of $X_{\preceq\mu}(b)$ (respectively of $X_\mu(b)$) are projective (respectively quasi-projective).
\end{corollary}

\begin{proof}
Let $V\subset X_{\preceq\mu}(b)$ be an irreducible component and let $\eta$ be its generic point. At $\eta$ the universal $\ppsi$, $\ppsi^{-1}$, and $\s\ppsi$ are bounded by $2n\rho\dual$ for some $n$. Due to Lemma~\ref{LemmaBoundedIfQC} they are then bounded by $2n\rho\dual$ on all of $V$. Hence $V$ is contained in the projective scheme $(\CM^n)_\red$ as a closed subscheme, and therefore $V$ is projective. Since $X_\mu(b)$ is open in $X_{\preceq\mu}(b)$, each irreducible component of $X_\mu(b)$ is quasi-projective.
\end{proof}

Note that if one only wants to prove this corollary, one can use a significantly simplified version of the proof of Theorem~\ref{ThmADLV=RZSp}. Namely one does not need the distinction between $\CM^n$ and $\CM_n$, and the fact that $\CM_n$ is a formal scheme formally of finite type.

\begin{remark}\label{RemWeylModules2}
If one chooses a different definition for boundedness as mentioned in Remark~\ref{RemWeylModules}, one possibly obtains a different Rapoport-Zink space $\wh{X}_{\preceq\mu}(b)$, whose underlying topological space still coincides with $X_{\preceq\mu}(b)$. If one enlarges the class of $G$-modules which are used to define boundedness, one obtains the new Rapoport-Zink space as a closed subspace of the former one. 
\end{remark}

%%%%%%%%%%%%%%%%%%%%%%%%%%%%%%%%%%%%%%%%%%%%%%%%%%%%%%%%%%%%%%%%%%%%%%
%
%   Newton Stratification
%
%%%%%%%%%%%%%%%%%%%%%%%%%%%%%%%%%%%%%%%%%%%%%%%%%%%%%%%%%%%%%%%%%%%%%%

\section{The Newton stratification}\label{secnewton}

Let $\ul\CG$ be a local $G$-shtuka over an algebraically closed field $k$. By Remark~\ref{RemEtaleTrivial}, $\CG$ is isomorphic to the trivial $K$-torsor on $k$ and the Frobenius is given by some element $b\in LG(k)$. Changing the trivialization corresponds to $\sigma$-conjugation of $b$ by elements of $K(k)$. On the other hand, considering $b$ up to $\sigma$-conjugation by $LG(k)$ corresponds to considering the \emph{local $G$-isoshtuka} $\ul{\CL\CG}\cong(LG_k,b\s)$. We briefly review Kottwitz's classification of the latter kind of $\sigma$-conjugacy classes. To be precise, we use the analog for the equal characteristic case of his classification.

\begin{remark}\label{remclass}
Let $G$ be a connected reductive group over an algebraically closed field $k$ of characteristic $p$. Let $B(G)$ be the set of $\sigma$-conjugacy classes of elements $b\in LG(k)$. We denote the class of $b$ (which contains all $g^{-1}b\s(g)$ for $g\in LG(k)$) by $[b]$. In \cite{Kottwitz85}, \cite{Kottwitz97}, Kottwitz classifies the classes $[b]\in B(G)$ by two invariants. For split groups the first invariant, the \emph{Kottwitz point} $\kappa(b)$ has the following explicit definition. Let $\mu_{\ul\CG}\in X_*(T)$ be dominant with $b\in K(k)z^{\mu_{\ul\CG}}K(k)$. Then $\kappa(b)$ is the image of $\mu_{\ul\CG}$ under the projection $X_*(T)\rightarrow \pi_1(G)$. Indeed, the maps $\kappa_G$ considered by Kottwitz are invariant under $\sigma$-conjugation of $b$, they are group homomorphisms, and are natural transformations in $G$. For $G=\mathbb{G}_m$ we have $\kappa_G(b)=v_z(b)$. See \cite[Section 1]{RapoportRichartz} for the reformulation as maps to the fundamental group. Let now $G$ be split. The properties of $\kappa_G$ mentioned above show that $\kappa_G$ is trivial on $K(k)$ as each such element is $\sigma$-conjugate to $1$ (by a version of Lang's theorem), and that the torus element $z^{\mu}$ is mapped to its image in $\pi_1(G)$. Together we obtain the explicit description of $\kappa_G$ given above. Hence $\kappa_G:LG(k)\rightarrow\pi_1(G)$ coincides with the map in Proposition~\ref{PropHPLocConst}. The second invariant is the \emph{Newton point} or \emph{Newton polygon} $\nu_b\in X_\ast(T)_{\mathbb{Q}}$. For $G=GL_n$ this is the usual Newton polygon of the $\sigma$-linear map $b\s$, for general $G$ it is defined by requiring that it be functorial in $G$. The two invariants have the same image in $\pi_1(G)_{\mathbb{Q}}$, which is the quotient of $X_\ast(T)_{\mathbb{Q}}$ by the sub-vector space generated by the coroots. Note that not all elements of $X_\ast(T)_{\mathbb{Q}}$ occur as Newton points. For a description of the set of Newton polygons
  in $X_\ast(T)_{\mathbb{Q}}$ see also \cite[4]{Chai}.
\end{remark}

If $\ul\CG$ is a local $G$-shtuka on $S\in\NilpF$, we consider the function $[b_{\ul\CG}]:S\rightarrow B(G)$ with $s\mapsto  [b(s)]$ where $[b(s)]$ is the element associated to the reduction of $\ul\CG$ in the geometric point $s$. The Newton polygon of $[b(s)]$ is called the Newton polygon of $\ul\CG$ in $s$. 

We use the following analogy to isocrystals with $G$-structure defined by Rapoport and Richartz, \cite{RapoportRichartz}. Let $S$ be a connected $\mathbb{F}_q$-scheme. Let $\CI sosh_S$ be the category of pairs $(N,F)$ consisting of an $S\dpl z\dpr $-module $N$ which is \'etale-locally on $S$ free of finite rank and of an isomorphism $F:\s(N)\rightarrow N$. The elements of $\CI sosh_S$ are called local isoshtukas, and are a function field analog of isocrystals. Let $G$ be as above and $\ul\CG=(\CG,\varphi)$ a local $G$-shtuka on $S$. Let $V$ be an $\mathbb{F}_q\dpl  z\dpr $-representation of $G$. Then let $\CG_V$ be the sheaf associated with the presheaf
\[
\Test\;\longmapsto\;\Bigl(\CG(\Test)\times\bigl(V\otimes_{\BF_q\dpl z\dpr}\CO_S\dpl z\dpr(\Test)\bigr)\Bigr)\big/K(\Test)\,.
\]
Together with the $\sigma$-linear isomorphism induced by $(\varphi,\id_V)$ it is an element of  $\CI sosh_S$, since $\CG$ is trivialized by an \'etale covering of $S$ according to Proposition~\ref{PropGzTorsor}. This defines an exact faithful tensor functor $\ul\CG:\Rep_{\mathbb{F}_q\dpl  z\dpr }G\rightarrow \CI sosh_S$. These functors, called \emph{isoshtukas with $G$-structure}, are the function field analog of isocrystals with $G$-structure as defined by Rapoport and Richartz in \cite[Definition 3.3]{RapoportRichartz}. Note that the rational invariant $[b]$ of $\ul\CG$ as defined above is the same as the $\sigma$-conjugacy class of the Frobenius of the local isoshtuka with $G$-structure associated to $\ul\CG$ in \cite[3.4(i)]{RapoportRichartz}. 

\begin{proposition}\label{mazur}
Let $\ul\CG$ be a local $G$-shtuka over an algebraically closed field in $\Nilp_{\BFZ}$. Let $\nu$ be its Newton polygon and let $\mu$ be its Hodge polygon (as in Definition \ref{DefHodgePoly}). Then $\nu\preceq\mu$ as elements of $X_\ast(T)_{\mathbb{Q}}$.
\end{proposition}

\begin{proof} 
This result is first shown by Katz in \cite[1.4.1]{Katz} (Mazur's Theorem) for $\sigma^a$-$F$-crystals. It is generalized in \cite[Theorem 4.2 (ii)]{RapoportRichartz} to elements $b\in G(L)$ 
where $G$ is an unramified reductive group over $k$ and where $L$ is the quotient field of the Witt ring of $k$. Proposition \ref{mazur}, which is the function field analog of those results, can be shown by the same argument.
\end{proof}

The behaviour of $[b]$ under specialization is described by the following theorem. Its proof is completely analogous to the corresponding proof in \cite{RapoportRichartz}, using again the above analogy between local $G$-shtukas and isocrystals with $G$-structure. Note that the additional statement in \cite[Theorem 3.6(i)]{RapoportRichartz} that the image of $\nu$ in $\pi_1(G)_{\mathbb{Q}}$ is locally constant is in our context a consequence of Proposition \ref{PropHPLocConst}.

\begin{theorem}[{\cite[Theorem 3.6]{RapoportRichartz}}]\label{spezial}
Let $S\in\NilpF$ and let $\ul\CG$ be a local $G$-shtuka on $S$. For each $b_0\in B(G)$, the subset $\{s\in S:\nu_{b(s)}\preceq \nu_{b_0}\}$ is Zariski-closed on $S$ and locally on $S$ the zero set of a finitely generated ideal. 
\end{theorem}
Let $S$ and $\ul\CG$ be as in the theorem and let $\nu\in X_\ast(T)_{\mathbb{Q}}$ be a Newton polygon. Then by Theorem \ref{spezial} the reduced subscheme $\CN_{\preceq\nu}$ of $S$ with $$\mathcal{N}_{\preceq\nu}=\{s\in S :  \nu_{b(s)}\preceq \nu\}$$ is a closed subscheme of $S$ and similarly $$\mathcal{N}_{\nu}=\{s\in S :  \nu_{b(s)}= \nu\}$$ defines a locally closed subscheme of $S$, called the \emph{Newton stratum} associated to $\ul\CG$ and $\nu$. 

\begin{theorem} \label{reinheit}
Let $S$ be an integral and locally noetherian $\mathbb{F}_q$-scheme and let $\ul\CG$ be a local $G$-shtuka on $S$. Let $\nu$ be the Newton polygon in the generic point of $S$. Then the Newton stratification on $S$ defined by $\ul\CG$ satisfies the purity property, that is the inclusion of the stratum $\mathcal{N}_{\nu}$ in $S$ is an affine morphism.
\end{theorem}
\begin{proof}
For $G=\GL_r$ this can be shown in mostly the same way as Vasiu's corresponding result for $F$-crystals, \cite[Theorem 6.1]{Vasiu}. The only place where one needs more than the obvious translation is that the results from his Section 5.1 (whose proof does not generalize directly) have to be replaced by our Corollary \ref{corhombound}. As the composition of affine morphisms is affine, this also implies the theorem for all $G$ which are a product of finitely many factors $\GL_{r_i}$. For general $G$ one considers a finite number of representations $\lambda_i$ of $G$ distinguishing the different Newton polygons on $S$ (for example one for each simple factor of the adjoint group $G_{\ad}$). Then the Newton stratifications of $S$ corresponding to  $\ul\CG$ and  $\ul\CG_{\lambda_1}\times \dotsc\times \ul\CG_{\lambda_l}$ coincide.
\end{proof}
\begin{corollary}\label{corpure}
Let $S$ be an integral, locally noetherian $\mathbb{F}_q$-scheme of dimension $d$ and let $\ul\CG$ be a local $G$-shtuka on $S$. Let $S'$ be the complement of the generic Newton stratum. Then $S'$ is empty or pure of dimension $d-1$.
\end{corollary}
\begin{proof}
By replacing $S$ by suitable open subschemes, we may assume that $S$ is affine and that $S'\neq\emptyset$. Similarly it is enough to show that $\dim (S')=d-1$. By Theorem \ref{reinheit}, $S\setminus S'$ is affine. By \cite[IV Corollaire 21.12.7]{EGA} this implies that $\dim (S)-\dim(S')\leq 1$.
\end{proof}
See \cite[Remark 6.3(a)]{Vasiu} for another proof of this corollary.

\begin{definition}
Let $S$ be a scheme and let $\ul\CG$ be a local $G$-shtuka on $S$. Let $\Test$ be an irreducible component of the Newton stratum $\CN_{\nu}\subseteq S$ for some $\nu$. We call a chain of Newton polygons $\nu=\nu_n\prec\nu_{n-1}\prec \dotsc\prec \nu_0$ with $\nu_i\neq\nu_{i-1}$ for all $i$ \emph{realizable in $S$ at $\Test$} if for each $i$ there is an irreducible subscheme $S_i$ of the corresponding Newton stratum $\mathcal{N}_{\nu_i}$ such that $S_{i}\subseteq\ol{S_{i-1}}\setminus S_{i-1}$ for all $i>0$ and such that $S_n=\Test$. We call $n$ the \emph{length} of the chain.
\end{definition}
 
Let $\alpha_1\dual,\dotsc,\alpha_r\dual$ be the simple coroots of $G$. Then we choose $\omega_i\in X^*(T)_{\mathbb{Q}}$ for $i=1,\dotsc,r$ with $\langle\omega_i,\alpha_j\dual\rangle=\delta_{ij}$. Note that these elements are in general not unique. However, we will only use them in expressions of the form $\langle\omega_i,\mu\rangle$ for $\mu$ in the sub-vector space of $X_*(T)_{\mathbb{Q}}$ generated by the coroots. These values do not depend on the particular choice of the $\omega_i$.

\begin{corollary}\label{cor3.4}
Let $S$ be an irreducible $\mathbb{F}_q$-scheme and let $\ul\CG$ be a local $G$-shtuka on $S$. Let $\nu_0$ be the Newton polygon of $\ul\CG$ at the generic point of $S$. Let $\nu\preceq \nu_0$ with $\mathcal{N}_{\nu}\neq\emptyset$ and let $\Test$ be an irreducible component of $\mathcal{N}_{\nu}$. 
\begin{enumerate}
\item The difference $\dim S-\dim \Test$ equals the maximal length $n$ of a realizable chain of Newton polygons in $S$ at $\Test$.
\item $\dim (\Test)\geq \dim(S)-\sum_i \lceil\langle\omega_i,\nu_0-\nu\rangle\rceil.$
\end{enumerate}
\end{corollary}

\begin{proof}
Let $\nu_n\prec\dotsc\prec \nu_0$ be a maximal realizable chain at $\Test$ and let $S_i$ be corresponding irreducible subschemes of $\CN_{\nu_i}$. As the chain is maximal, $\nu_0$ is equal to the Newton polygon at the generic point of $S$, so we may assume that $S_0$ is dense in $S$. We have to show that we can modify the $S_i$ such that they satisfy in addition that $\dim S_i -\dim S_{i+1}=1$ for all $i$. As the $S_i$ are irreducible and not equal, the difference is at least one. As $S_{i+1}$ is irreducible, it is contained in an irreducible component of $\ol{S_i}\setminus (\CN_{\nu_i}\cap \ol{S_i})$, which by Corollary \ref{corpure} has dimension $\dim (S_i)-1$. As the chain is maximal, the generic Newton polygon of this component is $\nu_{i+1}$. For $i+1<n$ this implies that we may replace $S_{i+1}$ by that component (using increasing induction on $i$) and obtain $\dim S_{i+1}=\dim S_i-1$. For $i+1=n$ this implies that $S_{i+1}$ is already equal to that component, and therefore has dimension $\dim S_{n-1}-1$.  
 
(b) is an immediate consequence of \cite[Theorem 7.4 (iv)]{Chai} on the maximal length of chains of comparable Newton polygons between $\nu_0$ and $\nu$ (that are not necessarily realizable).
\end{proof}
Note that a consequence of the proof is that already the $S_i$ we started with have dimension $\dim S -i$.

\begin{proposition}\label{propnewtondim}
Let $\Defo$ be the coordinate ring of the universal deformation bounded by $\mu$ of a local $G$-shtuka over $k$ with Newton polygon $\nu$ and let $\ol\Defo=(\Defo/\zeta\Defo)_\red$. Let $\ul\CG$ be the local $G$-shtuka over $S=\Spec\ol\Defo$ corresponding to the universal family over $\Spf\ol\Defo$ as in Proposition~\ref{remdejong}. Let $\mathcal{N}_{\nu}$ be the closed Newton stratum inside $S$ and let $\Test$ be any irreducible component of $\mathcal{N}_{\nu}$. Then 
$$\dim(\Test)\geq \langle 2\rho, \mu\rangle-\sum_i \lceil\langle\omega_i,\mu-\nu\rangle\rceil.$$
\end{proposition}
\begin{proof}
By Proposition \ref{PropReducedDefoSp},  $S=\Spec\ol\Defo$ is equidimensional of dimension $ \langle 2\rho, \mu\rangle$. We apply Corollary \ref{cor3.4} (b) to the reduced subscheme $S'$ of an irreducible component of $S$. By the corollary it is enough to show that the Newton polygon in the generic point of $S'$ is $\preceq\mu$. This follows from Proposition \ref{mazur}, because the same inequality holds for the Hodge polygon by Lemma~\ref{LemmaHPBounded}.
\end{proof}

\begin{remark}\label{remnewtondim}
By \cite{Kottwitz2006}, the estimate can be rewritten as
\begin{equation}\label{eq06}
\langle 2\rho, \mu\rangle-\sum_i \lceil\langle\omega_i,\mu-\nu\rangle\rceil=\langle \rho, \mu+\nu\rangle-\frac{1}{2}{\rm def}(b).
\end{equation}
Here $b\in LG(k)$ is any element with Newton polygon $\nu$ and $\kappa(b)=[\mu]$. Furthermore, ${\rm def}(b)=\rk(G)-\rk_{\mathbb{F}_q\dpl z\dpr}(J)$ where $J:=\QIsog(\ul\CG)$ is the group of quasi-isogenies of $\ul\CG=(K_k,b\s)$ and where $\rk_{\mathbb{F}_q\dpl z\dpr}J$ is the rank of a maximal $\mathbb{F}_q\dpl z\dpr$-split subtorus of $J$. The group $J$ is the set of $\BF_q\dpl z\dpr$-valued points of a reductive algebraic group over $\BF_q\dpl z\dpr$ which is an inner form of a Levi subgroup of $G$. This can be shown using the same argument as for the analogous statement for $p$-divisible groups, compare \cite{Kottwitz85}, \cite[Corollary 1.14]{RZ} or \cite{Kottwitz97}. Note that using 1.1 of \cite{Kottwitz2006}, one sees that (\ref{eq06}) also holds without the additional assumption that the derived group of $G$ is simply connected (which is a general assumption in Kottwitz's paper).
\end{remark}

%%%%%%%%%%%%%%%%%%%%%%%%%%%%%%%%%%%%%%%%%%%%%%%%%%%%%%%%%%%%%%%%%%%%%%
%
%    Basic Newton strata
%
%%%%%%%%%%%%%%%%%%%%%%%%%%%%%%%%%%%%%%%%%%%%%%%%%%%%%%%%%%%%%%%%%%%%%%
\section{Basic Newton strata}\label{SectBasic}

An element $b\in LG(k)$ is called \emph{basic} if its Newton point $\nu\in X_\ast(T)_\BQ$ is central in $G$ (compare \cite[\S 5.1]{Kottwitz85}). Note that this property only depends on the $\sigma$-conjugacy class of $b$.

\begin{proposition}\label{prop41}
Let $\ul\CG=(\CG,\varphi)$ be a local $G$-shtuka over a reduced noetherian complete local ring $R$ over $\mathbb{F}_q$ with algebraically closed residue field. Let its Newton polygon be constant on $\Spec R$ and basic. Then there exists a local $G$-Shtuka $\ul\CF$ over $\mathbb{F}_q$ and a quasi-isogeny $\ul\CG\rightarrow \ul\CF\times_{\Spec\mathbb{F}_q}\Spec R$.
\end{proposition}

\begin{proof}
By Remark~\ref{RemEtaleTrivial} the $K$-torsor $\CG$ over $R$ is isomorphic to the trivial torsor $K_R$, and choosing a trivialization $\varphi$ can be written as $b\s$ for some $b\in LG(R)$. Let $\mathfrak{m}$ be the maximal ideal of $R$ and let $k=R/\mathfrak{m}$. There is an $h>0$ and a lift of $h\nu\in X_*(T)_{\mathbb{Q}}$ to $X_*(T)$ which maps to $h[\mu_{\ul\CG}]$ in $\pi_1(G)$. Indeed, after multiplying $\nu$ by the common denominator $h_1$ we obtain an element of $X_*(T)$ whose class in $\pi_1(G)$ coincides with $h_1[\mu_{\ul\CG}]$ up to torsion. By multiplying with a second integer, we may assume that the two elements are equal. As $k$ is algebraically closed, \cite[\S 4.3]{Kottwitz85} shows that there is a $y\in LG(k)$ such that the reduction $\ol{\varphi}^h$ of $\varphi^h$ modulo $\Fm$ maps $y$ to $cy$ for the central element $c:=z^{h\nu}\in LG(k)$ (where $h\nu$ denotes the lifted element of $X_*(T)$). Since $k\subset R$ we may replace $\ul\CG\cong\bigl(K_R,b\s\bigr)$ by the quasi-isogenous local $G$-shtuka $\bigl(K_R,y^{-1}b\s(y)\cdot\s\bigr)$ and we may assume that $\ol\phi^h=c\,(\sigma^h)^\ast$.\\
 
\noindent {\it Claim.} There is a quasi-isogeny between $\ul\CG$ and a local $G$-shtuka $(\CG',\varphi')$ with $\CG'=K_R$ and $(\varphi')^h=c\,(\sigma^h)^\ast$.\\  

\noindent Identifying $\CG$ with the trivial torsor, the claim is equivalent to the existence of an element $x\in LG(R)$ with $\varphi^h(x)=cx$. Replacing $\varphi$ by $c^{-1}\varphi^{h}$ and $\sigma$ by $\sigma^h$ we may assume that $h=1$, and $c=1$. As $c$ is central and commutes with $\s$, the $\sigma$-conjugacy class of this renormalized $\varphi$ is $[1]$ in each point of $\Spec{R}$. Let $\tilde{x}\in LG(R)=G(R\dpl z\dpr)$ be a lift of $1\in LG(k)$. Then $\varphi(\tilde x)\equiv \tilde x\pmod{\mathfrak{m}}$. As $\varphi$ is $\sigma$-linear, $\varphi^n(\tilde{x})\equiv \varphi^{n-1}(\tilde x)\pmod{(\s)^{n-1}(\mathfrak{m})}$. Thus the sequence $\varphi^n(\tilde x)$ has a limit $x\in G\bigl(\widehat{R\dpl z\dpr}\bigr)$. The congruences further show that $\varphi(x)=x$. To prove the claim it remains to show that $x\in LG(R)$. As in the proof of Proposition \ref{remdejong} we have to show that the denominators occurring in $\varphi^n(\tilde{x})$ are bounded independently of $n$. As $R$ is a reduced noetherian ring, it is enough to check this in each generic point $\eta$ of $\Spec R$ separately. Let $\ol{k(\eta)}$ be an algebraic closure of $k(\eta)$ and let $\tilde {x}_{\eta}$ be the image of $\tilde x$ in $LG(\ol{k(\eta)})$. Let $\varphi_{\eta}$ be the induced map on $LG(\ol{k(\eta)})$. As $\ol{k(\eta)}$ is algebraically closed and as $\varphi_{\eta}$ is in the class $[1]\in B(G)$, there is a $y\in LG(\ol{k(\eta)})$ with $\varphi_{\eta}(y)=y$. Then $$\varphi_\eta^n(\tilde {x}_{\eta})=\varphi_\eta^n(yy^{-1}\tilde x_\eta)=\varphi_\eta^n(y)\cdot(\sigma^n)^\ast(y^{-1}\tilde{x}_{\eta})=y\cdot(\sigma^n)^\ast(y^{-1}\tilde{x}_{\eta}).$$ This shows that the $z$-powers occurring in $\varphi_{\eta}^n(\tilde{x}_{\eta})$ are bounded independently of $n$. Especially, $x\in LG(R)$. This proves the claim.

It remains to show that if $\ul\CG$ is a local $G$-shtuka with $\varphi^h=c\,(\sigma^h)^\ast$, then $\ul\CG$ is isogenous to a constant local $G$-shtuka $\ul\CF_R$ as in the assertion. Over the residue field $k$ this follows again from the classification of $\sigma$-conjugacy classes, see Remark \ref{remclass}. As $\CG$ is trivial we may thus assume that $\varphi=b\s$ for some $b$ whose reduction $\ol b$ modulo $\mathfrak{m}$ is defined over $\mathbb{F}_q$. Let $\tilde b$ be the image of $\ol b\in LG(\mathbb{F}_q)$ in $LG(R)$. It suffices to show that $b=\tilde b$. We use induction on $n$ to show that $b\equiv \tilde b\pmod{(\s)^n(\mathfrak{m})}$. For $n=0$ this follows from the definition of $\tilde b$. Assume that $b\equiv \tilde b\pmod{(\s)^n(\mathfrak{m})}$. Then
$$b\s(b\s)^{h-1}=c=\tilde{b}^h=\tilde b\s(\tilde{b}^{h-1})\equiv \tilde b\s(b\s)^{h-1}\pmod{(\s)^{n+1}(\mathfrak{m})}.$$ Thus $b\equiv \tilde{b}\pmod{(\s)^{n+1}(\mathfrak{m})}.$
\end{proof}

We are now able to give the proofs of Theorem~\ref{thmnewtonadlv}, Theorem \ref{thmequidim} and Corollary \ref{corunivnp}.

\begin{proof}[Proof of Theorem~\ref{thmnewtonadlv}.]

By Theorem \ref{ThmADLV=RZSp} we obtain a local $G$-shtuka $\ul\CG'$ over $(X_{\preceq\mu}(b))\kompl_g$ together with a quasi-isogeny to $\bigl(K_{(X_{\preceq\mu}(b))\kompl_g},b\s\bigr)$. Then $\ul\CG'$ is a deformation of its special fibre, which is isomorphic to $(K_k,g^{-1}b\s(g)\s)$. By definition it is isogenous to a constant local $G$-shtuka, and especially its Newton polygon is constant. Also, its Hodge polygon is $\preceq\mu$ in every point of  $(X_{\preceq\mu}(b))\kompl_g$. Since $\zeta=0$ on $(X_{\preceq\mu}(b))\kompl_g$ and $(X_{\preceq\mu}(b))\kompl_g$ is reduced (being the completion of a reduced local ring) $\ul\CG'$ defines a morphism $f: (X_{\preceq\mu}(b))\kompl_g\rightarrow \mathcal{N}_{\nu}$.

To construct a morphism in the opposite direction, we use Proposition \ref{prop41}. Let $\ul\CG$ be the universal object over $\mathcal{N}_{\nu}$. By the proposition there is a quasi-isogeny $\ppsi$ from $\ul\CG$ to a constant local $G$-shtuka. Changing this quasi-isogeny by a quasi-isogeny between constant local $G$-shtukas, we may assume that the constant local $G$-shtuka is $(K_k,b\s)$ and that the quasi-isogeny at the special fibre is given by $g$. Since $\zeta=0$ on $\CN_\nu$ and $\CN_\nu$ is reduced this yields a morphism $h:\mathcal{N}_{\nu} \rightarrow (X_{\preceq\mu}(b))\kompl_g$. The composition $f\circ h$ first constructs the additional quasi-isogeny $\ppsi$, and then forgets it. Thus it is the identity. The composition $h\circ f$ is the identity by rigidity of quasi-isogenies of local $G$-shtukas (Proposition~\ref{PropRigidity}).
\end{proof}

\begin{proof}[Proof of Theorem \ref{thmequidim}] Note that as $b$ is basic, $\langle \rho, \nu\rangle=0$. By \cite{GHKR},  and \cite{Viehmann06} we already know that $X_{\mu}(b)$ and $X_{\preceq \mu}(b)$ have dimension $\langle \rho, \mu\rangle-\frac{1}{2}{\rm def}(b)$. It remains to show that each irreducible component has at least this dimension. This is an immediate consequence of Theorem \ref{thmnewtonadlv} together with the lower bound on the dimension of each irreducible component of $\mathcal{N}_{\nu}$ in Proposition \ref{propnewtondim} and Remark \ref{remnewtondim}. 
\end{proof}

\begin{proof}[Proof of Corollary \ref{corunivnp}]
  From Theorem \ref{thmequidim} and Theorem \ref{thmnewtonadlv} we obtain that $\mathcal{N}_{\nu}$ is equidimensional of dimension  $\langle \rho, \mu\rangle-\frac{1}{2}{\rm def}(b)=\langle 2\rho, \mu\rangle-\sum_i \lceil\langle\omega_i,\mu-\nu\rangle\rceil$. Let $\nu_0$ be the Newton point at the generic point of $\Spec \ol\Defo$.  By Corollary \ref{cor3.4}(b) we obtain $$\sum_i \lceil\langle\omega_i,\mu-\nu\rangle\rceil\leq \sum_i \lceil\langle\omega_i,\nu_0-\nu\rangle\rceil.$$ The two sides of this inequality are the maximal lengths of chains of comparable Newton polygons between $\nu$ and $\mu$, respectively $\nu_0$ (see \cite[Theorem 7.4 (iv)]{Chai}). As $\nu_0\preceq\mu$, the right hand side is strictly smaller than the left hand side unless $\mu=\nu_0$.
\end{proof}

%%%%%%%%%%%%%%%%%%%%%%%%%%%%%%%%%%%%%%%%%%%%%%%%%%%%%%%%%%%%%%%%%%%%%%
%
%    Application to Newton strata in the Iwahori case
%
%%%%%%%%%%%%%%%%%%%%%%%%%%%%%%%%%%%%%%%%%%%%%%%%%%%%%%%%%%%%%%%%%%%%%%

\section{The Iwahori case}\label{sec9}

In this section we replace the maximal bounded open subgroup $K$ by an Iwahori subgroup $I$ of $G$ and explain how some of the preceding results can be generalized to this situation. Again we can relate the dimension of the basic Newton stratum of some universal $(\zeta=0)$-deformation to the dimension of an affine Deligne-Lusztig variety. However, in this case one does not have a closed expression for the dimension of the basic Newton stratum yet. 

As before let $G$ be a split connected reductive group over $\mathbb{F}_q$. Let $B$ be a Borel subgroup and let $T\subseteq B$ be a split maximal torus. Let $\pi:R\dbl z\dbr\rightarrow R$ be the projection. Let $I$ be the Iwahori subgroup scheme of $K$ over $\BF_q$ defined as
\[
I(R)\;:=\;\{\,g\in G\bigl(R\dbl z\dbr\bigr): \pi(g)\in B(R)\,\}
\]
on any $\BF_q$-algebra $R$. Let $\widetilde{W}\cong W\ltimes X_\ast(T)$ denote the extended affine Weyl group of $G$. Then $LG(k)=\coprod_{x\in\widetilde{W}}I(k)xI(k)$ by the Bruhat-Tits decomposition. 

We now define the analog of local $G$-shtukas, with $K$ replaced by $I$. We restrict our attention to base schemes $S$ with $\zeta=0$.

\begin{definition} \label{DefLocGIShtuka}
\begin{enumerate}
\item A \emph{local $G$-shtuka with $I$-structure} over an $\mathbb{F}_q$-scheme $S$ is a pair $\ul\CI=(\CI,\phi)$ consisting of an $I$-torsor $\CI$ on $S$ (for the \'etale topology) together with an isomorphism $\phi:\s\tilde\CI \isoto\tilde\CI$. Here $\tilde\CI$ denotes the $LG$-torsor associated to $\CI$.
\item We fix an element $x\in\wt{W}$. A local $G$-shtuka with $I$-structure $\ul\CI$ over $S$ is \emph{of affine Weyl type $x$} if \'etale-locally on $S$, the shtuka is isomorphic to a trivial $I$-torsor $I_S$ with $\phi=b\s$ for some $b\in I(S)xI(S)$. 
\end{enumerate}
\end{definition}

\begin{remark}\label{remIwa2}
\begin{enumerate}
\item Lemma \ref{lemdefi} shows that this rather ad hoc definition is analogous to the notion of a local $G$-shtuka with fixed Hodge polygon $\mu$.
\item Let $R$ be a complete local noetherian ring with algebraically closed residue field. In the same way as Proposition \ref{remdejong} one can show that there is an equivalence of categories between local $G$-shtukas with $I$-structure of affine Weyl type $x$ over $\Spec R$ and over $\Spf R$.
\end{enumerate}
\end{remark}
\begin{lemma}\label{lemdefi}
Let $(\CG,\phi)$ be a local $G$-shtuka which has constant Hodge polygon $\mu$ and is bounded by $\mu$ over an $\BFZ/(\zeta)$-scheme $S$. Then \'etale-locally on $S$ the $K$-torsor $\CG$ is trivial and $\phi$ is given by an element of $K(S)z^{\mu}K(S)$.
\end{lemma}
\begin{proof}
By Proposition~\ref{PropGzTorsor} there is an \'etale covering
$S'\to S$ such that $\CG_{S'}$ is trivial. Choosing a trivialization we
can write $\phi=b\s$ for some $b\in LG(S')$. Then $b^{-1}$
induces a morphism $S'\to \X$ which factors through $\Gr^{\preceq\mu}$ by
Proposition~\ref{PropProSchubertCell} since $\phi$ is bounded by $\mu$.
Since the Hodge polygon is equal to $\mu$ it factors through the open
subscheme $\Gr^\mu$ which is the homogeneous space $Kz^{(-\mu)_{\dom}} K/K$. This
proves our assertion.
\end{proof}
Like $\X$ also the \emph{affine flag variety} $\CF:=LG/I$ is an ind-scheme of ind-finite type. Analogously to Proposition \ref{PropSchubertCell} we have for any $x\in\wt W$:

\begin{remark}\label{Prop9.3}
\begin{enumerate}
\item 
The set $I(k)xI(k)/I(k)$ is the set of $k$-valued points of a locally closed reduced subscheme $\CF_x$ of $\CF$ which is of finite type over $k$.
\item 
$\CF_x$ is irreducible and of dimension $\ell(x)$, the length of $x\in \widetilde{W}$.
\end{enumerate}
\end{remark}

Let $\ul\BG=(\BG,\phi_{\BG})$ be a local $G$-shtuka with $I$-structure of affine Weyl type $x$ over an algebraically closed field $k\in\NilpF$. Choosing a trivialization of $\mathbb{G}$, we can write $\phi_{\BG}=b_0\s$ for some $b_0\in I(k)xI(k)$. Then $b_0^{-1}$ defines a point in $\CF_{x^{-1}}(k)$. Let $\DefoI$ be the complete local ring of $\CF_{x^{-1}}$ at this point. It is a complete noetherian local ring over $k$. Note that $\ell(x^{-1})=\ell(x)$, so $\CF_{x^{-1}}$ has dimension $\ell(x)$.

\begin{theorem}\label{ThmDefoSpI}
$\DefoI$ pro-represents the formal $(\zeta=0)$-deformation functor of $\ul\BG$
\begin{eqnarray*}
& & F:\es\bigl(\text{Artinian local $k\dbl\zeta\dbr/(\zeta)$-algebras with residue field
  $k$}\bigr)\longto \Sets\\[2mm]
A&\longmapsto& \Bigl\{\;\text{Isomorphism classes of pairs }(\ul\CG,\beta)\text{ where}\\
& & \quad \ul\CG\text{ is a local $G$-shtuka with $I$-structure of affine Weyl type $x$ over $\Spec A$\,,}\\
& & \quad \beta:\ul\BG\isoto\ul\CG\otimes_A k \text{ is an isomorphism of local $G$-shtukas with $I$-structure}\;\Bigr\}
\end{eqnarray*}
where $(\ul\CG,\beta)$ and $(\ul\CG',\beta')$ are isomorphic if there exists an isomorphism $\gamma:\ul\CG\to\ul\CG'$ with $\beta'=(\gamma\otimes_A k)\circ\beta$.
\end{theorem}
\begin{proof}
The proof of this theorem is the same as for Theorem \ref{ThmDefoSp}, always replacing $K$ by $I$.
\end{proof} 

Let $\ul\CG$ be the local $G$-shtuka with $I$-structure over $\Spec \DefoI$ which corresponds via Remark \ref{remIwa2} (b) to the chosen universal local $G$-shtuka with $I$-structure over $\Spf \DefoI$. It consists of the trivial $I$-torsor $I_\DefoI$ together with $\phi=b_\DefoI\s$ for some element $b_\DefoI\in LG(\DefoI)$. As in Section \ref{secnewton} we can consider the function assigning to each geometric point of $\Spec \DefoI$ the $\sigma$-conjugacy class of $b_\DefoI$ in that point. This induces a Newton stratification on $\Spec \DefoI $. We denote by $\CN^I_{\nu}$ the locally closed reduced subscheme of $\Spec \DefoI$ where the Newton polygon of $b_\DefoI$ is equal to $\nu$.

We fix an element $b\in LG(k)$ and some $x\in \widetilde{W}$. The affine Deligne-Lusztig variety $X^I_x(b)$ associated to $G$, $I$, $b$, and $x$ is defined as the locally closed reduced subscheme of the affine flag variety $\CF$ whose $k$-valued points for an algebraically closed field $k$ are given by $$X^I_x(b)(k)\;=\;\bigl\{\,g\in LG(k)/I(k):\; g^{-1}b\s(g)\in I(k)xI(k)\,\bigr\}.$$
\begin{theorem}\label{thmi}
Let $g\in LG(k)$ be a representative of an element of $X^I_x(b)$ and assume that the Newton polygon $\nu$ of $b$ is basic. Let $\mathcal{N}^I_{\nu}$ be the closed Newton stratum in the universal $(\zeta=0)$-deformation of the local $G$-shtuka with $I$-structure $\bigl(I_k,g^{-1}b\s(g)\s\bigr)$. Then $\mathcal{N}^I_{\nu}$ is isomorphic to the completion of $X^I_x(b)$ at $g$. 
\end{theorem}
\begin{proof}
The proof of this theorem is the same as for Theorem \ref{thmnewtonadlv}, always replacing $K$ by $I$. The main ingredient is in both cases Proposition \ref{prop41}, which considers elements of $LG(\DefoI)$ up to $\sigma$-conjugation. This result does not depend on the subgroups $K$ or $I$.
\end{proof} 
\begin{corollary} Let $b$ be basic.
Let $S$ be an irreducible component of $X^I_x(b)$ and let $g\in S$ be a $k$-valued point not contained in any other irreducible component of $X^I_x(b)$. Then $\dim S=\ell(x)-d(g,x)$ where $d(g,x)$ is the maximal length of a realizable chain of Newton polygons in the universal deformation of $(I_k,g^{-1}b\s(g)\s)$ at its closed Newton stratum (which is irreducible). 
\end{corollary}
\begin{proof}
This is an immediate consequence of Theorem \ref{thmi} and Corollary \ref{cor3.4}(a).
\end{proof}
\begin{remark}
For deformations of local $G$-shtukas (without Iwahori-structure) we saw that such maximal lengths of realizable chains in the universal deformation of $g$ bounded by some $\mu$ are equal to the maximal length of chains of Newton polygons between $\mu$ and the basic Newton polygon. These lengths are given by Chai's formula, compare Corollary \ref{cor3.4}(b). For deformations of local $G$-shtukas with $I$-structure, the situation is more difficult. Let $b\in I(k)xI(k) $ and let $\mu$ be the dominant element in the Weyl group orbit of the translation part of $x\in\widetilde{W}$. Let $\nu$ be the Newton polygon of $b$. Then in general not every Newton polygon $\nu'$ with $\nu\preceq\nu'\preceq \mu$ arises as the Newton polygon of some point in the universal deformation of $b$ of affine Weyl type $x$. Let $N(b,x)$ be the set of Newton polygons that occur in the universal deformation of $b$. Then the maximal chain lengths in $N(b,x)$ are in general shorter than those in the set of all Newton polygons between $\nu$ and $\mu$. Up to now there is no closed expression for $N(b,x)$. A second question is whether there always is a chain in $N(b,x)$ of maximal length that is realizable in the universal deformation. However this last assertion holds in all known examples (for example for $G=SL_3$ and all $b$, $x$ by \cite{Beazley}).
\end{remark}

We spend the rest of this section using Theorem \ref{thmi} to prove the basic case of a conjecture of Beazley \cite[Conjecture 1]{Beazley} comparing dimensions of affine Deligne-Lusztig varieties inside the affine flag variety with codimensions of Newton strata in $I(k) xI(k)$. Let $b\in LG(k)$ be basic and $x\in\widetilde{W}$ with $X_x^I(b)\neq\emptyset$. We consider the subset $IxI:=I(k)xI(k)\subset LG(k)$. Let $\nu$ be the Newton polygon of $b$ and let $(IxI)_{\nu}$ be the subset of $IxI$ of elements having Newton polygon $\nu$. Theorem \ref{thmuniformbound} for $\CB=IxI$ implies that there is a bounded open subgroup $I_0$ of $I(k)$ such that the Newton polygon of an element of $IxI$ only depends on its image in $IxI/I_0$. This last set is naturally the set of $k$-valued points of a scheme. The specialization property of Newton polygons then shows that $(IxI)_{\nu}/I_0$ is the set of $k$-valued points of a closed subscheme of $IxI/I_0$. We define $\codim\bigl((IxI)_{\nu}\subseteq IxI\bigr):= \dim IxI/I_0 -\dim (IxI)_{\nu}/I_0$. By our choice of $I_0$ it is independent of $I_0$, provided $I_0$ is small enough. A subset $S$ of $IxI$ that is right invariant under a (sufficiently small) group $I_0$ is called irreducible if $S/I_0$ is irreducible. Again this is independent of the choice of $I_0$. The notions of irreducible component or smoothness for such subsets are defined similarly.

\begin{proposition}\label{Prop9.7}
Let $b\in LG(k)$ be basic and $x\in\widetilde{W}$ such that $X_x^I(b)\neq\emptyset$. Then
$$\dim X_x^I(b)=\ell(x)-\codim\bigl((IxI)_{\nu}\subseteq IxI\bigr).$$
\end{proposition}
This proposition proves the basic case of \cite[Conjecture 1]{Beazley}.
\begin{proof}
By Theorem \ref{thmi} and Remark~\ref{Prop9.3} it is enough to show that $\codim((IxI)_{\nu}\subseteq IxI)$ is equal to $\dim \Spec \DefoI - \dim \mathcal{N}^I_{\nu}$ where $\mathcal{N}^I_{\nu}$ is the Newton stratum in the universal $(\zeta=0)$-deformation $\Spec \DefoI$ associated to some $g\in X_x^I(b)$ such that $X_x^I(b)$ is smooth and of maximal dimension in $g$. As in the proof of Proposition~\ref{propnewtondim} one sees that $\Spec \DefoI$ and the completion of $IxI/I_0$ in $b$ are equidimensional. The assertion is then an immediate consequence of Proposition \ref{propi} and Corollary \ref{cor3.4} (a).
\end{proof}

\begin{proposition}\label{propi}
Let $b\in I(k)xI(k)$ with Newton polygon $\nu$, let $\Test$ be an irreducible component of $(IxI)_{\nu}$. Assume that $\Test$ is the only such component containing $b$, and that $\Test$ is normal in $b$. Then the closed Newton stratum $\CN_{\nu}^I$ in the universal $(\zeta=0)$-deformation of the local $G$-shtuka with $I$-structure $(I_k,b\s)$ is irreducible. Let $I_0$ be a subgroup of $I$ such that the Newton polygon of an element of $IxI$ only depends on its $I_0$-coset. Then a chain of Newton polygons is realizable in the universal $(\zeta=0)$-deformation at its closed Newton stratum if and only if it is realizable in $IxI/I_0$ at $\Test/I_0$.
\end{proposition}
\begin{proof}
If $\Spf R_1$ and $\Spf R_2$ are the completions of $IxI/I_0$ and $(IxI)_{\nu}/I_0$ in $b$, let $Y=\Spec R_1$ and $Y_{\nu}=\Spec R_2$. We choose a representative $f\in I(Y)xI(Y)$ of the universal element of $IxI/I_0$ over $Y$ with $f\otimes_{R_1} k=b$. We consider the local $G$-shtuka with $I$-structure $\ul\CI=(I_Y,f\s)$ over $Y$. Since $\ul\CI$ is a deformation of its special fiber $(I_k,b\s)$ this induces a morphism $\pi:Y\rightarrow \Spec \DefoI$ mapping $Y_{\nu}$ to $\mathcal{N}^I_{\nu}$ and also each other Newton stratum of $Y$ to the corresponding Newton stratum in $\Spec \DefoI$. Let $f'$ be such that $(I_\DefoI,f'\s)$ is the universal object over $\Spec \DefoI$ and that $f'\otimes_{\DefoI}k =b$. Then $f'\in I(\DefoI)xI(\DefoI)$. This provides a section of $\pi$. Especially $\pi$ is surjective. Our assumptions on $\Test$ imply that the closed Newton stratum in $Y$ is irreducible by Zariski's Main Theorem~\cite[Theorem VIII.32]{ZariskiSamuel2}. Therefore the closed Newton stratum $\mathcal{N}^I_{\nu}$ in $\Spec \DefoI$ is also irreducible. Thus the chains realizable in $Y$ at $Y_{\nu}$ are also realizable in $\Spec \DefoI$ at $\CN^I_{\nu}$.

For the other direction let $\nu=\nu_n\prec\dotsc\prec\nu_0$ be a realizable chain of Newton polygons in $\Spec \DefoI$ at $S_n=\mathcal{N}^I_{\nu}$ and let $S_i$ be corresponding irreducible subschemes of $\Spec \DefoI$. Let $\Test_i\subseteq Y$ be the inverse image of $S_i$ under $\pi$. Then $\Test_n=Y_{\nu}$. We want to show that $\nu_n\prec\dotsc\prec\nu_0$ is also realizable in $Y$. To do that we use decreasing induction on $i$ to show that there is an irreducible component $\Test_i'$ of $\Test_i$ with $\Test_{i+1}'\subseteq \overline{\Test_i'}$. As $\Test_{i+1}'$ is (by induction) irreducible, it is enough to show that $\Test_{i+1}'\subseteq \overline{\Test_i}$, or that $\Test_{i+1}\subseteq \overline{\Test_i}$. Let $y$ be a geometric point of $\Test_{i+1}$ with values in some field $\tilde{k}$. As $S_{i+1}\subseteq \overline{S_i}$ there is a $\tilde{k}\dbl t\dbr$-valued point of $S_{i}\cup S_{i+1}$ with special point $\pi(y)$ and general point in $S_{i}$. Let $f'_{\tilde{k}\dbl t\dbr}$ be the reduction of $f'$ at this point and let $f'_{\tilde{k}}$ be the reduction in its special point. As $\pi(f'_{\tilde k})=\pi(y)$ there is an $i\in I(\tilde{k})$ with $i^{-1}f'_{\tilde k}\s(i)=y$. Then $i^{-1}f'_{\tilde{k}\dbl t\dbr}\s(i)$ is a $\tilde k\dbl t\dbr$-valued point with special point $y$ and general point in $\Test_i$. This finishes the proof that chains realizable in $\Spec \DefoI$ are also realizable in $Y$.
\end{proof}

\section*{Appendix: Admissibility}
\addcontentsline{toc}{section}{Appendix: Admissibility}
\addtocounter{section}{1}
\setcounter{theorem}{0}

In this section we provide the results needed to generalize Vasiu's proof of purity for $F$-crystals to local $G$-shtukas. Theorem \ref{thmuniformbound} also implies admissibility of the Newton stratification of $I$-double cosets needed in Section \ref{sec9}.

For $n\geq 0$ let $I_n=\{g\in K(k): g\equiv 1 \pmod{z^n}, g\in B \pmod{z^{n+1}}\}$. Here $B$ is the chosen Borel subgroup. Recall that a subset of $LG(k)$ is bounded if and only if it is contained in a finite union of $K$-double cosets $Kz^{\mu}K$. 

\begin{theorem}\label{thmuniformbound}
Let $\CB$ be a bounded subset of $LG(k)$. Then there is a $c\in \mathbb{N}$ such that for each $d\in\mathbb {N}$, each $g\in \CB$ and each $h\in I_{d+c}$ there is a $k\in I_{d}$ with $gh=k^{-1}g\s(k)$. 
\end{theorem}

\begin{proof}
Note that the set of $\sigma$-conjugacy classes of elements of $\CB$ is finite. We write $\CB$ as a disjoint union of its intersections with the different $\sigma$-conjugacy classes and consider each subset separately. Thus from now on we assume that all elements of $\CB$ are $\sigma$-conjugate under $LG(k)$. Let $\nu$ be their Newton polygon. Then $\CB$ lies in a finite union of double cosets $Kz^{\mu'}K$ with $\nu\preceq\mu'$ and all occurring $\mu'$ have the same image in $\pi_1(G)$. Hence $\CB\subset\coprod_{\mu'\preceq\mu} Kz^{\mu'}K$ for some $\mu$.

We first show that there is a bounded subset $\CC$ of $LG(k)$ and an element $b\in LG(k)$ such that each element of $\CB$ is $\sigma$-conjugate via an element of $\CC$ to $b$. There is a standard parabolic subgroup $P=MN$ of $G$ with Levi component $M$ and a $b\in M$ in the given $\sigma$-conjugacy class with $M$-dominant Newton polygon $\nu$ such that $b$ is superbasic in $M$. Indeed, $M$ is simply the smallest standard Levi subgroup of the centralizer of $\nu$ in $G$ such that the given $\sigma$-conjugacy class contains an element of $M$. Each $g\in \CB$ is of the form $g=f^{-1}b\s(f)$ for some $f\in LG(k)$. We have to show that we may take all $f$ in some bounded subset of $LG(k)$. Using the Iwasawa decomposition we write $f=mnk$ with $m\in M(k\dpl z\dpr)$, $n\in N(k\dpl z\dpr)$ and $k\in K$. Thus all elements of $\CB$ are $\sigma$-conjugate via $K$ to elements of the form $m^{-1}b\s(m)\bigl((m^{-1}b\s(m))^{-1}n^{-1}(m^{-1}b\s(m))\s n\bigr)\in MN$. These elements are still all in $\coprod_{\mu'\preceq\mu} Kz^{\mu'}K$ and in the same $\sigma$-conjugacy class. We claim that the Levi parts $l:=m^{-1}b\s(m)$ of these elements are then again in a finite union of $M(k\dbl z\dbr)$-double cosets, namely those contained in $\coprod_{\mu'\preceq\mu} Kz^{\mu'}K$. To see this, consider for a $\lambda\in X_*(T)$ which is central in $M$ and with $\langle\alpha,\lambda\rangle >0$ for each root of $T$ in $N$ the element $g_x\in LG(k[x])$ given by $g_0=l$ and $g_x=\lambda(x)g\lambda(x)^{-1}$ for $x\neq 0$. An easy calculation shows that this defines an element of $LG(k[x])$. As $\lambda(x)\in K$, all $g_x$ for $x\neq 0$ are in $\coprod_{\mu'\preceq\mu} Kz^{\mu'}K$. As this union is closed, it also contains $g_0=l$. Reformulating the boundedness for the elements $l$ in terms of $m$, we obtain $m\in X^M_{\mu'}(b)$, where $\mu'$ is one of these finitely many dominant coweights of $M$. By \cite[Proposition 1]{conncomp}, $m$ is of the form $m=jm_0$ for some $j\in J_M=\{x\in M(k\dpl z\dpr): x^{-1}b\s(x)=b\}$ and $m_0$ in some fixed connected component of $X^M_{\preceq\mu'}(b)$. As the element $g$ we started with only determines $m$ up to left multiplication by $J_M$, we may assume that $m=m_0$. As $b$ is superbasic in $M$, each connected component of $X_{\preceq\mu'}^M(b)$ is a projective scheme of finite type (see for example \cite{Viehmann06}). Especially, all elements $m$ (for one $\mu'$, but then also for all of the finitely many occurring $\mu'$) are in some bounded subset of $LG(k)$. By $\sigma$-conjugating $g$ with $m^{-1}=m_0^{-1}$ we obtain an element $g'$ of the form $\tilde{n}^{-1}b\s(\tilde n)$ for $\tilde n=mnm^{-1}\in N(k\dpl z\dpr)$. Thus $g'=bn'$ with $n'=b^{-1}\tilde{n}^{-1}b\s(\tilde n)\in N$. So far we only conjugated $g$ by elements in a bounded subset of $LG(k)$, thus the $n'$ still are contained in a subset of $N(k\dpl z\dpr)$ which can be bounded by a bound only depending on the original $\CB$. Let $N_0$ be a given bounded open subgroup of $N(k\dpl z\dpr)$. By $\sigma$-conjugating by a sufficiently dominant element of $T(\mathbb{F}_q\dpl z\dpr)$ which is central in $M$ we may map all $n'$ in this bounded subset to $N_0$. Note that again the element we conjugate by is uniformly bounded (by a bound depending on $\CB$ and the chosen $N_0$). Using Lemma \ref{lemapp} for $b$ and $N_0=I_{c_b}\cap N$ (for $d=0$), we obtain that each element of $\CB$ is $\sigma$-conjugate by an element of a bounded set $\CC$ to $b$. 

Let $c_1$ be such that $f I_{d+c_1}f^{-1}\in I_d$ for all $f\in \CC$ and all $d$. Let further $c_2$ be such that $\s(f)^{-1} I_{d+c_2}\s(f)\in I_d$ for all $f\in \CC$ and all $d$. Let $c=c_b+c_1+c_2$. Let $g\in \CB$ and $h\in I_{d+c}$ for some $d>0$. Let $f\in \CC$ with $f^{-1}g\s(f)=b$. Then 
\begin{align*}
gh&=fb\bigl(\s(f^{-1})h\s(f)\bigr)\s(f^{-1})\\ \intertext{As the expression in the bracket is in $I_{d+c_b+c_1}$ by our choice of $c_2$, Lemma \ref{lemapp} shows that there is an $k'\in I_{d+c_1}$ such that this equals}
&=fk'^{-1}b\s(k')\s(f^{-1})\\
&=(fk'f^{-1})^{-1}g\s(fk'f^{-1}).
\end{align*}
As $fk'f^{-1}\in I_{d}$ by the choice of $c_1$, this proves the theorem.
\end{proof}

\begin{lemma}\label{lemapp}
Let $b\in LG(k)$. Then there is a $c_b\in \mathbb{N}$ such that for each $d\in\mathbb {N}$ and each $g\in I_{d+c_b}$ there is a $k\in I_{d}$ with $bg=k^{-1}b\s(k)$.
\end{lemma}
\begin{proof}
By $\sigma$-conjugating $b$ we may assume that $b$ is equal to the standard representative of its $\sigma$-conjugacy class as defined in \cite[7.2]{GHKR2}. More precisely, this implies that $b$ has the following properties. Let $M$ be the centralizer of the $G$-dominant Newton polygon of $b$, it is the Levi component of a standard parabolic subgroup $P$ of $G$. Then the standard representative is an element $b$ of the extended affine Weyl group $\widetilde{W}_M$ of $M$ with $bI_Mb^{-1}=I_M$ where $I_M=I\cap M$. Note that this $M$ is not the same as in the proof of Theorem \ref{thmuniformbound}. There we considered a Levi such that no $\sigma$-conjugate of $b$ is contained in an even smaller one. Here, the Levi subgroup $M$ is the largest standard Levi subgroup such that $b$ is basic in $M$. As $b$ is in $\widetilde{W}_M$, there is an $r>0$ such that $b^r=z^{r\nu}$.

It is enough to prove the lemma for $d>0$. Let $c_b$ be so large that $b^i I_{d+c_b}b^{-i}\subseteq I_d$ for all $i\in [-r+1,\dotsc,r-1 ]$. (Using that $I_d$ is generated by the corresponding affine root subgroups, one sees that if this condition holds for one $d$ then it holds for all.) For the groups $I_n$ we have an Iwahori decomposition $I_n=I_{N,n}I_{M,n}I_{\overline N,n}$ where the factors are the intersections of $I_n$ with $N$, $M$ and $\overline N$ and where $N$ and $\overline N$ are the unipotent radicals of $P$ and its opposite parabolic. Let $g\in I_{d+c_b}$ and let $g=g_Ng_Mg_{\ol N}$ be its decomposition. We begin by showing that there is an $f_{\overline N}\in I_{\overline N,d}$ with $f_{\overline N}^{-1}bg\s(f_{\overline N})\in bI_{M,d+c_b}I_{N,d+c_b}$. Recall that $b^r=z^{r\nu}$ for a $\nu$ with $\langle \alpha,\nu\rangle<0$ for every root of $T$ in $\ol N$. Thus $g\mapsto b^{-1}gb$ is an elementwise topologically nilpotent map on $\ol N$. The definition of $c_b$ together with $b^r=z^{r\nu}$ also implies that $b^{-i}gb^i\in I_d$ for all $i\geq 0$ and $g\in I_{\ol N,d+c_b}$. Let $f_1=(\s)^{-1}(g_{\ol N})$. Then 
\begin{eqnarray*}
f_1^{-1}bg_Ng_Mg_{\ol N}\s(f_1)&=&f_1^{-1}bg_Ng_M\\
&=&b(b^{-1}f_1^{-1}b)g_Ng_M.
\end{eqnarray*} 
Note that $[I_a,I_{a'}]\subset I_{a+a'}$ for all $a,a'$. As $b^{-1}f_1^{-1}b\in I_d$, the above product lies in $bg_Ng_M(b^{-1}f_1^{-1}b)I_{2d+c_b}$. Repeating this construction for $g'\in g_Ng_M(b^{-1}f_1^{-1}b)I_{2d+c_b}$ and its Iwahori decomposition $g'=g'_Ng'_Mg'_{\ol N}$ yields an $f_2=(\s)^{-1}(g'_{\ol N})$. We have $g'_{\ol N}\in (b^{-1}f_1^{-1}b)I_{\ol N,2d+c_b}$, and by the choice of $c_b$, this is in $I_{d}$. We iterate this process and let $f_{\ol N}=f_1\circ f_2\circ\dotsm$. As conjugation by $b$ is elementwise topologically nilpotent on $\ol N$, the product converges. We obtain that $bg$ is $\sigma$-conjugate via an element of $I_{\ol N,d}$ to an element of the form $b\tilde{g}$ with $\tilde{g}=\tilde{g}_N\tilde{g}_M\in I_{N,d+c_b}I_{M,d+c_b}$. A similar argument (using $f_1=b\tilde{g}_Nb^{-1}$ and that $g\mapsto bg b^{-1}$ is elementwise topologically nilpotent on $N$) shows that we may also achieve that $\tilde{g}_N=1$. For $g\in I_{M,d+c_b}$ we use that $b^{-1}I_{M,a}b=I_{M,a}$ for all $a$. In this situation the argument in \cite[6.3]{GHKR2} also shows that $bg$ is $\sigma$-conjugate via $I_{M,d}$ (or even $I_{M,d+c_b}$) to $b$. 
\end{proof}

\begin{corollary}\label{corhombound}
Let $(M,b\s)$ and $(M',b'\s)$ be two effective local shtukas over $k$ of rank $r$ bounded by $\mu$ and $\mu'$. Then there is a $c\in \mathbb{N}$ only depending on $\mu$ and $\mu'$ with the following property. Let $g:M\rightarrow M'$ be a linear map with $(b'\s (g)- gb)(x)\in {z^{c+d}}M'$ for some $d$ and every $x\in \s M$. Then there exists a homomorphism of local shtukas $f:(M,b\s)\rightarrow (M',b'\s)$ with $(f-g)(x)\in z^dM'$ for every $x\in M$.
\end{corollary}
Recall that a local shtuka over $k$ is effective if and only if its Hodge polygon has only non-negative slopes.
\begin{proof}
Let $\tilde{M}=M\oplus M'$ and $\tilde b:\s\tilde M\rightarrow \tilde M$ with $\tilde b=(b,b')$. Consider $\tilde{g}:\tilde M\rightarrow \tilde M$ with $\tilde g|_{M'}=\id_{M'}$ and $\tilde{g}(m)=m+g(m)$ for $m\in M$. In matrix form we obtain
$$\tilde g^{-1}\tilde b\s(\tilde g)=\left(\begin{array}{cc}b&0\\-gb+b'\s(g)&b'\end{array}\right),$$ so the entries in the lower left block are congruent to zero modulo $z^{c+d}$. We use Theorem \ref{thmuniformbound} for $G=GL(\tilde M)$ and $\CB=K\tilde{b}K$. We obtain that there is a $c$ (only depending on $\CB$, i.~e. on the Hodge polygon of $\tilde b$ or equivalently on those of $b$ and $b'$), such that such a matrix is $\sigma$-conjugate via a $k\in I_d$ to $\tilde b$. An easy calculation shows that if we write $k=\left(\begin{array}{cc}k_1&k_2\\k_3&k_4\end{array}\right)$, then $f=gk_1+k_3$ has all claimed properties.
\end{proof}

The last corollary was also proved by Fessler~\cite{Fessler} using a different method.

%%%%%%%%%%%%%%%%%%%%%%%%%%%%%%%%%%%%%%%%%%%%%%%%%%%%%%%%%%%%%%%%%%%%%%
%
%    Bibliography
%
%%%%%%%%%%%%%%%%%%%%%%%%%%%%%%%%%%%%%%%%%%%%%%%%%%%%%%%%%%%%%%%%%%%%%%

\vfill

\begin{minipage}[t]{0.5\linewidth}
\noindent
Urs Hartl\\
Universit\"at M\"unster\\
Mathematisches Institut \\
Einsteinstr.~1\\
D -- 48149 M\"unster
\\ Germany
\\[1mm]
{\small www.math.uni-muenster.de/u/urs.hartl/}
\end{minipage}
\begin{minipage}[t]{0.45\linewidth}
\noindent
Eva Viehmann\\
Universit\"at Bonn\\
Mathematisches Institut \\
Endenicher Allee 60\\
D -- 53115 Bonn
\\ Germany
\\[1mm]
\end{minipage}

\end{document}